%% file: arith-theta.tex
\documentclass[11pt, final]{amsart}
\input{preamble}

\hypersetup
{
    pdfauthor={Ehlen, Stephan and Sankaran, Siddarth},
    pdftitle={On two arithmetic theta lifts},
    pdfsubject={MSC2010: 11F12, 11F27, 11G18, 11G15},
    pdfkeywords={Kudla program, Green function, arithmetic generating series, arithmetic geometry},
    final=true
}

\numberwithin{equation}{section}

\begin{document}
\title{On two arithmetic theta lifts}
 
\author{Stephan Ehlen}
\email{stephan.ehlen@math.uni-koeln.de}
\address{Mathematisches Institut, University of Cologne, Weyertal 86-90, 50931 Cologne, Germany}
\author{Siddarth Sankaran}
\email{siddarth.sankaran@umanitoba.ca}
\address{Department of Mathematics, University of Manitoba, 420 Machray Hall, Winnipeg, Canada}

\begin{abstract}
Our aim is to clarify the relationship between Kudla's and Bruinier's Green functions
attached to special cycles on Shimura varieties of orthogonal and unitary type, which play a key role in the arithmetic geometry of these cycles in the context of Kudla's program.
In particular, we show that the generating series obtained by taking the differences of the two families of Green functions is  a non-holomorphic modular  form and has trivial (cuspidal) holomorphic projection.

Along the way, we construct a section of the Maa\ss{} lowering operator for moderate growth  forms valued in the Weil representation using a regularized theta lift, 
which may be of independent interest, as it in particular has applications to mock modular forms.

We also consider arithmetic-geometric applications to integral models of $U(n,1)$ Shimura varieties. 
Each family of Green functions gives rise to a formal arithmetic theta function, valued in an arithmetic Chow group, that is conjectured to be modular; our main result is the modularity of the difference of the two arithmetic theta functions.

Finally, we relate the arithmetic heights of the special cycles to special derivatives of Eisenstein series,
as predicted by Kudla's conjecture, and describe a refinement of a theorem of Bruinier-Howard-Yang on arithmetic intersections against CM points.
\end{abstract}
\maketitle

\setcounter{tocdepth}{2}
\tableofcontents

\section{Introduction}

The aim of this paper is to clarify the relationship between two families of Green functions, attached to special divisors on Shimura varieties of type $O(n,2)$ and $U(n,1)$, that appear frequently in the literature in the context of Kudla's programme. The first family, introduced by Kudla, is closely related to the theta series considered by Kudla and Millson \cite{kudlamillsonintersection}, while the second family was constructed by Bruinier \cite{brhabil} following ideas of Borcherds. 

The first step, which we expect to be of independent interest, is the study of a family of \emph{truncated Poincar\'e series}; we show that these series can be used to construct a  section of the Maa\ss \ lowering operator on forms with certain growth conditions, and give a concrete characterization of the image of this section. 

We then show that Kudla's Green functions can be obtained by integrating (in a regularized sense, as in Borcherds \cite{boautgra})  truncated Poincar\'e series against  Siegel theta functions. As a corollary, the generating series obtained by taking the differences of the two families of Green functions is a non-holomorphic modular form with trivial cuspidal holomorphic projection; this result can be viewed as a lifting of (a special case of)  a theorem of Bruinier-Funke \cite{brfugeom} to the level of Green functions.

Finally, we explore consequences of this discussion to arithmetic divisors on ${U(n-1,1)}$ Shimura varieties: each of the two families of Green functions gives rise to a formal arithmetic theta function, and we prove that their difference is modular. We also prove a version of Kudla's conjecture relating arithmetic heights to the derivative of an Eisenstein series; here, the Siegel-Weil formula and the characterization of our section of the lowering operator combine to give an heuristic explanation for the connection of Kudla's Green functions to the derivative of the Eisenstein series in this context. As a final application, we discuss a refinement of a formula due to Bruinier-Howard-Yang \cite{bruinier-howard-yang-unitary} on intersection numbers with small CM cycles.

We now describe these results in more detail. Suppose $V$ is a quadratic space over $\Q$ and, for the purposes of the introduction, assume that $\dim V$ is even. 
Let $L \subset V$ be an even lattice with quadratic form $Q$ and dual lattice $L'$. For each coset $\mu \in L' / L$, define an element $\varphi_{\mu} \in S(V(\adeles_f))$ in the space of Schwartz-Bruhat functions on $V(\adeles_f):= V \otimes_{\Q} \bbA_f$ by setting
\[ 
  \varphi_{\mu} \ := \ \text{characteristic function of } \mu  +  L \otimes_{\Z} \widehat \Z. 
\]
The finite-dimensional space
\[ 
  S(L) \ := \ \mathrm{span}_{\C} \left\{ \varphi_{\mu} , \, \mu \in L'/L  \right\}  \ \subset \ S(V(\bbA_f))
\]
admits an action of $\SL_2(\Z)$, via the Weil representation, that we denote by $\rho_L$.

Fix any $k \in \Z$. The point of departure is the introduction of the $m$'th \emph{truncated Poincar\'e} series $P_{m, w, \mu}$ of weight $k$
which is an $S(L)$-valued (discontinuous) function that transforms as a modular form of weight $k$, and depends on additional parameters  $w \in \R_{>0}$ and $\mu \in S(L)$; 
 it is defined by the formula
	\[	
		P_{m, w, \mu}(\tau) \ := \ \frac12 \sum_{\gamma \in \Gamma_{\infty} \backslash \SL_2(\Z) }  \left( \sigma_w( \tau) \, q^{-m} \, \varphi_{\mu} \right)\mid_{k}[\gamma],
	\]
where $q = e^{2 \pi i \tau}$ and $\sigma_w( \tau)$ is the cutoff function
	\[ \sigma_w( \tau) \ = \ \begin{cases}
										1 , 	& \text{ if } \Im(\tau) \geq w \\
										0, 		& \text{ otherwise. }
									\end{cases}
	\]
This definition can be extended, by linearity, to produce a  function
\[ P_{m,w} \colon \uhp \ \to \ S(L)^{\vee} \otimes_{\C} S(L) \]
for every $m \in \Q$ and $w \in \R_{>0}$, cf.\ \Cref{rmk:intrinsicPoincare}. Note that there is a positive integer $N$, the level of $L$, such that $P_{m,w} = 0$ whenever $m \notin N^{-1} \Z$. 

Similarly, in \Cref{sec:relat-non-holom},  we construct a family of harmonic weak Maa\ss \ forms
\[ F_{m} \colon \uhp \to S(L)^{\vee} \otimes_{\C} S(L) \]
of weight $k$, indexed by $m \in \Q$. When $k < 0$ and $m> 0$, this form is an  $S(L)^{\vee}\otimes_{\C} S(L)$-valued version of the Hejhal-Poincar\'e series
as considered in \cite{brhabil}.

In \Cref{sec:invert-xi}, we consider a space $\ALmod{\kappa}(\rho_L^{\vee})$ of  $S(L)^{\vee}$-valued functions that transform as modular forms of weight $\kappa$, that are $C^{\infty}$,  have at worst moderate (polynomial) growth towards $\infty$, and whose constant terms satisfy a certain technical condition, see \Cref{def:modgrowthforms}. 
 
\begin{theorem} \label{thm:introLsec}
Suppose $f\in \ALmod{\kappa}(\rho_L^{\vee})$ and, for simplicity, that $\kappa>0$ (in the main text, we treat arbitrary weights $\kappa \in \frac12 \Z$, see Theorem \ref{thm:MaassSection}). 

For $\tau = u + iv \in \uhp$ and $q = e^{2 \pi i \tau}$,  define the generating series
\[ F(\tau) \ :=  \  \sum_{m} \langle P_{m, v}  - F_m, \, f  \rangle^{\reg} \ q^m \]
whose terms are  integrals of the form
\[ \langle G,  \, f  \rangle^{\reg} \ := \ \int_{\SL_2(\Z) \bs \uhp}^{\reg} \, G(\tau') \, f(\tau') \ d\mu(\tau') \]
that are regularized as in \cite{boautgra}, see also \Cref{def:regpairing} below, and involve Poincar\'e series of weight $k = - \kappa$. 

Then $F(\tau)$ converges to a smooth function on $\uhp$  that can be characterized uniquely by the following properties:
\begin{enumerate}
\item $F$ is smooth and transforms as a modular form of weight $\kappa+2$;
\item $F$ has at worst exponential growth at $\infty$;
\item $\Low(F) = -f$, where $\Low$ is the Maa\ss\, lowering operator; and
\item $F$ has trivial {principal part} and trivial 
      {cuspidal holomorphic projection}\footnote{By the (cuspidal) holomorphic projection of $F$, we mean the cusp form representing the linear functional $g \mapsto \langle g,\, v^{p/2+1} \bar F\rangle^\reg$ on the space of cusp forms $S_{p/2+1}(\rho_L)$. See also \Cref{sec:holom-proj}. For the definition of the principal part, see \Cref{prop:uniquepreimage}.}.
\end{enumerate}
\end{theorem}
When  $\kappa  \leq 0$  an additional normalization is required, on account of the presence of holomorphic modular forms of weight $-\kappa$. 

As a special case, suppose $f_0 \in M_k(\rho_L)$ is a holomorphic modular form, so that
\[ f := - v^k \overline{f_0}  \ \in \ \ALmod{-k}(\rho_L^{\vee}). \]
Then the function $F\in \ALexp{2-k}(\rho_L^{\vee})$ determined by applying \Cref{thm:introLsec} to $f$ satisfies $\xi(F) = f_0$, where $\xi$ is the Bruinier-Funke operator \cite{brfugeom}, see also \Cref{sec:modforms}. It follows immediately that $F$ is a harmonic weak Maa\ss\ form in the terminology of \emph{loc.\ cit.} 

In this setting, Theorem \ref{thm:introLsec} 
implies that the generating series 
\[
	\sum_{m \gg - \infty} \langle F_m, f \rangle \ q^m
\]
is a \emph{mock modular form}, see \cite{Zwegers:mocktheta, DIT-annals, DMZ-black-holes}, whose shadow is $f_0$.

Turning to the geometric applications of this theorem, we now assume that the signature of $V$ is $(p,2)$. Let $\mathbb D^o(V)$ denote the locally symmetric space attached to $O(V)$, which we realize as the space of oriented negative-definite planes in $V\otimes_{\Q} \R$, and is a complex manifold of dimension $p$. If $x \in V$ has positive norm, then
\[ Z(x) \ := \ \{ z \in \domain^o(V) \ | \ x \perp z \} \ \subset \ \domain^o(V) \]
is a complex-codimension-one submanifold. For any $m \in \Q$, define an $S(L)^{\vee}$-valued cycle  $Z(m)$ by the formula
\[ Z(m)(\varphi) \ = \ \sum_{ \substack{x \in V \\ Q(x) = m}}  \, \varphi(x) \, Z(x), \qquad \qquad \varphi \in S(L). \]
This sum is locally finite, in the sense that for a given compact subset $K \subset \domain^o(V)$, there will only be finitely many $x$'s appearing in the sum with $\varphi(x) \neq0 $ and $K \cap Z(x) \neq \emptyset$. Moreover, the condition that $V$ has signature $(p,2)$ implies that $Z(m) = 0$ whenever $m \leq 0$.

In this context, we say that a current 
\[ \lie g \in D^{(0,0)}(\domain^o(V))\otimes_{\C} S(L)^{\vee} \]
 is a $\log$-singular \emph{Green function} for $Z(m)$ if for every $\varphi \in S(L)$, it satisfies (i) $\ddc \lie{g} (\varphi)$ is dual to $Z(m)(\varphi)$ under Poincar\'e duality for each $\varphi\in S(L)$, and (ii) $\lie{g}(\varphi)$ is smooth on the complement $\domain^o(V) - \abs{Z(m,\varphi)}$ with logarithmic singularities along $Z(m,\varphi)$.

In \cite{kudla-annals-central-der}, Kudla  constructs a family of  {Green functions} $\KGrO(m,v)$ for the cycle $Z(m)$, depending on a real parameter $v \in \R_{>0}$, and with the property that $\ddc \KGrO(m,v)$ is the $m$'th Fourier coefficient of the Kudla-Millson theta function \cite{kudlamillsonintersection}.

Our next result relates $\KGrO(m,v)$  to the \emph{Siegel theta function}
\[ \Theta_L( \tau, z) \colon \mathbb H \times \mathbb D^o(V) \ \to \ S(L)^{\vee} \]
whose definition is recalled in \Cref{sec:arch}; for a fixed $z\in \domain^o(V)$, the Siegel theta function is a non-holomorphic 
modular form $\Theta(\cdot, z) \in \ALmod{p/2-1}(\rho_L^{\vee})$ of moderate growth.
\begin{theorem} \label{thm:introXiRegLift}
For any $m \neq 0$ and  $z \notin  \abs{Z(m)}$,
	\[
         \KGrO(m,v)(z) \ = \ \langle  P_{m,v},  \ \Theta_L(\cdot, z) \rangle^{\reg} , 
        \]
	and
	\[ 
          \KGrO(0,v)(z) \ = \ \langle  P_{0,v},  \ \Theta_L(\cdot, z) \rangle^{\reg} \ + \ \log(v) \, \varphi_0^{\vee} 
        \]
	where $\varphi_0^{\vee} \in S(L)^{\vee}$ is the linear functional $\varphi \mapsto \varphi(0)$.
	
	Moreover, the right hand sides of these formulas are defined for all $z \in \domain^o(V)$,
        and hence give a discontinuous extension of $\KGrO(m,v)$. 
\end{theorem}

Bruinier \cite{brhabil} showed that for $m>0$,
\[ 
  \BGrO{m}(z) \ := \ \langle F_m, \, \Theta_L(\cdot,z) \rangle^{\reg} 
\]
is also a Green function for $Z(m)$. 
Combining these facts with \Cref{thm:introLsec}, we obtain
\begin{theorem} \label{thm:introArchGenSer}  
	For each fixed $z \in \domain^o(V)$, the generating series
\[ - \log v \, \varphi_0^{\vee} \ + \  \sum_{m} \, \left( \KGrO(m,v) - \BGrO{m} \right)(z) \ q^m \]
	is a non-holomorphic modular form of weight $p/2+1$ with trivial cuspidal holomorphic projection.
\end{theorem}

The main result of \cite{brfugeom}, applied to the case at hand, implies the modularity of the generating series
	\[ 
		  \sum_{m} \, \ddc \left( \KGrO(m,v) - \BGrO{m} \right) \, q^m  
	\]
 valued in differential $(1,1)$-forms,  with trivial cuspidal holomorphic projection. 
We might view \Cref{thm:introArchGenSer} as a `lifting' of their result to the level of Green functions; 
the title of the present paper reflects the role that their work played in motivating our investigation. 

Finally, we turn to an arithmetic-geometric version of this result. Fix an imaginary quadratic field $\kb$ of odd discriminant, let $o_\kb$ be the ring of integers, and suppose that $\calV$ be a Hermitian vector space over $\kb$ of signature $(n-1,1)$ that contains a self-dual lattice $\calL$. 

Following \cite{KRunn2,howard-unitary-2}, we define a stack $\calM_{\calV}$ over $\Spec(o_\kb)$ via a moduli problem that parametrizes $n$-dimensional abelian varieties with an $o_\kb$-action, and a compatible principal polarization, see \Cref{sec:arithUnitary}. There is a family of \emph{special divisors}
\[ 
  \{ \Zed_{\calV}(m) \}_{m  \in \Q_{\neq 0}}
\]
on $\M_{\calV}$, also defined via a moduli problem, and whose complex points can be described in terms of the cycles $Z(m)$ considered previously. In particular, the Green functions $\KGrO(m,v)$ and $\BGrO{m}$ considered above descend to functions $\KGr(m,v)$ and $\BGr{m}$ on $\M_{\calV}(\C)$ with logarithmic singularities along $\Zed_{\calV}(m)(\C)$. 

Moreover, by the results \cite{howard-unitary-2}, the stack $\M_{\calV}$ admits a canonical toroidal compactification $\M_{\calV}^*$ whose boundary $\M_{\calV}^* - \M_{\calV}$ is a divisor. Howard also studies the boundary behaviour of $\KGr(m,v)$, proving that after adding an explicit linear combination of boundary components, the pair
\[ \ZedVKud(m,v) \ : = \ \big( \Zed_{\calV}(m) \ + \ \text{boundary components}  , \ \KGr(m,v) \ \big) \]
defines an element
\[ \ZedVKud(m,v) \ \in \ \ChowHatC(\M_{\calV}^*) \otimes_{\C} \calS^{\vee}, \]
where $\calS \subset S(\calL)$ is a certain $\SL_2(\Z)$-stable subspace of $S(\calL)$, and $\ChowHatC(\calM_{\calV}^*)$ is an \emph{arithmetic Chow group}; here we employ the $\log\text{-}\log$ singular version of this latter group, whose construction is detailed in \cite{bkk} and generalizes the arithmetic Chow groups of \cite{gillet-soule}.

Similarly, work of Bruinier-Howard-Yang \cite{bruinier-howard-yang-unitary} shows that taking Bruinier's Green functions and making a (different) modification at the boundary yields classes
\[ \ZedVBru(m) \ : = \ \big(  \Zed_{\calV}(m) \ + \ \text{boundary components}  , \ \BGr{m} \ \big)  \ \in \ \ChowHatC(\M_{\calV}^*) \otimes_{\C} \calS^{\vee} \]
for $m \neq 0$. 

For appropriate definitions in the case $m=0$, both constructions give rise to formal $q$-expansions
\[ \ThetaVKud(\tau) := \sum_{m} \ZedVKud(m,v) \, q^m \qquad \text{and} \qquad \ThetaVBru(\tau) := \sum_{m} \ZedVBru(m) \, q^m \]
whose coefficients\footnote{If $n > 2$, then $\ZedVBru(m)$ vanishes for all $m < 0$. If $n = 2$, then $\ZedVBru(m)$ may be non-zero for at most finitely many $m<0$.}
are valued in $\ChowHatC(\M_{\calV}^*) \otimes_{\C} \calS^{\vee}$. 

\begin{theorem} The difference $\ThetaVKud(\tau) - \ThetaVBru(\tau)$ is a non-holomorphic modular form of weight $n$ with values in $\ChowHatC(\M_{\calV}) \otimes_{\C} \calS^{\vee}$, and has trivial cuspidal holomorphic projection. \qed
\end{theorem}
Some caution is required in interpreting this statement, since the coefficients of $\ThetaVKud(\tau)$ and $\ThetaVBru(\tau)$ lie in the very large space $\ChowHatC(\calM_{\calV}^*) \otimes_{\C} \calS^{\vee}$. Roughly speaking, the assertion is that
\begin{equation} \label{eqn:introArithThetaDiff}
 \ThetaVKud(\tau) - \ThetaVBru(\tau) \ = \ \sum_{i=1}^r \, f_i(\tau) \, \widehat\Zed_i \ + \ (0, g(\tau,z))   
\end{equation}
 for some $\calS^{\vee}$-valued modular forms $ f_1(\tau)\, \dots f_r(\tau)$ and classes $\hat\Zed_1, \dots \hat \Zed_r \in \ChowHatC(\M_{\calV}^*)$ , together with a family of `smooth classes'
\[ \tau \ \mapsto \ (0, g(\tau, \cdot) )  \ \in \  \ChowHatC(\M_{\calV}^*) \otimes_{\C} \calS^{\vee} \]
that transforms as a modular form in $\tau$; the equality \eqref{eqn:introArithThetaDiff} is to be interpreted as an equality of Fourier coefficients, with an appropriate notion of convergence for such Fourier series.  We discuss this point in \Cref{def:ModularityDef} below. 

We also remark that an analogous modularity result was obtained by Berndt and K\"{u}hn \cite{kuehn-berndt-split} 
for Hecke correspondences on the self-product $X(1) \times X(1)$ of the full level modular curve.
 
Finally, in \Cref{sec:arithHeights} and \Cref{sec:refBHY}, we put our results in the context of two conjectures on arithmetic intersections for unitary Shimura varieties. The first is Kudla's conjecture relating the \emph{arithmetic heights} of the special divisors to derivatives of Eisenstein series. Let
\[ \widehat\omega \ \in \ \ChowHatC(\M_{\calV}^*) \]
denote the tautological bundle on $\M_{\calV}^*$, equipped with an appropriate metric as in \cite{bruinier-howard-yang-unitary}; assuming that there is a reasonable arithmetic intersection theory\footnote{The $\log\text{-}\log$ theory of \cite{burgos-kramer-kuhn} is, strictly speaking, valid for arithmetic schemes; the present authors are unaware of a general theory that simultaneously allows for $\log\text{-}\log$ singularities, and is valid for Deligne-Mumford stacks such as $\calM_{\calV}^*$. See \Cref{sec:arithHeights} for more detail on this point.} for $\calM_{\calV}^*$, we can consider the `arithmetic height'
\[ [ \widehat \Zed : \widehat{\omega}^{n-1} ] \ \in \C, \qquad \text{ for } \widehat\Zed \in \ChowHatC(\calM_{\calV}^*). \]

Applying this to the coefficients of $\ThetaVKud(\tau)$,  Kudla's conjecture in this setting predicts 
\[ [ \ThetaVKud(\tau) : \, \widehat\omega ]  \ \approx \ 2 \, \kappa \, E'_n(\tau, n-1) \ + \ (\text{correction terms}) \]
up to some correction terms involving vertical and boundary components; here $\kappa$ is an explicit constant, and 
	\[ E'_n(\tau, n-1) \ = \ \left[ \frac{\partial}{\partial s} \, E_n(\tau,s) \right]_{s = n-1} \]
	 is a special value of the derivative of the standard Eisenstein series $E_n(\tau,s)$ of weight $n$ attached to $\calL$.
	 
We prove that the integral of the archimedean component $g(\tau,z)$ in the difference of the theta functions \eqref{eqn:introArithThetaDiff} gives the main term in Kudla's conjecture. 
\begin{theorem} Suppose $n>2$. Then
	\begin{equation} \label{eqn:introGreenInt}
		\int_{\calM_{\calV}^*(\C)}  g(\tau,z)\ c_1(\widehat\omega)^{n-1} \ = \ 2 \kappa \ E'_n(\tau, n-1) .
	\end{equation}
	Consequently, given the existence of a `reasonable' arithmetic intersection pairing on $\calM_{\calV}^*$, see \Cref{sec:arithHeights},
	\[ 
		[\ThetaVKud(\tau) - \ThetaVBru(\tau) : \widehat{\omega}^{n-1}] \ = \ 2 \, \kappa \,  E_n'(\tau, n-1) \ + \ (\text{explicit boundary terms}).
	\]
\end{theorem}
It is perhaps worth mentioning that after justifying the interchange of integration over $\calM_{\calV}^*(\C)$ with the regularized integral defining the Green functions, the proof of the theorem  follows very easily from the Siegel-Weil formula and the uniqueness statement in \Cref{thm:introLsec}, see \Cref{thm:GreenInt} in the text.

As a final application, in \Cref{sec:refBHY} we give a refined version of a theorem of Bruinier-Howard-Yang \cite{bruinier-howard-yang-unitary} concerning intersection numbers between the arithmetic cycles $\ZedVBru(m)$ and a certain \emph{small CM cycle} $\mathcal Y$ attached to a positive definite Hermitian lattice $\Lambda$ of signature $(n-1,0)$. Their main result states that a  prescribed linear combinations of these intersection numbers can be expressed as the central derivative of a convolution $L$-function involving the theta function $\Theta_{\Lambda}$, cf. \Cref{thm:BHY}.

By combining \Cref{thm:introXiRegLift} and the Siegel-Weil formula, we prove the following refinement, see \Cref{thm:refBHY} of the text:
 \begin{theorem}
\[ [\ThetaVKud(\tau) : \calY ] \ = \ -  \deg(\calY(\C)) \   \calE_1'(\tau) \otimes \Theta_{\Lambda}(\tau) \]
where $\calE_1'(\tau)$ is the central derivative of an incoherent Eisenstein series of weight 1 associated with $\calY$.  \qed
\end{theorem}

\subsection*{Acknowledgements} The authors thank J. Bruinier, G. Freixas and S. Kudla for helpful conversations, and especially J. Bruinier for suggesting the use of truncated Poincar\'e series in this context. The bulk of the work on this paper was conducted while the authors were 
at McGill University and the CRM; we thank both institutions for their hospitality.
S.S. acknowledges financial support from NSERC.
We would also like to thank the anonymous referees for their helpful comments.

\section{Truncated Poincar\'e series and the Maa\ss \ lowering operator}

\subsection{The Weil representation} \label{sec:WeilRep}
The aim of this section is to fix notation; see \cite[\S 1.1]{brhabil} for a reference for the facts mentioned here, and \cite{gelbart-weilrep, fredrik-weilfqm, Weil} for a more general discussion of the Weil representation. 

Let $(V,Q)$ be a quadratic space over $\Q$ of signature $(p,q)$, and denote by $S(V(\bbA_f))$ the space of \emph{Schwartz-Bruhat functions} on $V(\bbA_f)$, i.e.\ the space of complex valued functions on $V(\bbA_f)$ that are locally constant and compactly supported. The \emph{Weil representation} is an action 
\[ 
  \rho = \rho_f \colon \Mp_2(\bbA_f)  \ \longrightarrow \  \Aut \, S(V(\bbA_f)), 
\]
where $\Mp_2$ is the  \emph{metaplectic extension} of $\SL_2$.

Elements of $\Mp_2(\R)$ are represented by pairs $\tilde \gamma = (\gamma, \phi)$ where $\gamma = ( \begin{smallmatrix}a & b \\ c & d \end{smallmatrix})\in \SL_2(\R)$ and $\phi \colon \uhp \to \C$ is a holomorphic function 
satisfying $\phi(\tau)^2 = c \tau + d$. The product of two elements is given by
\[ 
  (\gamma_1, \phi_1) \cdot (\gamma_2, \phi_2) \ = \ \left(\gamma_1 \cdot \gamma_2, \, \phi_1(\gamma_2 \tau) \,  \phi_2(\tau) \right). 
\]
The group $\Gt:=\Mp_2(\Z)$, defined as the inverse image of $\G:=\SL_2(\Z)$ under the covering map, is generated by the elements
\[ 
  \mathbf{S} \ := \ \left( \begin{pmatrix} & -1 \\ 1 & \end{pmatrix}, \, \sqrt{\tau} \right) \qquad 
  \text{and} \qquad  \mathbf{T} \ := \ \left( \begin{pmatrix} 1 & 1 \\  & 1 \end{pmatrix}, \, 1 \right).
\]

Suppose $L\subset V$ is an even integral lattice (i.e. $L$ is a free $\Z$-submodule for which $Q(x) \in  \Z$ for all $x \in L$) and let $L'$ be the dual lattice. 

We let $\Zhat = \prod_p \Z_p$, where the product is over all rational primes, and define $\Lhat = L \otimes_\Z \Zhat$.
For a coset $\mu \in L'/L \cong (\Lhat)'/\Lhat$, let
\[ 
  \varphi_{\mu}  \ := \ \text{ characteristic function of } \mu  + \Lhat  \ \in \ S(V(\bbA_f)) ; 
\] 
the finite-dimensional space 
\[ 
  S(L) \ := \ \mathrm{span}_{\C} \left( \varphi_{\mu}, \ \mu \in L'/L \right)  \ \simeq \ \C[L'/L] 
\]
 is stable under $\Gt$. 
The restriction is denoted
\[ 
  \rho_L \colon \Gt \ \to \ \Aut S(L). 
\]
It can be seen, for example via the explicit formulas in \cite[\S 1.1]{brhabil}, that the image of $\rho_L$ is in fact a finite group.  

Moreover, when $\dim (V)$ is even,  $\rho_L$ factors through the map $\Gt \to \SL_2(\Z)$, and we denote the resulting action of $\SL_2(\Z)$ on $S(L)$ by the same symbol $\rho_L$. 

For $v \in S(L)^\vee$ and $w \in S(L)$, we frequently simply write $v \cdot w$ or $vw$ for the image of an element
 $v \otimes w \in S(L)^\vee \otimes S(L)$ under the canonical contraction map $S(L)^{\vee} \otimes_{\C} S(L) \to \C$.

Finally, it will be useful to introduce the Hermitian pairing (conjugate-linear in the second argument)
\[ 
  \langle \cdot, \cdot \rangle_{S(L)} \colon S(L) \times S(L) \to \C, \qquad \text{determined by } \langle \varphi_{\mu}, \, \varphi_{\nu} \rangle_{S(L)} \ = \ \delta_{\mu,\nu}. 
\]
It is easily checked, for example by using the explicit formulas in \cite[\S 1]{brhabil} that the induced $\C$-linear isomorphism
\begin{equation} \label{eqn:WeilRepDualConj}
 \overline{S(L)} \isomto S(L)^{\vee} , \qquad w \mapsto \langle \cdot , \, w \rangle_{S(L)}
\end{equation}
identifies the conjugate representation $\overline{\rho_L}$ with the dual representation $\rho_L^{\vee}$. 

This discussion is intended to justify the following abuse of notation: if $\varphi \in S(L)$, let $\overline \varphi \in S(L)^{\vee}$ denote the linear functional $\langle \cdot, \varphi \rangle_{S(L)}$.

\subsection{Some spaces of non-holomorphic modular forms} \label{sec:modforms}
Suppose $(W,\rho)$ is a finite dimensional representation of $\Gt = \Mp_2(\Z)$. 
If $k \in \frac12 \Z$ and $\tilde \gamma = (\gamma, \phi) \in \Gt$,
 define the \emph{slash operator}\footnote{For brevity, we do not include the representation $\rho$ in the notation but of course the slash action does depend on $\rho$.} on functions $f \colon \uhp \to W$ by the formula
\[ 
  f\mid_{k}[\tilde \gamma] \, (\tau) \ := \ \phi(\tau)^{-2k} \cdot \rho(\tilde \gamma)^{-1}  \, f \left( \frac{a \tau + b }{c\tau + d} \right),
\]
where $\gamma = \abcd$.

We say that a function $f \colon \uhp \to W$ \emph{transforms as a modular form} of weight $k$ and representation $\rho$ if 
\begin{equation}
  \label{eq:modtrafo}
   f\mid_{k}[\tilde\gamma] \, (\tau) \ = \ f(\tau) \qquad \text{for all } \tilde\gamma \in \Gt. 
\end{equation}

In this paper, we will primarily be interested in the case where $(W, \rho)$ is either $(S(L), \rho_L)$ or its dual.

Suppose that $(W, \rho) = (S(L), \rho_L)$, and consider the element $ \mathbf Z = (  - \mathtt{Id}, i)$, so that $\mathbf Z^2 = ( \mathtt{Id}, -1)$. Then 
\[ 
  f\mid_{k}[ \mathbf Z^2] \ = \ (-1)^{-2k + q - p} \, f, 
\]
for any function $f \colon \uhp \to S(L)$,
so there exist no non-zero functions satisfying \eqref{eq:modtrafo} unless $2k \equiv q-p \bmod 2$, in which case $\mid_{k}[\mathbf Z^2]$ acts trivially.
Moreover, such functions then take values in the subspace of $S(L)$ spanned by the vectors
\begin{equation}
  \label{eq:tildephi}
  \tilde\varphi_\mu \ :=  \ \frac12(\varphi_{\mu} + (-1)^{-k+\frac{q-p}{2}}\varphi_{-\mu})
\end{equation}
for $\mu \in L'/L$.

Note also that if $f \colon \uhp \to S(L)$ transforms as a modular form of weight $k$, then 
\[ 
  g(\tau) := v^{k} \overline{f(\tau)} 
\]
transforms as a modular form of weight $-k$ for the dual representation $\rho_L^{\vee}$.

\begin{definition} \label{def:regpairing}
Suppose $f \colon \uhp \to S(L)$ transforms as a modular form of weight $k$ with representation $\rho_L$, and $g \colon \uhp \to S(L)^{\vee}$ transforms with weight $-k$ and representation $\rho_L^{\vee}$, so that the product
\[ 
  f(\tau) \cdot g(\tau) 
\]
is $\Gt$-invariant. Following \cite{boautgra}, we define the \emph{regularized pairing} of $f$ and $g$, when it exists, as
\[
  \langle f,g \rangle^\reg  \ := \   \CT_{s=0}\ \left( \lim_{T \to \infty} \ \int_{\calF_T}  f(\tau)  \ g(\tau) \ v^{-s} \ d\mu(\tau) \right),
\]
where $s \in \C$ is a complex parameter, $\CT_{s=0}$ stands for ``the constant term in the Laurent expansion at $s=0$", and for $T \in \R_{>0}$,
\[ 
  \calF_T \ := \ \left\{ \tau = u+ iv \in \uhp\ | \  \abs{\tau} \geq 1, \  u \in \left[-\frac12, \frac12\right] \ \text{ and } v \leq T \right\} 
\]
is a truncation of the standard fundamental domain for $\SL_2(\Z) \bs \uhp$. We also write $d \mu(\tau) =  du\cdot dv / v^2$ for the usual hyperbolic measure.

We say the \emph{regularized pairing exists} if for $\Re(s)$ sufficiently large, the limit $T \to \infty$ defines a holomorphic function in $s$  that has a meromorphic continuation to $\Re(s)>-\epsilon$ for some $\epsilon>0$, and so the above definition makes sense.
 
Similarly, if $f\colon \uhp \to S(L)$ and $g\colon \uhp \to S(L)$ both transform as modular forms of weight $k$ with representation $\rho_L$, we define the \emph{regularized Petersson pairing}
\[ 
  \langle f,\, g\rangle_{\Pet}^{\reg} \ := \ \langle f, \,  v^{k}  \, \overline g \rangle^{\reg}
\]
when it exists.
If  $f, g \in M_k(\rho_L)$ are holomorphic modular forms, and at least one of them is cuspidal, 
then $\langle f,\, g\rangle_{\Pet}^{\reg} = \langle f,\, g\rangle_{\Pet}$ agrees with the usual Petersson inner product.

\end{definition}
\begin{definition}[\cite{brfugeom}] Suppose $k \in \frac12 \Z$, and let $H_k(\rho_L)$  be the space of (twice continuously differentiable) functions $f \colon \uhp \to S(L)$ such that
  \begin{enumerate}
  \item
    $f \mid_{k} [\tilde\gamma] = f$ for all $\gamma \in \tilde\Gamma$;
  \item
    there is a $C>0$ such that\footnote{Throughout, when we employ the asymptotic notation $f(\tau) = O(g(\tau))$ for a vector-valued function $f$, we simply mean that the components of $f$ with respect to any basis satisfy the given asymptotic.} $f(u + iv)=O(e^{C v})$ as $v\to \infty$ uniformly in $u$;
  \item
    $\Delta_k f = 0$, where
    \begin{align*}
      \Delta_k := -v^2\left( \frac{\partial^2}{\partial u^2}+
        \frac{\partial^2}{\partial v^2}\right) + ikv\left(
        \frac{\partial}{\partial u}+i \frac{\partial}{\partial v}\right)
    \end{align*}
    is the hyperbolic Laplace operator in weight $k$;
    \item	
    and	finally, that 
    		\[ 
                   \xi_k(f)  \ := \ v^{k-2} \overline{\Low(f)} 
                \]
    		is a holomorphic cusp form of weight $2 - k$ valued in $\rho_L^{\vee}$. Here
    		\[     
                    \Low = - 2 i v^2 \frac{\partial}{\partial \overline \tau} 
                \]
    	is the Maa\ss\, lowering operator, and satisfies $ \Low(g \mid_{k} [\tilde \gamma]) \ = \ \Low(g)\mid_{k-2} [\widetilde \gamma]$ for any function $g \colon \uhp \to S(L)$; in particular, it lowers the weight by two.
  \end{enumerate}
Functions satisfying $(i)$ -- $(iii)$ are called \emph{harmonic weak Maa\ss\ forms} of weight $k$, valued in $S(L)$.
\end{definition}
Contained in $H_k(\rho_L)$ is the space 
\[
	M^!_k(\rho_L) \ := \ \ker(\xi_k) 
\]
of \emph{weakly holomorphic forms}; these are precisely the forms of weight $k$ that are holomorphic on $\uhp$ and meromorphic at the cusp $\infty$. In turn, the space $M^!_k(\rho_L)$ contains the space of classical holomorphic modular forms $M_k(\rho_L)$, namely those forms that are holomorphic at $\infty$, and the space of cusp forms $S_k(\rho_L)$. 

We record some consequences of the definitions that will prove useful in the sequel, cf.\ \cite[\S 3]{brfugeom}. 
First, the Fourier expansion of  $f \in H_k(\rho_L)$  admits a decomposition
\begin{equation} \label{eqn:weakMaassDecomp}
 f(\tau) \ =\  f^+(\tau) \ + \ f^-(\tau) 
\end{equation}
into its \emph{holomorphic} and \emph{non-holomorphic parts}, where the holomorphic part
\[ 
	f^+(\tau) \ = \ \sum_{m \gg - \infty} c_f^+(m) \, q^m 
\]
has only finitely many negative terms. The non-holomorphic part
\[
	f^-(\tau) \ = \ \sum_{m < 0 } c_f^-(m) \  \Gamma(1-k, 4 \pi m v)\, q^m 
\]
has non-zero Fourier coefficients only for negative indices; here
\[
  \Gamma(1-k, a ) = \int_{a}^\infty e^{-t}t^{k-2}\, dt .
\]
Moreover, a form $f$ is in $M^!_k(\rho_L)$ if and only if $f^- = 0$.
Finally, we note that
\begin{equation} \label{eqn:WHMPrincAsymp} 
  f \ -  \ \sum_{n\leq 0} c^+_f(n) \, q^n \ = \ O(e^{-Cv}) \qquad \text{ as } v \to \infty
\end{equation}
for some constant $C>0$. 

\subsection{A family of harmonic weak Maa\ss \ forms}
\label{sec:relat-non-holom}
In this section, we fix a special family $\{F_{m,\mu}\} \subset H_k(\rho_L)$ that will play an important role throughout.

Suppose for the moment that $k <0$. Following \cite{brhabil}, if $m \in \Q_{> 0}$ and $\mu \in L'/L$ with $m \in Q(\mu) + \Z$, define
\begin{equation}
  \label{eq:nhp-br}
  F_{m, \mu}(\tau) = \frac{(4 \pi m)^{1-k}}{4\cdot \Gamma(2-k)} \sum_{\tilde \gamma \in \Gt_\infty \bs \Gt} \left( v^{1-k} \, M(1, 2- k , 4 \pi m v) \,   e^{- 2 \pi i m u} \, \varphi_{\mu} \right) \Big\vert_{k} [\tilde \gamma]
\end{equation}
where $M(a,b,c)$ is Kummer's confluent hypergeometric function (see \cite{abst}, 13.1.2, 13.1.32); 
the observant reader will note that $F_{m,\mu}(\tau)$ differs from its namesake in \cite{brhabil} by a factor of $1/2$. 
The forms $F_{m, \mu}$ span $H_k(\rho_L)$ if $k<0$ (see Proposition 1.12 in \cite{brhabil}).

The holomorphic part $F_{m,\mu}^+(\tau)$ is of the form
	\begin{equation}	\label{eqn:FmHolPart}
		F_{m, \mu}^+(\tau) \ = \ q^{-m}\tilde\varphi_\mu \ + \ \sum_{ n \geq 0} \, c_{F_{m,\mu}}^+(n) \, e^{2 \pi i n \tau} 
	\end{equation} 
for some coefficients $c_{F_{m,\mu}}^+(n)$ and where $\tilde\varphi_\mu = \frac12(\varphi_{\mu} + (-1)^{k+\frac{q-p}{2}}\varphi_{-\mu})$; in fact, $F_{m,\mu}$ is the unique harmonic weak Maa\ss\ form whose holomorphic part has this form \cite[Proposition 3.11]{brfugeom} and moreover $\xi(F_{m,\mu})$ is a (holomorphic, cuspidal) Poincar\'e series (cf. Remark 3.10 in \cite{brfugeom}).

This observation suggests the following approach to defining an analogous function $F_{m,\mu}$ in the case that $k \geq 0$.

We first recall the following \emph{modularity criterion}, due to Borcherds \cite{bogkz}. Let $\mathsf{Sing}_{2-k}(\rho_L)$ be the space of formal Fourier polynomials with negative indices that are invariant under the action of the elements $\mathbf Z$ and $\mathbf T$ of $\Gt$:
\[
	 \mathsf{Sing}_{2-k}(\rho_L^{\vee}) \ \ :=  \ \left\{ P \ =\ \sum_{ m \leq 0} a_P(m) \, q^m \ \Big| \ a_P(m) \in S_L^{\vee},\  P\mid_{2-k}[\mathbf Z] \ = \ P\mid_{2-k}[\mathbf T] = P \right\}
\]
where $q = e^{2 \pi i \tau}$. We have a map
\[
 	\psi\colon	\mathsf{Sing}_{2-k}(\rho_L^{\vee}) \ \to \ M_{k}(\rho_L)^{\vee}, \qquad  \psi(P)(g) \ = \    \sum_{m \leq 0} a_P(m) \cdot c_g(-m)  ,
\]
and also the \emph{principal part map}
\[
	 P \colon H_{2-k}(\rho_L^{\vee}) \ \to \ \mathsf{Sing}_{2-k}(\rho_L^{\vee}), \ \qquad \ \ P(f) \ =   \ \sum_{m \leq 0} c_f(m) \, q^m.
\]
Note that this latter map factors through the quotient by the space of cusp forms $S_{2-k}(\rho_L^{\vee})$.

\begin{theorem}[{\cite[Theorem 3.1]{bogkz}}] \label{thm:serreduality}
The sequence induced by the above maps
\[
	\begin{CD}  0 @>>> M^!_{2-k}(\rho_L^{\vee}) / S_{2-k}(\rho_L^{\vee}) @>P>> \mathsf{Sing}_{2-k}(\rho_L^{\vee}) @>\psi>> M_{k}(\rho_L)^{\vee} @>>> 0
	\end{CD}
\]
is exact. \qed
\end{theorem}

This sequence will be used to normalize the forms $F_{m, \mu}$: for each $k$,  fix a splitting
\begin{equation}\label{eqn:etaSplitting}
	\begin{tikzcd}
	 			0 \ar{r} &  M^!_{2-k}(\rho_L^{\vee}) / S_{2-k}(\rho_L^{\vee}) \ar{r} &  \mathsf{Sing}_{2-k}(\rho_L^{\vee})  \ar{r} & M_k(\rho_L)^{\vee} \ar{r} \ar[bend right, above]{l}{}[swap]{\eta}  & 0.
		\end{tikzcd}
\end{equation}
The map $\psi^{\vee} \colon M_{k}(\rho_L) \to \mathsf{Sing}_{2-k}(\rho_L^{\vee})^{\vee}$ may be extended to $H_k(\rho_L)$, by setting
\[
	\psi^{\vee}(F)(P) \ := \ \sum_{m \leq 0} a_P(m) \cdot c_F^+(-m) , \qquad F \in H_k(\rho_L)
\]
where $c_F^+(-m)$ are Fourier coefficients of the holomorphic part $F^+$ of $F$. 

\begin{lemma} \label{lem:uniqueHPSeries}
Suppose $m \in \Q$ and $\mu \in L'/L$ with $m \equiv Q(\mu) \bmod \Z$. 
For each $k \in \frac12\Z$, there is a unique element $F_{m,\mu} \in H_k(\rho_L)$ such that 
	\begin{enumerate} 
		\item The holomorphic part of $F_{m, \mu}$ has the form
			\[ 
                           F^+_{m,\mu} (\tau) \ = \ q^{-m}\tilde\varphi_\mu \ + \ \sum_{n \geq 0} c(n) \, q^n \ \qquad \text{ for some } c(n) \in S(L); 
                        \]
		\item and $\psi^{\vee}(F_{m,\mu})(P)  = a_P(m)(\varphi_{\mu})$ for every $P \in \mathrm{im}(\eta)$. 
	\end{enumerate}  
\end{lemma}

\begin{remark}\label{rem:eta}\ \\
  \begin{itemize}
  \item Note that if $k<0$ then $M_k(\rho_L) =\{ 0\}$ and $\eta  = 0$, and condition {\it (ii)} is superfluous; in this case the construction of $F_{m,\mu}$ as in the lemma coincides with  \eqref{eq:nhp-br} for $m>0$, and is identically zero for $m \leq 0$.
  \item In general,  $F_{m,\mu} = 0$ for almost all negative $m$.
  \item The form $F_{m, \mu}$ will depend on the choice of $\eta$, but as this choice will not play a significant role in our applications, we omit this dependence from our notation. 
  \item A natural choice of $\eta$ can be constructed as follows. For $m \in \Q$ and $\mu \in L'/L$, let $\phi_{m,\mu} \in M_{k}(\rho_L)^{\vee}$ denote the functional
  \[
 	\phi_{m,\mu}\colon f \mapsto  a_{f}(m)(\mu), \qquad \text{for all } f(\tau) = \sum_{n} a_f(n) q^n \in M_k(\rho_L) .
 \]
The collection of all $\{ \phi_{m,\mu}\}$ with $m \in \Q$ and $\mu \in L'/L$ spans the finite-dimensional space $M_k(\rho_L)^{\vee}$, and so there are positive integers $n_1 \leq \ldots \leq n_d$ and $\mu_1, \ldots, \mu_d \in L'/L$ such that $\{\phi_{n_i,\mu_i}\}$
  is a basis of $M_k(\rho_L)^\vee$. Then a possible choice for $\eta$ is given by setting
  \[ \eta(\phi_{n_i, \mu_i}) = q^{-n_i} \tilde\varphi_{\mu_i}^{\vee} ,\]
  where 
  \[ \tilde\varphi_{\mu_i}^{\vee} \ =  \ \varphi_{\mu_i}^{\vee}  \ +  \ (-1)^{-k + \frac{q-p}{2}} \varphi_{-\mu_i}^{\vee}  \ \in \ S(L)^{\vee}\] 
  is the vector dual to $\tilde\varphi_{\mu_i}$ with respect to the basis $\{ \varphi_{\mu}\}$ of $S(L)$.
  Note that with this choice, we have $F_{m,\mu} = 0$ for $m \leq 0$ unless $m = -n_i$ for some $i \in \{1, \ldots, d\}$ and $\mu = \pm \mu_i$. Moreover, the family of non-zero $F_{m,\mu}$, over all $m \in \Q$ and $\mu \in L'/L$, form a basis of $H_k(\rho_L)$. If $k>2$ and $m<0$, the forms $F_{m,\mu}$ are  (holomorphic) Poincar\'{e} series.
  \end{itemize}
\end{remark}

\begin{proof}
The existence of a form $F \in H_{k}(\rho_L)$ satisfying $(i)$ follows from \cite[Proposition 3.11]{brfugeom}. 
Recall we had fixed a splitting morphism $\eta$ as in \eqref{eqn:etaSplitting}. Let 
\[ 
  \phi_{m, \mu} \in \mathsf{Sing}_{2-k}(\rho_L^{\vee})^{\vee}, \qquad \phi_{m,\mu}(P) = a_P(m)(\varphi_{\mu}) 
\]
and set
\[ 
  g  \ = \  \eta^{\vee} \left( \psi^{\vee}(F) \ - \ \phi_{m,\mu} \right) \in M_k(\rho_L);
\]
then the difference $F - g$ satisfies the properties in the lemma, proving existence.

To show uniqueness, suppose $F_1$ and $F_2$ satisfy both properties, and consider $F := F_1 - F_2$. Note that $P(F)$ is constant. We first claim that $F \in M_k(\rho_L)$. 
Indeed, by \cite[Proposition 3.5]{brfugeom},
\[
	\la \xi(F), h \ra_{\Pet} \ = \ \sum_{m \leq 0}  \ c_F^+(m) \cdot c_h(-m) \ \qquad \text{ for all } h \in M_{2-k}(\rho_L^{\vee}),
\]
where the left hand side is the usual Petersson pairing between $S_{2-k}(\rho_L^{\vee})$
and $M_{2-k}(\rho_L^\vee)$; note here that $\xi(F)$ is cuspidal. On the other hand,
since $P(F)$ is constant this quantity vanishes if $h$ is cuspidal. 
Since the Petersson pairing is non-degenerate when restricted to  cusp forms, it follows that $\xi(F) = 0$.
Thus, $F \in M_k^!(\rho_L)$ and has constant principal part, which implies the stronger statement $F \in M_k(\rho_L)$.

By the second property in the lemma, 
\[
	 0 \ = \ (\eta^{\vee} \circ \psi^{\vee})(F) \ = \ F,
\] 
proving uniqueness.
\end{proof}

When  $m \notin Q(\mu) + \Z$, set $F_{m,\mu} = 0$. We may define an $S(L)^{\vee} \otimes_{\C} S(L)$-valued version
\[ 
	F_{m} \colon \mathbb H \ \to \ S(L)^{\vee} \otimes_{\C} S(L) 
\]
 by setting 
\[
	F_{m}(\tau)(\varphi) \ = \ \sum_{\substack{\mu \in L'/L\\Q(\mu) \equiv m \bmod \Z}} a_{\mu} \, F_{m,\mu}(\tau)
\]
whenever $\varphi = \sum_{\mu} a_{\mu} \varphi_{\mu} \in S(L)$. 

\subsection{Truncated Poincar\'{e} series}
\label{sec:trunc-poinc}
\newcommand{\Rp}{\R_{>0}}
For a parameter $w \in \Rp$ and $\tau = u+iv \in \uhp$, consider the cutoff function 
\[
  \sigma_{w}(\tau) =
  \begin{cases}
     1, & \text{ if } v \geq w,\\
     0,               & \text{ if } v < w.
  \end{cases}
\]
Fix $w \in \Rp$ and a half-integer $k \in \frac12 \Z$ such that $2 k \equiv p \bmod{2}$. For a coset $\mu \in L'/L$ and a rational number $m \in Q(\mu) + \Z$, define the \emph{truncated Poincar\'{e} series} $P_{m, w, \mu} \colon \uhp \to S(L)$ by the formula
\begin{equation}  \label{eqn:TPSdef}
    P_{m, w, \mu}(\tau)\ =  \ \frac{1}{4}  \,
                                    \sum_{ \tilde \gamma \in \Gt_\infty \bs \Gt} \left( \sigma_{w}(\tau) \, e^{-2 \pi i m \tau }\,  \varphi_{\mu} \right) \big\vert_{k}[\tilde \gamma], 
\end{equation}
where $\Gt_\infty$ is the subgroup of $\Gt$ generated by $\mathbf T$. 

For a fixed $\tau$, there are only finitely many elements $\gamma \in \Gamma_{\infty} \backslash \SL_2(\Z)$ such that $\sigma_w(\gamma \tau) \neq 0$; thus the sum \eqref{eqn:TPSdef} is locally finite, and in particular, is absolutely convergent. By construction, $P_{m, w, \mu}(\tau)$ is evidently invariant under the weight $k$ slash operator, i.e.\ it is a (discontinuous!) ``modular form'' of weight $k$. 

\begin{remark}  \label{rmk:intrinsicPoincare}
It will be useful to give  a more intrinsic definition for a truncated Poincar\'e series $P_{m,w}(\tau)(\varphi)$  valid for any $m \in \Q$ and $\varphi \in S(V(\bbA_f))$. Fix an integer $N \gg 0$ such that $m \in N^{-1}\Z$ and $\rho_L(\Gt_{\infty,N})$ acts trivially on $\varphi$, where $\Gt_{\infty,N} = \langle \mathbf T^N \rangle$. 
Then, for any fixed parameter $w \in \Rp$, set 
\[  
  P_{m,w}(\tau)(\varphi)\ :=  \ \frac{1}{4N}  \,
	                                    \sum_{ \tilde \gamma \in \Gt_{\infty,N} \bs \Gt} \left( \sigma_{w}(\tau) \, e^{-2 \pi i m \tau }\,  \varphi \right) \big\vert_{k} [\tilde \gamma], 
\]
which defines a map                           
\[ 
  P_{m,w}(\tau) \colon \uhp \ \to \ S(V(\bbA_f))^{\vee}  \otimes_{\C}  S(V(\bbA_f)). 
\]
Since $S(L)$ is $\Gt$-invariant, restricting to $S(L)$ yields a map
\[  
  P_{m,w}(\tau) \colon \uhp \ \to \ S(L)^{\vee}  \otimes_{\C}  S(L), 
\]
denoted by the same symbol. 
We may thereby view $P_{m,w}(\tau)$ as  valued in $S(L)^{\vee}  \otimes_{\C}  S(L)$, where $\Gt$ acts by $\rho^{\vee} \otimes 1$. 

Alternatively, if we define $P_{m, w, \mu}(\tau) = 0$ whenever $m \notin Q(\mu) + \Z$, then for any 
\[ 
  \varphi  \ = \ \sum_{\mu} \, a_{\mu} \, \varphi_{\mu} \ \in S(L), 
\]
it is easy to check that 
	\begin{equation} \label{eqn:PmwAltDef}
	 P_{m,w}(\tau)(\varphi)\   = \  \sum_{\mu \in L'/L} \, a_{\mu} \, P_{m, w, \mu}(\tau) \ = \ \sum_{\substack {\mu \in L'/L \\ Q(\mu) \equiv m \bmod{\Z}}} a_{\mu} \, P_{m, w, \mu}(\tau) ;  
	\end{equation}
this provides an alternative definition of $P_{m,w}(\tau)(\varphi)$. 
\end{remark}

\subsection{A section of the Maa\ss \ lowering operator}
\label{sec:invert-xi}
In this section, we show that for a form $f$ satisfying certain mild analytic conditions, our  Poincar\'e series can be used to generate the Fourier coefficients of a distinguished preimage $F \in \mathbf{L}^{-1} (f)$ of $f$ under the Maa\ss \ lowering operator.

We begin by fixing some notation. Suppose $L \subset V$ is an even integral lattice, and $\kappa \in \frac12 \Z$ with $2\kappa \equiv p-q \bmod{2}$. 
\begin{definition} \label{def:modgrowthforms}
Let $\ALmod{\kappa}(\rho_L^{\vee})$ be the space of $\calC^{\infty}$ functions 
\[ f \colon \uhp \ \to \ S(L)^{\vee} \]
such that 
	\begin{enumerate}
	\item $f\mid_{\kappa} [\tilde \gamma] (\tau)  = f(\tau)$ for all $\tilde \gamma \in \Gt$;
	\item $f$ has at most `moderate growth at $\infty$', \\
             i.e. for all $\alpha, \beta$, there is some $\ell  \in \Z$  (possibly depending on $\alpha$ and $\beta$), such that 
              $\frac{\partial^\alpha}{\partial u^\alpha}\frac{\partial^\beta}{\partial v^\beta}f(\tau) = O(v^{\ell})$ 
              as $v \to \infty$.
	\item writing the Fourier expansion of $f$ as
		\[ f(\tau) \ = \ \sum_{m \in \Q} c(m, v) \, e^{2 \pi i m \tau} , \qquad \text{ with } c(m, v) \colon \R_{>0} \to S(L)^{\vee}, \]
		we require that the constant term $c(0, v)$ has the form
		\[ c(0, v) \ = \ \sum_{\mu \in L'/ L} \ \sum_{i=1}^r \, \alpha_{\mu,i} \, v^{\beta_{\mu,i}} \, \varphi_{\mu}^{\vee} \ + \ \tilde c(0,v) \]
		for some $\alpha_{i,\mu} \in \C$, $\beta_{\mu,i} \in \Q$, and a smooth function $\tilde c(0, v) \colon \uhp \to S(L)^{\vee}$ satisfying $\tilde c(0,v) = O(e^{-Cv})$ as $v \to \infty$ for some $C>0$.
	\end{enumerate}
\end{definition}

\begin{definition}
Define another space   $\AkappaLexp(\rho_L^{\vee})$, consisting of $\calC^{\infty}$ functions $F \colon \uhp \to S(L)^{\vee}$ such that
	\begin{enumerate}
		\item $F(\tau)\mid_{\kappa} [\tilde \gamma] = F(\tau) $ for all $\widetilde \gamma \in \Gt$;
		\item $F$ has at worst `exponential growth at $\infty$', i.e.\ $F(\tau) = O(e^{Cv})$ as $v \to \infty$ for some constant $C>0$; and
		\item $\Low(F) \in \ALmod{\kappa-2}(\rho_L^{\vee})$.
	\end{enumerate}
	Note that $\AkappaLexp(\rho_L^{\vee})$ contains $H_{\kappa}(\rho_L^{\vee})$ and therefore also $M^!_{\kappa}(\rho_L^{\vee})$. 
\end{definition}

The following Proposition generalizes Theorem 3.7 in \cite{brfugeom}.
\begin{proposition} \label{prop:Lexactseq}
  The following sequence is exact:
  \begin{equation}
    \label{eq:Lsurj}
      \xymatrix@1{ 0 \ar[r] & M_{\kappa}^!(\rho_L^{\vee}) \ar[r] & \AkappaLexp(\rho_L^{\vee}) \ar[r]^-{\Low} &\ALmod{\kappa-2}(\rho_L^{\vee})\ar[r] & 0}.
  \end{equation}
\end{proposition}
\begin{proof}[Sketch of proof:]
The proof proceeds along the same lines as \cite[Theorem 3.7]{brfugeom}, except we replace the sheaves of  $\calC^{\infty}$ forms appearing in \emph{op.~cit.} with the sheaves of $\log$-singular forms considered in \cite{bkk}. 

We sketch the argument here: fix a normal subgroup $\Gamma_0 \subset \Gt = \Mp_2(\Z)$ of finite index, such that $\Gamma_0$ acts freely on $\uhp$. Write $\calO$ for the sheaf of holomorphic functions on the compactified curve $X = \Gamma_0 \bs\uhp^*$, and $\mathscr{E}^{p,q} = \mathscr{E}^{p,q}\langle D \rangle$ for the sheaf of $\calC^\infty$ differential forms of  Hodge type $(p,q)$ and with `logarithmic growth along the cuspidal divisor $D$', as in \cite[Definition 2.2]{bkk}; using Theorem 2.13 of \emph{op.~cit.}, these sheaves can be identified with sheaves of differential forms generated by differentials of functions with moderate growth at the cusps.

Let $\calL_{\kappa,L}$ be the $\calO$-module sheaf of modular forms of weight $\kappa$ and representation $\rho_L$ on $X$. The sections
in $\calL_{\kappa,L}(U)$ are given by holomorphic $S(L)$-valued functions on the preimage of $U$ under the canonical projection from $\uhp$ to $X$ that satisfy the usual transformation property with respect to $\Gamma_0$ and are holomorphic at the cusps.
  Let  $n > 0$ be a positive integer. By the Dolbeault Lemma for the sheaves $\mathscr{E}^{p,q}$, see \cite[Lemma 2.44]{bkk}, the sequence of sheaves
  \begin{equation}
    \label{eq:resolution}
      \xymatrix{ 0 \ar[r] & \calL_{\kappa,L} \otimes_{\calO} \calO_{nD} \ar[r] & \mathscr{E}^{0,0} \otimes_{\calO} \calL_{\kappa,L} \otimes_{\calO} \calO_{nD} \ar[r]^{\overline\partial} & \mathscr{E}^{0,1} \otimes_{\calO} \calL_{\kappa,L} \otimes_{\calO} \calO_{nD} \ar[r]& 0},
  \end{equation}
is exact. Moreover, since $\mathscr{E}^{p,q} \otimes_{\calO} \calL_{k,L} \otimes_{\calO} \calO_{nD}$ is a fine sheaf (as a module over the sheaf of $\calC^\infty$ functions on $X$), taking global sections gives an exact sequence in cohomology
  \begin{equation}
    \label{eq:dolbeault}
      \xymatrix{ 0 \ar[r] & (\calL_{k,L} \otimes \calO_{nD})(X) \ar[r] & \mathscr{E}^{0,0}(X,\calL_{k,L} \otimes \calO_{nD}) \ar[r]^{\overline\partial} &  \mathscr{E}^{0,1}(X,\calL_{k,L} \otimes \calO_{nD}) \ar[d]&\\
       & & & H^1(X,\calL_{k,L} \otimes \calO_{nD}) \ar[r] & 0}
  \end{equation}
For fixed $\kappa$, an application of Serre duality, as in the proof of \cite[Theorem 3.7]{brfu06}, implies that $H^1(X,\calL_{\kappa,L} \otimes \calO_{nD})$ vanishes for $n$ sufficiently large.

We now show that $\Low$ is surjective in the sequence \eqref{eq:Lsurj}. 
Let $n$ be a large enough positive integer such that $H^1(X,\calL_{\kappa,L} \otimes \calO_{nD}) = 0$ in the above discussion, and suppose $f \in \ALmod{\kappa-2}(\rho_L^{\vee})$.  
Since $f$ has moderate growth, \cite[Theorem 2.13]{bkk} implies that pullback of the differential form 
	\[ \eta \ := \ v^{-2} \ f \ d \overline\tau \]
along the covering map 
\[ X \to  \Gt \backslash \uhp^*\] 
is a $ \Gt / \Gamma_0$-invariant global section in $\mathscr{E}^{0,1}(X, \calL_{\kappa,L} \otimes \calO_{n D})$. Therefore, taking $\Gt/\Gamma_0$-invariants in the exact sequence \eqref{eq:dolbeault}, there exists 
\[ F \ \in \ \mathscr{E}^{0,0}(X, \calL_{\kappa,L} \otimes \calO_{n D})^{\Gt/\Gamma_0} 
\]
with $\overline \partial F = \eta$, or equivalently, $\Low(F) = f$. Moreover, by construction $F$ may be written, in a neighbourhood of each cusp, as a sum of products of moderate growth and meromorphic forms, and so has at worst exponential growth approaching the cusps. Since $F$ is furthermore $\Gt/\Gamma_0$-invariant, it descends to  $\Gt \backslash \uhp^*$, i.e.\ it is an element of $\AkappaLexp(\rho^{\vee}_L)$.
\end{proof}

Suppose $F \in \AkappaLexp(\rho_L^{\vee})$ and $f = \Low(F) \in \ALmod{\kappa-2}(\rho_L^{\vee})$ with Fourier expansions
	\[	F(\tau) \ = \ \sum_m c_F(m,v) \, e^{2 \pi i m \tau} 
			\qquad \text{and} \qquad
		f(\tau) \ = \ \sum_m c_f(m,v) \, e^{2 \pi i m \tau} 
	\]
where $\tau = u + i v$. 

The relation $\Low(F) = f$ implies that for each $m$,
	\begin{equation} \label{eqn:cFlowering} c_F(m,v) \ = \ c_F(m,1)  \ + \ \int_{1}^v c_f(m,t) \, t^{-2} \, dt  .
	\end{equation}
Since $f(u+it) = O(t^{l})$ as $t \to \infty$ by assumption, 
	\begin{equation}  \label{eqn:cfbound}
	| c_f(m,t)  | \ = \  e^{2 \pi m t} \cdot \left| \int_{-1/2}^{1/2} \, f(u+it) \, e^{-2 \pi i m u} \,  du \right| \ \ll \ e^{2 \pi m t} \cdot t^{l} 
	\end{equation}
for all $t>1$, where the implied constant is independent of $m$. 
In particular, for all $m<0$, substituting this bound in \eqref{eqn:cFlowering} gives 
	\[ \kappa_F(m) \ := \ \lim_{v \to \infty} \, c_F(m,v) \ < \ \infty. \]
	Moreover, it follows from the assumption that $F$ has exponential growth that 
		\[ \kappa_F(m) \ = \ 0 \qquad \text{ for all but finitely many } m <0. \]
		and
	\begin{equation} \label{eqn:cFstrong}
		c_F(m,v) - \kappa_F(m) \ = \ e^{-2 \pi (|m| - \epsilon) v} \cdot O(1) 
	\end{equation}
for any fixed $\epsilon > 0$, where the implied constant depends only on $\epsilon$ and $f$. 

Finally, we define $\kappa_F(0)\in S(L)^{\vee}$ by  noting that \Cref{def:modgrowthforms}(iii) implies that $c_F(0,v)$ can be written in the form
\begin{equation}
  \label{eq:cF0}
  c_F(0,v) \  = \ \kappa_F(0) \ + \  \sum_{\mu}  \left[ \sum_{\substack{j \\ \beta_{j,\mu} \neq 1}} \frac{\alpha_{j,\mu}}{\beta_{j, \mu}-1} \, v^{\beta_{j,\mu}-1} \ + \ \gamma_{\mu}\log(v) \right] \varphi_{\mu}^{\vee} \ + \ O(e^{-cv})
\end{equation}
for some $\gamma_{\mu} \in \C$ .
In other words, $\kappa_F(0)$ is the constant part of the constant term.

Recall that we had defined the regularized pairing $\langle f, g \rangle^{\reg}$ for functions $f$ and $g$ transforming of weight $k$ and $-k$ respectively, as in \Cref{def:regpairing}. Our next goal is to show this pairing exists in two particular situations. 

\begin{lemma} \label{lem:regPet}
Suppose  $F \in \AkappaLexp(\rho_L^{\vee})$ and that $G \in M_{\kappa}(\rho_L^{\vee})$ is a holomorphic modular form. Then the \emph{regularized Petersson pairing}
$\langle F,\, G\rangle_{\mathrm{Pet}}^{\reg}$ (see \Cref{def:regpairing}) exists.
\begin{proof}  
For $\Re(s)$ sufficiently large, we write
	\begin{align*}
		\lim_{T \to \infty} \, \int_{\calF_T} \, F(\tau) \,&\overline{G(\tau)} \, v^{\kappa- s} \, d\mu(\tau) \\
				&= \ \int\limits_{\calF_1}  F(\tau) \, \overline{G(\tau)} \, v^{\kappa- s} \, d\mu(\tau) \ + \ \lim_{T \to \infty} \, \int\limits_{\calF_T - \calF_1} F(\tau) \, \overline{G(\tau)} \, v^{\kappa-s} \, d \mu(\tau) .
	\end{align*}
As $\calF_1$ is compact, the first integral defines a holomorphic function for $s \in \C$. As for the second integral, we have
	\begin{align*} 
		\lim_{T \to \infty} \, \int\limits_{\calF_T - \calF_1} F(\tau) \, \overline{G(\tau)} \, v^{\kappa-s} \, d \mu(\tau) \ =& \ \lim_{T \to \infty} \int_1^{T} \int_{-\frac12}^{\frac12} \, F(\tau) \, \overline{G(\tau)} \, v^{\kappa - s-2} \, du \, dv \\
		=& \ \lim_{T \to \infty} \ \int_1^{T} \ \sum_{m \geq 0} c_F(m, v) \, \overline{c_G(m)} \,  e^{-4 \pi m v} \, v^{\kappa-s-2} \, dv
	\end{align*}
where in the second line, we  inserted the Fourier expansions of $F$ and $G$, and took the integral over $u$.

For $v > 1$, use \eqref{eqn:cFlowering} to write
\[ 
	c_F(m,v) \ =\ c_F(m,1) \  + \ O(e^{2 \pi m v} v^{\ell}) \qquad \text{ as } v \to \infty
\]
for some $\ell$, where the implied constant is independent of $m$. Note also that
\[ \sum_{m > 0}| c_F(m,1) \, e^{- 2 \pi m v }|^2  \ < \  \sum_{m \in \Q} |c_F(m,1) \, e^{-2 \pi m  }|^2 \ = \ \int_{-1/2}^{1/2} |F(u + i)|^2 du 
\]
so $c_F(m,1) \cdot e^{- 2 \pi m v}$ is bounded by an overall constant independent of $v$ and $m$.

Using the estimate $c_G(m) = O(m^{\kappa-1})$ for a holomorphic modular form, it is easy to verify that 
	\[ 
		\lim_{T \to \infty} \ \int_1^{T} \ \sum_{m > 0} c_F(m, v) \, \overline{c_G(m)} \, e^{-4 \pi m v} \, v^{\kappa-s-2} \, dv	
	\]
converges uniformly for all $s \in \C$ and thereby defines a holomorphic function in $s$. 

For the zeroth Fourier coefficient, writing $c_F(0,v)$ as in \eqref{eq:cF0} we see that
	\[ \lim_{T \to \infty}  \ \int_1^{T} \, c_F(0,v) \ \overline{c_G(0)} \ v^{\kappa - s - 2} \ dv \]
extends to a meromorphic function in $s \in \C$, as required.
\end{proof}
\end{lemma}

\begin{remark}
The proof of the proposition implies that if $G \in S_{\kappa}(\rho_L^{\vee})$, then
\[ 
 \langle F, G \rangle_{\mathrm{Pet}}^{\reg} \ = \ \lim_{T \to \infty} \, \int_{\calF_T} F(\tau) \, v^{\kappa} \, \overline{ G(\tau)} \, d \mu(\tau). 
\]   
\end{remark}

Recall the fixed splitting
\begin{equation} \label{eqn:splittingRedux}
\begin{tikzcd}
	 			0 \ar{r} &  M^!_{\kappa}(\rho_L^{\vee}) / S_{\kappa}(\rho_L^{\vee}) \ar{r}{}{P} &  \mathsf{Sing}_{\kappa}(\rho_L^{\vee})  \ar{r} & M_{2 - \kappa}(\rho_L)^{\vee} \ar{r} \ar[bend right, above]{l}{}[swap]{\eta}  & 0.
		\end{tikzcd}
\end{equation} 
of the exact sequence in \Cref{thm:serreduality}; note that we have relabelled the subscripts here. Define an extension of the map $P$ to $\AkappaLexp(\rho_L^{\vee})$ by setting
\[ 
	P \colon \AkappaLexp(\rho_L^{\vee}) \ \to \ \mathsf{Sing}_{\kappa}(\rho_L^{\vee}), \ \qquad P(F) \ := \ \sum_{m \leq 0} \kappa_F(m) \cdot q^m .
\]
\begin{proposition}  \label{prop:uniquepreimage}
For any $f \in \ALmod{\kappa-2}(\rho_L^{\vee})$, there exists a unique $F \in \AkappaLexp(\rho_L^{\vee}) $ such that
	\begin{enumerate}
		\item $\Low(F) = f$;
		\item $P(F) \in \mathrm{im}(\eta)$; and
		\item $F$ has `trivial cuspidal holomorphic projection', i.e., $\langle F, G\rangle^{\reg}_{\mathrm{Pet}} = 0$ for every $G \in S_{\kappa}(\rho_L^{\vee})$.
	\end{enumerate}
We denote this unique preimage as $F = \Lsharp(f)$. 
\end{proposition}
Note that when $\kappa > 2$, the space $M_{2-\kappa}(\rho_L) = \{0\}$ and so $\eta$ is identically zero. In this case, condition $(ii)$ asserts that $P(\Lsharp(f)) = 0$, i.e.\ $\Lsharp(f)$ has ``trivial principal part".
\begin{proof}
	To prove the uniqueness statement, suppose $F_1$ and $F_2$ are as above, and set $F = F_1 - F_2$. Then $\Low(F)=0$, and so $F$ is a weakly holomorphic form, i.e.\  $F \in M_{\kappa}^!(\rho_L^{\vee})$. Moreover $P(F) \in \mathrm{im}(\eta)$, which implies that $P(F) = 0$ by the exactness of \eqref{eqn:splittingRedux}, and hence $F \in S_{\kappa}(\rho_L^{\vee})$ is a cusp form. But $F$ is orthogonal to cusp forms, so $F = 0$. 
	
	To show existence, we start by choosing any preimage $F_0$ of $f$, which exists by virtue of \Cref{prop:Lexactseq}. Appealing again to the exact sequence \eqref{eqn:splittingRedux}, there exists a form $G \in M^!_{\kappa}(\rho_L^{\vee})$ such that $P(F_0 - G) \in \mathrm{im}(\eta)$. Then, choosing an orthonormal basis $h_1, \dots, h_r \in S_{\kappa}(\rho_L^{\vee})$,  the function
	\[ 
		F(\tau)  \ := \ F_0(\tau) \ -  \ G(\tau) \ - \ \sum_i \langle F_0 - G, \ h_i\rangle^{\reg}_{\mathrm{Pet}}  \ h_i(\tau) \ \in \ \AkappaLexp(\rho_L^{\vee}) 
	\]
	satisfies the hypotheses in the Proposition.
\end{proof}

Our next aim to calculate regularized pairings against the harmonic Maa\ss{} forms of weight $ k  = 2 - \kappa$ introduced previously. 
Fix $\mu \in L'/L$ and $m \in Q(\mu) + \Z$. For convenience, write the Fourier expansion of the weight $k$ Poincar\'e series $F_{m, \mu}(\tau)$, defined for all $m$ and all weights as in \Cref{lem:uniqueHPSeries}, as
	\begin{equation} \label{eqn:FmFourierExp} 
		F_{m,\mu}(\tau) \ = \  \sum_{n} \, b_{m,\mu}(n,v) \, q^n.
	\end{equation}
Recall from \Cref{sec:relat-non-holom} that $b_{m,\mu}(n,v)$ is independent of $v$ when $n \geq 0$. 
 
\begin{proposition} \label{prop:poincarereg} Let  $f \in \ALmod{{\kappa-2} }(\rho_L^{\vee})$, and fix
  $\mu \in L'/L$ and  $m \in Q(\mu) + \Z$. Then
 \begin{align*}
     \langle  P_{m, w, \mu} - F_{m,\mu}, \, f\rangle^\reg \ &=  \ \lim_{T \to \infty}\Big[ \int_{\calF_T} \left(  P_{m, w, \mu}(\tau) -  F_{m,\mu}(\tau)\right)  f(\tau) \, d\mu(\tau) \\
      &\phantom{= \lim_{T\to\infty}\Big(}+  \int_{1}^T   \left( b_{m,\mu}(0)  - \delta_{m,0} \, \tilde\varphi_{\mu}  \right) \cdot c_f(0,v) \, v^{-2}\,dv \Big] \\
      &\phantom{= \lim_{T\to\infty}} - \CT_{s=0} \, \lim_{T \to \infty}  \int_{1}^T  \left( b_{m,\mu}(0) - \delta_{m,0} \,  \tilde\varphi_{\mu} \right) \cdot c_f(0,v) \, v^{-s-2}\,dv.
   \end{align*}
   Here, 
  \[
     \tilde\varphi_{\mu} = \frac12 ( \varphi_{\mu}+ (-1)^{\kappa + (q-p)/2} \varphi_{-\mu}).
  \]
   
In particular, the regularized integral exists. 

\begin{proof}
 Recalling the definition of  the truncated Poincar\'e series
 		\[ P_{m, w, \mu}(\tau)  \ := \ \frac{1}{4}  \,
 		                                    \sum_{ \tilde \gamma \in \Gt_\infty \bs \Gt} \left( \sigma_{w}(\tau) \, e^{-2 \pi i m \tau }\,  \varphi_{\mu} \right) \big\vert_{k} [\tilde \gamma], \]
 	   note that if $\Im(\tau) > w_0 := \max(w, 1/w)$, then
 		\[ P_{m, w, \mu}(\tau) \ =  \ e^{-2 \pi i m \tau} \, \tilde\varphi_\mu, \]
 		
 		 In particular, if $v = \Im(\tau) > w_0$ then
 		\begin{equation} \label{eqn:QExpDiffPoincare} 
 			P_{m, w, \mu}(\tau) \  -  F_{m,\mu}(\tau) \ + \  b_{m,\mu}(0)\  - \  \delta_{m,0}\, \tilde\varphi_\mu \ =  \ O(e^{-Cv})
 		\end{equation}
 		for some $C>0$; indeed, when $\kappa>2$ and $m >0$, this bound follows from \eqref{eqn:WHMPrincAsymp}, since $	P_{m, w, \mu}(\tau) +  b_{m,\mu}(0)$ is precisely the principal part of $F_{m,\mu}(\tau)$, and the bound is trivial to verify in all other cases. For convenience, set
 		\[ \tilde b_{m,\mu}(0) \ := \ b_{m,\mu}(0) \ - \ \delta_{m,0}\, \tilde\varphi_\mu. \]
 		
 For $\Re(s)$ sufficiently large, 
 		 \begin{align} 
 			  \lim_{T \to \infty} \int_{\calF_T} & (P_{m, w, \mu}(\tau) - F_{m,\mu}(\tau))\,  f(\tau) \ v^{-s} \ d\mu(\tau)  \notag \\ 
 			  &=  \ \int_{\calF_{w_0}} ( P_{m, w, \mu}(\tau) - F_{m,\mu}(\tau))  \, f(\tau) \ v^{-s} \ d\mu(\tau) \notag \\ 
 			& \qquad   \qquad \ + \lim_{T \to \infty}  \int\limits_{\calF_T  - \calF_{w_0}}  ( P_{m, w, \mu}(\tau) - F_{m,\mu}(\tau)) f(\tau) \ v^{-s} d \mu(\tau). \notag
 			 \end{align}
 As before, the integral over $\calF_{w_0}$ is holomorphic in $s$, and contributes its value 
 \[ I_0 := \int_{\calF_{w_0}} ( P_{m, w, \mu}(\tau) - F_{m,\mu}(\tau))  \, f(\tau)  \ d\mu(\tau) \]
 at $s=0$ to the regularized integral. For the other integral, inserting the Fourier expansion \eqref{eqn:QExpDiffPoincare} and carrying out the integral over $u=\Re(\tau)$ gives
 	\begin{align*}
 		\lim_{T \to \infty}  \int\limits_{\calF_T  - \calF_{w_0}}  ( P_{m, w, \mu}(\tau) - F_{m,\mu}(\tau)) f(\tau) \ v^{-s} d \mu(\tau) \ =  I_1(s) \ + \ I_2(s) 
 	\end{align*}
 	where 
\[	I_1(s) \ = \ \lim_{T \to \infty} \int\limits_{\calF_T  - \calF_{w_0}}  \left( P_{m, w, \mu}(\tau) - F_{m,\mu}(\tau) + \tilde b_{m,\mu}(0) \right) \cdot f(\tau) \ v^{-s} d \mu(\tau) \]
 	 and
  	\[ I_2(s) \ = \ - \lim_{T \to \infty} \int\limits_{\calF_T  - \calF_{w_0}} \tilde  b_{m,\mu}(0) \cdot f(\tau) v^{-s} d \mu(\tau) \ = -  \lim_{T \to \infty} \ \int_{w_0}^{T}  \ \tilde  b_{m,\mu}(0) \, c_f(0,v) \, v^{-s-2} dv.  \]
 It follows from \eqref{eqn:QExpDiffPoincare} that $I_1(s)$ defines a holomorphic function for $s \in \C$, and so contributes its value at $s=0$ to the regularized integral. On the other hand, the assumption in \Cref{def:modgrowthforms} for the shape of $c_f(0,v)$ immediately implies that $I_2(s)$ admits a meromorphic continuation to $s \in \C$. Therefore, the regularized pairing exists:
	\begin{equation}	\label{eqn:regPoincareIntFinal}
		\langle P_{m, w, \mu} - F_{m,\mu}, \ f \rangle^{\reg} \ = \ I_0 \ + \ I_1(0) \ + \ \CT_{s=0} I_2(s). 
	\end{equation}

To conclude the proof, write  
	\begin{align*}   I_1(0) \ 
	&= \ \lim_{T \to \infty}  \ \int\limits_{\calF_T - \calF_{w_0}}  \left(  P_{m, w, \mu}(\tau) - F_{m,\mu}(\tau) \right) \, f(\tau) \ d\mu(\tau) \ + \ \int_{w_0}^T  \tilde b_{m,\mu}(0) \, c_f(0,v) \, v^{-2} \, dv,
	\end{align*}
and combining the first integral with $I_0$ gives 
\[ I_0 \ + \ I_1(0) \ = \ \lim_{T \to \infty}  \ \int\limits_{\calF_T}  \left(  P_{m, w, \mu}(\tau) - F_{m,\mu}(\tau) \right) \, f(\tau) \ d\mu(\tau) \ + \ \int_{w_0}^T \tilde b_{m,\mu}(0) \, c_f(0,v) \, v^{-2} \, dv \]
Furthermore, observe that 
\[ \int_{w_0}^1 \tilde b_{m,\mu}(0) \, c_f(0,v) \, v^{-2} \, dv \ = \ \CT_{s=0} \int_{w_0}^1 \tilde b_{m,\mu}(0) \, c_f(0,v) \, v^{-s-2} \, dv ;\]
adding and subtracting this quantity in \eqref{eqn:regPoincareIntFinal} gives the expression in the proposition.
  	\end{proof}
\end{proposition}

We arrive at the \emph{raison d'\^etre} of this section: 

\begin{theorem} \label{thm:MaassSection} 
 Suppose  $f \in \ALmod{\kappa-2}(\rho_L^{\vee})$, and 
	\[F \ = \  \Lsharp(f)  \ \in \  \AkappaLexp(\rho_L^{\vee}) \] 
the preimage under $L$ as specified in \Cref{prop:uniquepreimage}.

Then for any $\mu \in L'/L$ and $m \in \Q$,
\begin{equation} \label{eqn:PRegmPos} \langle P_{m, w, \mu}-  F_{m,\mu} , \  f \rangle^{\reg} \ = \  -  c_F(m, w)(\varphi_{\mu})  \end{equation}

\begin{proof} 
If $m \notin Q(\mu) + \Z$, then both sides of \eqref{eqn:PRegmPos} are easily seen to vanish; we assume that $m \in Q(\mu) + \Z$ from now on.

We  proceed by examining the various pieces appearing in \Cref{prop:poincarereg}. Starting with the truncated Poincar\'e series
	\[ P_{m, w, \mu}(\tau)  \ := \ \frac{1}{4}  \,
	 		                                    \sum_{ \tilde \gamma \in \Gt_\infty \bs \Gt} \left( \sigma_{w}(\tau) \, e^{-2 \pi i m \tau }\,  \varphi_{\mu} \right) \big\vert_{k} [\tilde \gamma], \]
	note that for any fixed $T$, the restriction of $P_{m, w, \mu}(\tau)$ to the compact set $\calF_T$ is a finite sum, and so
	\begin{align*}
		\int_{\calF_T} \, P_{m, w, \mu}(\tau) \, f(\tau) \, d \mu(\tau) \ =& \ \frac14 \,	\int_{\calF_T}
			 		                                    \sum_{ \tilde \gamma \in \Gt_\infty \bs \Gt} \left( \sigma_{w}(\tau) \, e^{-2 \pi i m \tau }\,  \varphi_{\mu} \right) \big\vert_{k} [\tilde \gamma] \, f(\tau) \ d \mu(\tau)  \\
	 		                 =& \ \frac14 \,  \sum_{ \tilde \gamma \in \Gt_\infty \bs \Gt}  \int_{\gamma \cdot \calF_T} \,  \sigma_{w}(\tau) \, e^{-2 \pi i m \tau } \, f(\tau) (\varphi_{\mu}) \, d \mu(\tau). \\
	 		                 =& \ \int_{\calF_T}  \sigma_{w}(\tau) \, e^{-2 \pi i m \tau } \, f(\tau)(\varphi_{\mu}) \, d \mu(\tau) \\
	 		                 &\qquad \qquad + \ \frac14  \sum_{ \substack{ \tilde \gamma \in \Gt_\infty \bs \Gt \\ \tilde\gamma  \notin A}}  \int_{\gamma \cdot \calF_T} \,  \sigma_{w}(\tau) \, e^{-2 \pi i m \tau } \, f(\tau) (\varphi_{\mu}) \, d \mu(\tau),
	\end{align*}
	where in the last line, we separate out the contributions from the subgroup 
	\[ \calA :=  \{ (\mathtt{Id}, 1) , \ (\mathtt{Id}, -1),  \, (-\mathtt{Id}, i), \, ( - \mathtt{Id}, -i) \},  \]
	which acts trivially on $\uhp$, from the rest. 
	
	For any coset $\tilde\gamma \in \Gt_\infty \bs \Gt$ with $\tilde \gamma \notin \calA$ and $\tau \in \uhp$,
	\[ \Im(\gamma \tau) \ = \ \frac{\Im(\tau)}{|c\tau + d|^2} \leq \frac{1}{\Im(\tau)}.\]
	It follows that when $T>\max(w,1/w)$, the intersection of the translates of $\calF_T$ with the region where the cutoff function $\sigma_{w}(\tau)$ is nonzero fills out a rectangle:
        \begin{equation*}
          \left(\bigcup_{\gamma \in \Gt_\infty \bs \Gt} \gamma(\calF_T) \right) \ \bigcap \ \{ \Im(\tau) \geq w \} \ = \ \{ \Re(\tau) \in \left[ -1/2 , 1/2 \right], \ \Im(\tau) \in [w,T] \} \ =: \calR_w^T.
        \end{equation*}
	Therefore, for all $T>\max(w,1/w)$, 
	\begin{equation}  \label{eqn:intPFT}
		 \int_{\calF_T} \, P_{m, w, \mu}(\tau) \, f(\tau) \, d \mu(\tau)  \ =  \ \int_{\calR_w^T} e^{-2 \pi i m \tau} \, f(\tau)(\varphi_{\mu}) \, d\mu(\tau).
	\end{equation}
We calculate the right hand side using a standard Stokes' theorem argument: since $\Low(F) = f$, 
	\begin{align*} d( e^{-2 \pi i m \tau}  F \, d\tau) \ = \ -  e^{-2 \pi i m \tau} \cdot \frac{\partial}{\partial \overline \tau} \, F \, d \tau \wedge d \overline\tau \ =& \ -   e^{-2 \pi i m \tau} \, \Low(F) \, d \mu(\tau) \\
	 =& \  -   e^{-2 \pi i m \tau} \, f \, d \mu(\tau).  
	\end{align*} 
	Insert this expression into \eqref{eqn:intPFT} and use Stokes' theorem to integrate over the boundary $\partial \calR_w^T$: the `vertical' segments cancel on account of the invariance of the integrand under $\tau \mapsto \tau+1$, while the `horizontal' segments give the Fourier coefficients of $F$, i.e.
	\begin{align*}
	 \int_{\calF_T} \, P_{m, w, \mu}(\tau) \, f(\tau) \, d \mu(\tau) \ =& \ -   \int_{\partial \calR_w^T} e^{-2 \pi i m \tau} \,   F(\tau)(\varphi_{\mu})  \, d\tau   \\
		=& \ c_F(m,T)(\varphi_{\mu}) \ - \ c_F(m,w) (\varphi_{\mu}) . \label{eqn:TPunfoldfinal}
	\end{align*}
Similarly,
 	\begin{align}
	 	\int_{\calF_T} F_{m,\mu}(\tau) \, f(\tau) \, d \mu(\tau) \ =& \ -  \int_{\calF_T}   F_{m,\mu}(\tau) \, d\left( F(\tau) \,d \tau \right)  \notag \\
	 		=&  \ \int_{\calF_T} \, F \, d \left( F_{m,\mu}(\tau) \, d \tau \right) \ -  \int_{\partial \calF_T} F_{m, \mu}(\tau) \, F(\tau) \, d \tau.
 	\end{align}
We show the limit as $T \to \infty$ of the first integral vanishes: since $F_{m,\mu} \in H_{2 - \kappa}(\rho_L)$, it follows that $\xi(F_{m,\mu}) \in S_{\kappa}(\rho_L^{\vee})$, cf.\ \Cref{sec:relat-non-holom}.
Therefore 
	\begin{align*} \lim_{T \to \infty}  \int_{\calF_T} \, F \, d \left( F_{m,\mu}(\tau) \, d \tau \right) \ =& \ \lim_{T \to \infty} \int_{\calF_T} F \,  \overline{\xi(F_{m,\mu})}  \, v^{\kappa} \, d \mu(\tau) \ 	= \ \langle F, \, \xi(F_{m,\mu}) \rangle_{\mathrm{Pet}}^{\reg}  \ = \ 0
	\end{align*}
since $F$ was taken to be orthogonal to cusp forms. 

For the second integral, the modularity of the integrand implies that only the upper line segment of $\partial \calF_{T}$  contributes; 
inserting the Fourier expansions of $F$ and of  $F_{m,\mu}$, as in \eqref{eqn:FmFourierExp}, yields
\[
	 -  \int_{\partial \calF_T} F_{m, \mu}(\tau) \, F(\tau) \, d \tau \  = \ \sum_{n \in \Q} b_{m,\mu}(n,T) \cdot c_F(-n, T) ,
\]
which, by \Cref{lem:sumneq0coeffs} below, satisfies the asymptotic
\[
	\sum_{n} b_{m,\mu}(n,T) \cdot c_F(-n, T)   \ = \  \sum_{n \leq 0} b^+_{m,\mu}(n) \cdot c_F(-n,T) \ +  \ \sum_{n>0} b^+_{m,\mu}(n) \cdot \kappa_F(-n) \ + \ o(1) 
\]
as $T \to \infty$.

Putting these pieces back into \Cref{prop:poincarereg} gives
\begin{align*}
 \langle P_{m, w, \mu}-  F_{m,\mu} ,  \  f \rangle^{\reg} \ &= - c_F(m,w)(\varphi_{\mu}) \   - \ \sum_{n>0} b^+_{m,\mu}(n) \cdot \kappa_{F}(-n)  \\ 
 									&\phantom{=} + \lim_{T \to \infty} \bigg[ c_F(m,T)(\varphi_{\mu}) - \sum_{n\leq 0} b^+_{m,\mu}(n) \cdot c_F(-n,T) \\
 									&\phantom{= + \lim_{T \to \infty}} + \tilde{b}_{m,\mu}(0) \cdot  \left(  \ \int_1^T c_f(0,v) \, v^{-2} dv \right) \bigg] \\
 									&\phantom{=} -  \CT_{s=0} \, \tilde b_{m,\mu}(0) \int_1^{\infty}  \,  c_f(0,v) \, v^{-2-s}\, dv,
\end{align*}
where 
\begin{align*} 
	\tilde b_{m,\mu}(0)  \cdot c_f(0,v) \ =& \ \left(  b_{m,\mu}(0) \ - \ \delta_{m,0}\, \tilde\varphi_\mu \right) \cdot c_f(0,v) \\
        =& \ b_{m,\mu}(0) \cdot c_f(0,v) \ - \ \delta_{m,0} \, c_f(0,v)(\varphi_{\mu}).
\end{align*}
Furthermore, one may easily verify, via \eqref{eqn:cFlowering} and \Cref{def:modgrowthforms}(iii), that the relation $\Low(F) = f$ implies
\begin{align*} c_F(0,T) \ - \ \int_1^T c_f(0,v) \, v^{-2} dv & \ + \   \CT_{s=0} \int_1^{\infty}  \,  c_f(0,v) \, v^{-2-s} \ dv \ \\
 =&  \ c_F(0,1) \ + \ \CT_{s=0} \int_1^{\infty}  \,  c_f(0,v) \, v^{-2-s} \ dv \\ =& \ \kappa_F(0)
\end{align*}
and so 
 \begin{align}
  \langle P_{m, w, \mu} - F_{m,\mu} , \  f \rangle^{\reg} \ =   &- c_F(m,w)(\varphi_{\mu}) - \sum_{n \geq 0} b^+_{m,\mu}(n)  \cdot \kappa_F(-n) \notag \\
  &+  \lim_{T \to \infty}  \bigg[ c_F(m,T)(\varphi_{\mu}) - \sum_{n<0} b^+_{m,\mu}(n) \cdot c_F(-n,T)  \label{eqn:LSharpFourierComp1}\\
  &\phantom{+  \lim_{T \to \infty}  \bigg[} -  \delta_{m,0}\, (c_F(m,T)(\varphi_{\mu}) - \kappa_F(0)(\varphi_{\mu}))\bigg]. \notag
  \end{align}
Recall that the choice of $F$ in \Cref{prop:uniquepreimage} required 
\[ 
  P(F) = \sum_{n\leq 0} \kappa_F(n) q^n  \in \mathrm{im}(\eta) 
\]
where $\eta$ was a choice of a section in \eqref{eqn:etaSplitting}. Thus the normalization imposed on $F_{m,\mu}$ in \Cref{lem:uniqueHPSeries} (ii) implies that
\[
	\sum_{n \geq 0} b^+_{m,\mu}(n) \cdot \kappa_F(-n) \ = \ \kappa_F(m)(\varphi_{\mu});
\] 
note that this is simply $0$ if $m> 0$.  

On the othe hand, the shape of the principal part of $F_{m,\mu}$ imposed in  \Cref{lem:uniqueHPSeries}(i)
and the invariance of  $F$  under the action of $\mathbf Z \in \Gt$ imply that
\[ 
	\sum_{n<0}  \, b^+_{m,\mu}(n) \, c_F(-n,T) \ = \ 
        \begin{cases} 
          c_F(m,T)(\varphi_{\mu} ), & \text{if } m > 0, \\  
          0, & \text{if } m \leq 0.
        \end{cases}
\]
Therefore, when $m>0$
\[
 \lim_{T \to \infty} \left[ c_F(m,T)(\varphi_{\mu}) - \sum_{n<0} b^+_{m,\mu}(n) \cdot c_F(-n,T) \right]   \ = \ \lim_{T \to \infty} \left[  c_F(m,T)(\varphi_{\mu}) -  c_F(m,T)(\varphi_{\mu}) \right] \ = \ 0  
\]
 and when $m \leq 0$, the limit in \eqref{eqn:LSharpFourierComp1} equals
\begin{align*}
    \lim_{T \to \infty} \Big[ c_F(m,T)(\varphi_{\mu}) - \delta_{m,0}\, (c_F(m,T)(\varphi_{\mu}) - \kappa_F(0)(\varphi_{\mu})) \Big] &=
    \lim_{T \to \infty} 
    \begin{cases} 
      c_F(m,T)(\varphi_{\mu}), & \text{if }  m < 0, \\ 
      \kappa_F(0)(\varphi_{\mu}) , & \text{ if } m=0
    \end{cases} \\ 
 &= \kappa_F(m)(\varphi_{\mu});
\end{align*}
Thus, for all $m$, the limit in \eqref{eqn:LSharpFourierComp1} is equal to
\[
  \sum_{n \geq 0} b^+_{m,\mu}(n) \cdot \kappa_F(-n)
\]
and so
\[
  \langle P_{m, w, \mu} - F_{m,\mu} , \  f \rangle^{\reg} = - c_F(m,w)(\varphi_{\mu}),
\]
concluding the proof.
\end{proof}
 \end{theorem} 
Extending the theorem to the $S(L)^{\vee} \otimes_{\C} S(L)$-valued Poincar\'e series $P_{m,v}$ and $F_m$ yields:
\begin{corollary}  \label{cor:vectorValLSec}
For any $f \in \ALmod{\kappa - 2}(\rho_L^{\vee})$, the generating series
	\[ 
	 \sum_{  m \in\Q} \, \langle P_{m,v} - F_m , \, f \rangle^{\reg} \, q^m  \]
	with $\tau = u+iv$ and $q = e^{2 \pi i \tau}$, 
	is (the $q$-expansion of) an element of $\AkappaLexp(\rho_L^{\vee})$ . 
\end{corollary}

It remains to prove the following lemma:

\begin{lemma}  \label{lem:sumneq0coeffs} Suppose
 \[G(\tau) = \sum_{m} c_G(m,v) \, q^m \in H_{2-\kappa}(\rho_L) \qquad \text{and} \qquad F(\tau) = \sum_m c_F(m,v) \, q^m \in \AkappaLexp(\rho_L^{\vee}) . \]
 Then 
 \[ \sum_{n\in \Q} c_G(n,v) \cdot c_F(-n,v) \ = \ \sum_{n \leq 0} c^+_G(n) \cdot c_F(-n,v) \ +  \ \sum_{n>0} c^+_G(n) \cdot \kappa_F(-n) \ + \ o(1) \qquad \text{ as } v \to \infty, \]
 where $G^+(\tau) = \sum c^+_G(n) q^n$ is the holomorphic part of $G$, cf.\ \Cref{sec:relat-non-holom}. Note that both sums appearing are finite sums.
\begin{proof}
We consider separately the terms $n> 0$ and $n \leq 0$ in the sum on the left hand side. 

 Write $c_G(n,v) = c_G^+(n) + c_G^-(n,v)$ in terms of the decomposition $G= G^+ + G^-$. Recall, cf.\ \eqref{eqn:weakMaassDecomp}, that
\[ c^-_G(n,v) \ = \ 0 \qquad \text{ for } n \geq 0. \]
On one hand, we have the asymptotic $|c_G^+(n)| = O(e^{C \sqrt{n}})$ for $n > 0$, cf.\ \cite[Lemma 3.4]{brfugeom}. On the other hand, for $n > 0$, we have by \Cref{eqn:cFlowering} and the moderate growth of $\Low(F)$ that
\[ 
  c_F(-n, v) \ - \ \kappa_F(-n) \ = \ O(e^{-2 \pi n v } v^{\ell} ). 
\]
for some $\ell$.
Thus
\[
	\sum_{n> 0} c_G(n,v) \cdot c_F(-n,v) \ =\ \sum_{n \geq 0} c^+_G(n,v) \cdot \kappa_F(-n) \ + \ o(1) \qquad \text{as } v \to \infty.
\]

Turning to the non-positive terms, we consider the contributions from  $c^-_G(n,v)$, and write
	\begin{align} 
	\left| \sum_{n\leq 0} c^-_G(n,v) \cdot c_{F}(-n,v) \right|^2  \ =& \  \left| \sum_{n\leq 0}  \left(  e^{-2\pi n v } \,  c^-_G(n,v) \right) \cdot \left(e^{2 \pi n v} \,  c_{F}(-n,v)  \right) \right|^2   \notag \\  
	\leq& \ \left( \sum_{n\leq 0} \left|  c^-_G(n,v) \, e^{-2 \pi n v} \right|^2\right) \cdot \left( \sum_{n \geq 0}  \left| c_{F}(n,v) \, e^{-2 \pi n v} \right|^2\right).    \label{eqn:negTermsParseval}
	\end{align}
Note that the exponential growth condition implies that
\[ 
G(\tau) \ - \ \sum_{n \leq 0} c_G^+(n) \, q^n \ = \ \sum_{n \leq 0} c^-_G(n) \, q^n \ + \ \sum_{n>0} c_G(n) \, q^n \ = \ O(e^{-Cv})
\]
so, by Parseval's identity,
\begin{align*}
 \sum_{n\leq 0} \left|  c^-_G(n,v) \, e^{-2 \pi n v} \right|^2 \ \leq& \  \sum_{n\leq 0} \left|  c^-_G(n,v) \, e^{-2 \pi n v} \right|^2 + \sum_{n>0} \left|  c_G(n,v) \, e^{-2 \pi n v} \right|^2 \\
 =&  \ \int_{-1/2}^{1/2} \left| G(u+iv) - \ \sum_{n \leq 0} c_G^+(n) e^{- 2\pi n v} e^{2 \pi i u} \right|^2 \ du \ = \ O(e^{-Cv}).
\end{align*}
For the second sum in \eqref{eqn:negTermsParseval}, recall
	\[ c_F(n,T) \ = \ c_F(n,1) \ + \ \int_1^T c_f(n,v) \, \frac{dv}{v^2} \]
	where $f = \Low(F)$ has at worst polynomial growth at $\infty$. Observe that for any $T>1$
	\begin{align*}
		\sum_{n \geq 0}  \left| c_{F}(n,1) \, e^{-2 \pi n T} \right|^2  \ \leq \  \sum_{n\geq0 }  \left| c_{F}(n,1) \, e^{-2 \pi n } \right|^2  \	\leq  \ \int_{-1/2}^{1/2} \, \left| F(u+i) \right|^2 \, du \ = \ O(1)
	\end{align*}
	by Parseval's identity, and
	\begin{align*}
			\sum_{n > 0}  \left| \int_1^T c_{f}(n,v) \frac{dv}{v^2}  \cdot  e^{-2 \pi n T} \right|^2  \ \leq& \ \sum_{n > 0}  \left| \int_1^T c_{f}(n,v) \, e^{- 2 \pi n v}  \frac{dv}{v^2} \right|^2 \\
			\leq& \   \int_1^T \,  \int_{-1/2}^{1/2} \, \left| f(u+iv) \right|^2  v^{-2}\, du \, dv  \ = \ O(T^{\ell})
	\end{align*}
	for some $\ell$.
Thus 
\[ \sum_{n\leq 0} c^-_G(n,v) \cdot c_{F}(-n,v) \ = \ o(1) ,\]
proving the lemma.
\end{proof}

 \end{lemma}

\section{Green functions and Archimedean generating series} \label{sec:arch}

In this section, we specialize to the case where $V$ is a quadratic space of signature $(p,2)$. Fix an even integral lattice $L \subset V$ and let $L'$ denote the dual lattice.

Attached to $V$ is the \emph{symmetric space}
	\[ \domain^o(V) := \{ z \subset V(\R) \text{ an oriented negative definite plane}. \} \]
which is a model for the Hermitian symmetric domain attached to $SO(V)$; we use the superscript `$o$' to emphasize the connection to the orthogonal group. 
The map
	\begin{align*}
	 \{ \zeta \in \mathbb P(V(\C)) \ | \ (\zeta, \zeta) = 0, \, (\zeta, \overline \zeta) < 0 \}  \ \isomto& \ \domain^o(V)  \\
	   \zeta \ \mapsto& \ z = \mathtt{span}_{\R}( \Re(\zeta), \, \Im(\zeta)) 
	 \end{align*}
identifies $\domain^o(V)$ as a submanifold of $\mathbb P(V(\C))$, and hence $\domain^o(V)$ acquires the structure of a complex manifold. 

For a plane $z \in \domain^o(V)$ and a vector $x \in V(\R)$, we define 
\[ 
  R^o(x, z) \ = \  - 2  \, Q( \proj_z(x) ) \  \geq 0 \ ,
\]
where $\proj_z(x)$ is the orthogonal projection of $x$ onto $z$; note that the quadratic form
\[ 
  Q_z(x) \ := \  R^o(x,z) + Q(x)  
\]
is positive definite.  When $x$ has positive norm, let $Z(x)$ denote the complex codimension 1 submanifold
\[ 
  Z(x) := \{ z \in \domain^o(V) \ |  \ z \perp x \} \ = \ \{ z \in \domain^o(V) \ | \ R^o(x,z) = 0 \}. 
\]
For convenience, if the norm of $x$ is non-positive, set $Z(x) = \emptyset$.

If $m \in \Q$ and $\varphi \in S(L)$,  define the \emph{special cycle} $Z(m , \varphi)$ to be the formal sum 
\[ 
  Z(m, \varphi) \ := \  \sum_{\substack{x \in L' \\ Q(x) = m }} \, \varphi(x) \, Z(x). 
\]
This sum is locally finite, in the sense that only finitely many $Z(x)$'s appearing in the sum will intersect a given compact subset of $\domain^o(V)$. For an appropriate arithmetic subgroup $\Gamma \subset O(V)$, the quotient $[ \Gamma \bs Z(m, \varphi)]$ defines a (rational) algebraic cycle on the Shimura variety $[\Gamma \bs \domain^o(V)]$, a point of view that we will take up in the next section. 

We say that a current $[\lie{g}] \in \D^{0,0}(\domain^o(V))$   is  a \emph{Green function} for the cycle $Z(m,\varphi)$ if
\[ 
  \ddc \, [\lie g] \ + \ \delta_{Z(m,\varphi)}  
\]
is represented by a smooth $(1,1)$-form.

\subsection{Kudla's Green function as a regularized theta lift}
\label{sec:kg-theta}
Let 
	\[ \beta(r) \ := \ \int_r^{\infty} \, e^{-t} \frac {dt}{t} \ = \ \int_1^{\infty} e^{-rt} \, \frac{dt}{t} , \qquad \qquad \text{for } r>0, \]
and note that $\beta(r) + \log(r) = O(1)$ as $r \to 0$. 
\begin{definition} 
       For $m \in \Q$ with $m \neq 0$, $\varphi \in S(L)$ and a real parameter $w \in \R_{>0}$,  define  \emph{Kudla's Green function} as
	\[ 
          \KGrO(m, w, \varphi)(z) \ := \ \sum_{ \substack{x \in L' \\ Q(x) = m }} \, \varphi(x) \,  \ \beta \big( 2 \pi  w   R^o (x, z)\big), \qquad \text{ for } z \in \domain^o(V) \setminus Z(m, \varphi) 
        \]
	which is a Green function for $Z(m, \varphi)$, cf.\ \cite{kudla-annals-central-der,kudla-integrals}; when $m<0$, this means that $\KGrO(m, w, \varphi)$ is smooth. 
	We also define
	\[ 
           \KGrO(0, w, \varphi)(z) \ := \ \sum_{ \substack{x \in L' \\ Q(x) = 0 \\ x \neq 0}} \, \varphi(x) \,  \ \beta \big( 2 \pi  w   R^o (x, z)\big). 
       \]
\end{definition}

For convenience,  set 
\[ 
  \KGrO(m, w, \mu) \ := \ \KGrO(m,w,\varphi_{\mu}) \qquad \text{and} \qquad Z(m,\mu)  \ : = \ Z(m, \varphi_{\mu}) 
\]
for any $\mu \in L'/L$.
Our first aim is another construction of this Green function in terms of the \emph{Siegel theta function} 
	\[ 
          \Theta_L \colon \uhp \ \times \  \domain^o(V) \ \to \ S(L)^{\vee} ,
       \]
which is defined by the formula
\begin{align}
  \label{eq:siegeltheta}
  \Theta_L(\tau,z)(\varphi) \ :=&  \  v \, \sum_{\lambda \in L'}  \, \varphi(\lambda)\, e^{2 \pi i  \, \left(Q(\proj_{z^{\perp}}(\lambda)) \, \tau \ + \  Q( \proj_z(\lambda))\, \bar{\tau} \right) }  \\
   =& \ v \, \sum_{m \in \Q} \, \left( \sum_{\substack{\lambda \in L' \\ Q(\lambda) = m}}   \, \varphi(\lambda) \, e^{- 2 \pi v R^o(\lambda, z)}  \right) \, q^m. \notag
\end{align}
For a fixed $z$, it is a standard fact that $\Theta_L(\tau,z)$ transforms as a modular form of weight $p/2-1$ with respect to the Weil representation, and straightforward estimates imply that $\Theta_L(\tau, z)$ is $O(v)$ as $v \to \infty$ and satisfies \Cref{def:modgrowthforms}(iii), which implies $\Theta(\tau,z) \in \ALmod{p/2-1}(\rho_L)$. 

For $\mu \in L'/L$, $m \in Q(\mu) + \Z$, and $w \in \R_{>0}$, consider the regularized pairing  
\begin{equation}
  \label{eq:regtheta-jump}
	\langle P_{m, w, \mu} , \ \Theta_L(\cdot, z) \rangle^{\reg} \  = \  \CT_{s=0} \left( \lim_{T \to \infty}\int_{\calF_T} \,  P_{m, w, \mu}(\tau) \, \Theta_L(\tau,z) \,   v^{-s}\, d\mu(\tau)\right),
\end{equation}
where $P_{m, w, \mu}$ is the truncated Poincar\'{e} series of weight $k = 1-p/2$ and we view \eqref{eq:regtheta-jump} as a function in $z \in \domain^o(V)$.

\begin{lemma}
 \label{lem:xireg}
  The regularized integral \eqref{eq:regtheta-jump} exists for each $ z \in \domain^o(V)$ and
  \begin{equation}
    \label{eq:Xireg}
     \langle P_{m, w, \mu} , \ \Theta_L(\cdot, z) \rangle^{\reg} = \lim_{T \to \infty} \left(\int_{\calF_T} P_{m, w, \mu}(\tau) \, \Theta_L(\tau,z) \, d\mu(\tau)  \ -  \ S_{m,\mu}(z) \, \log(T)\right),
  \end{equation}
where $S_{m,\mu}(z) = \#\{ \lambda \in \mu +L \mid\ Q(\lambda) = m \text{ and } R^o(\lambda,z) = 0\}$.
\begin{proof}
	Let $w_0 = \max(w, 1/w)$ so that 
		\[ 
                    P_{m, w, \mu}(\tau) = q^{-m} \, \tilde\varphi_{\mu} = \frac{1}{2} q^{-m} \, (\varphi_{\mu} +\varphi_{-\mu}) 
                \]
	whenever $\Im(\tau) > w_0$. 
        Recall we assume that the signature of $V$ is $(p,2)$ and that $k = p/2 - 1$, so $\tilde\varphi = \frac12(\varphi_{\mu} + \varphi_{-\mu})$, cf.\ \eqref{eq:tildephi}.
	
	Thus, abbreviating $P = P_{m, w, \mu}$, 
	\begin{align*}
		\CT_{s=0} \,  \lim_{T \to \infty} \,& \int_{\calF_T} \, P(\tau) \,   \Theta_L(\tau,z)  \, v^{-s} \, d \mu(\tau) \\ 
		&= \ 	\CT_{s=0} \, \int_{\calF_{w_0}} \, P(\tau) \,  \Theta_L(\tau,z)  \, v^{-s} \, d \mu(\tau) \ + \ 	\CT_{s=0} \,  \lim_{T \to \infty} \, \int_{\calF_T - \calF_{w_0}}  P(\tau) \,  \Theta_L(\tau,z)  \, v^{-s} \, d \mu(\tau)
	\end{align*}
	
	The first integral defines a holomorphic function in $\C$, and contributes its value at $s=0$ to the regularized integral. For the second integral, observe that $\Theta_L(\tau, z)(\varphi_{\mu}) = \Theta_L(\tau,z)(\varphi_{-\mu})$ and substitute the Fourier expansion of $\Theta(\tau, z)(\varphi_{\mu})$ to obtain
	\begin{align*} 
			\CT_{s=0} \,  \lim_{T \to \infty} \, \int_{\calF_T - \calF_{w_0}} P(\tau)\, & \Theta_L(\tau,z)  \, v^{-s} \, d \mu(\tau) \\ 
			=&  \ \CT_{s=0} \,  \lim_{T \to \infty} \int_{w_0}^T \, \sum_{\substack{ \lambda  \in \mu + L \\ Q(\lambda) = m }} e^{-2 \pi v R^o(\lambda,z)}  \, v^{-s-1} \, d v \\
			=&  \ \CT_{s=0} \,  \lim_{T \to \infty} \int_{w_0}^T \, \left( \sum_{\substack{ \lambda  \in \mu + L \\ Q(\lambda) = m \\ R^o(\lambda,z) \neq 0 }} e^{-2 \pi v R^o(\lambda,z)}  + S_{m,\mu}(z) \right) \, v^{-s-1} \, d v 
	\end{align*}
	
	The sum over $\{R^o(\lambda,z) \neq 0\}$ is $O(e^{-Cv})$ and so the corresponding integral is holomorphic in $s$ for all $s \in \C$; it therefore contributes its value at $s=0$ to the regularized integral. Thus
		\begin{align} 
				\CT_{s=0} \,  \lim_{T \to \infty} \,  & \int_{\calF_T - \calF_{w_0}} P(\tau)\,  \Theta_L(\tau,z)  \, v^{-s} \, d \mu(\tau)  \notag \\ 
				=&  \  \lim_{T \to \infty} \int_{w_0}^T \, \sum_{\substack{ \lambda  \in \mu + L \\ Q(\lambda) = m \\ R^o(\lambda,z) \neq 0  }} e^{-2 \pi v R^o(\lambda,z)}  \, v^{-1} \, d v  \ + \   S_{m,\mu}(z) \cdot 	\CT_{s=0} \,  \lim_{T \to \infty} \int_{w_0}^T v^{-s-1} \, dv .  \label{eqn:PThetaRegAlt}
		\end{align}
		Since
			\[ 	\CT_{s=0} \,  \lim_{T \to \infty} \int_{w_0}^T v^{-s-1} \, dv \ = \  	\CT_{s=0} \,  \lim_{T \to \infty} \frac{T^{-s} - w_0^{-s}}{-s} \ = \ - \log(w_0) , \]
		we may continue
		\begin{align*}
			\eqref{eqn:PThetaRegAlt} \,	=&  \  \lim_{T \to \infty} \int_{w_0}^T \, \sum_{\substack{ \lambda  \in \mu + L \\ Q(\lambda) = m \\ R^o(\lambda,z) \neq 0  }} e^{-2 \pi v R^o(\lambda,z)}  \, v^{-1} \, d v  \ -  S_{m,\mu}(z) \cdot \log w_0 \\ 
				=& \  \lim_{T \to \infty} \int_{w_0}^T \, \left( \sum_{\substack{ \lambda  \in \mu + L \\ Q(\lambda) = m \\ R^o(\lambda,z) \neq 0  }} e^{-2 \pi v R^o(\lambda,z)} + S_{m,\mu}(z) \right)   v^{-1} \, d v  \ -  S_{m,\mu}(z) \cdot \log T \\
				=& \, \lim_{T\to\infty} \, \int_{\calF_T - \calF_{w_0}} \, P(\tau) \, \Theta_L(\tau) d \mu(\tau) \  - \ S_{m,\mu}(z) \, \log T.
		\end{align*}
	The lemma follows immediately.
\end{proof}

\end{lemma}

\begin{theorem}
  \label{thm:Xitheta}
	If $z \in \domain^o(V) \setminus Z(m,\varphi_{\mu})$ and $m \neq 0$, then
        \[
          \KGrO(m, w, \mu)(z)  = \langle P_{m, w, \mu} , \ \Theta_L(\cdot, z) \rangle^{\reg}.
        \]
        In particular, $\langle P_{m, w, \mu} , \ \Theta_L(\cdot, z) \rangle^{\reg}$ provides a (discontinuous!) extension of $\KGrO(m, w, \mu)$ to all $ z \in \domain^o(V)$. 
 
 Similarly, 
\[
  \KGrO(0, \mu, w)(z) \  = \langle P_{0,\mu,w}, \ \Theta(\cdot, z) \rangle^{\reg} + \delta_{\mu,0}\, \log(w).
\]
\begin{proof}
	First consider the case $m \neq 0$. Since $z \notin Z(m, {\mu})$, it follows that $S_{m,\mu}(z) = 0$, so 
	\[  
              \langle P_{m, w, \mu}, \ \Theta_L(\cdot, z) \rangle^{\reg} = \lim_{T \to \infty} \int_{\calF_T} P_{m, w, \mu}(\tau) \, \Theta_L(\tau,z) \, d \mu(\tau). 
        \]
	Arguing as in the proof of \Cref{thm:MaassSection} to unfold  the Poincar\'e series, see \eqref{eqn:intPFT}, 
	\begin{align*}  
	\lim_{T \to \infty}	\int_{\calF_T} P_{m, w, \mu}(\tau) \, \Theta(\tau,z) \, d \mu(\tau) \ =& \ \lim_{T \to \infty}  \int_{\calR_w^T} e^{-2 \pi i m \tau} \, \Theta_L(\tau,z)(\varphi_{\mu}) \, d\mu(\tau) \\
	=&  \  \lim_{T \to \infty} \int_w^T \, \sum_{ \substack{\lambda \in \mu+L \\ Q(\lambda) = m} } e^{-2 \pi v R^o(\lambda, z)} \, \frac{dv}v.  \\
	=&  \  \sum_{\substack{\lambda \in L' \\ Q(\lambda) = m}}  \, \varphi_{\mu}(\lambda) \, \beta(2 \pi w R^o(\lambda, z))  \ = \  \KGrO(m, w, \mu )(z).
	\end{align*}
   The case $m=0$ follows from similar considerations. 
\end{proof}
\end{theorem}

The natural $S(L)^{\vee} \otimes S(L)$-valued version holds as well. Define 
\[ 
  \KGrO(m,w) \colon \domain^o(V) \setminus \left( \bigcup_{\mu} \, Z(m,\mu) \right) \ \to \ S(L)^{\vee} , \qquad \KGrO(m,w)(z)(\varphi) \ = \ \KGrO(m, w, \varphi), 
\]
which defines a Green function for the $S(L)^{\vee}$-valued cycle 
\[ 
  Z(m) \colon \varphi \mapsto Z(m,\varphi). 
\]
Upon extending $S_{m,\mu}(z)$ to a functional $S_m(z) \in S(L)^{\vee} $ by linearity,
the regularized theta lift 
\[      
  \langle P_{m,w}, \ \Theta_L(\cdot, z) \rangle^{\reg} = \lim_{T \to \infty} \left(\int_{\calF_T} P_{m,w}(\tau) \, \Theta_L(\tau,z) \, d\mu(\tau)  \ -  \  S_{m}(z) \, \log(T)\right)  
\]
gives an extension of $\KGrO(m,w)$ to all $z \in \domain^o(V)$ when $m \neq 0$; similarly, 
\[ 
  \langle P_{0,w} , \ \Theta_L(\cdot, z) \rangle^{\reg} + \log w \, \varphi_0^{\vee}  
\]
gives an extension of $\KGrO(0,w)$.
\subsection{Bruinier's Green functions and the Archimedean generating series}
\label{sec:relat-borch-lift}
The second family of Green functions arise as regularized theta lifts of the harmonic Maa\ss{} forms $F_{m,\varphi}$ of weight $k = 1-p/2$: 
for $m \in \Q$ and $\varphi\in S(L)$, define \emph{Bruinier's Green function}
\[ 
  \BGrO{m,\varphi}(z) := \ \langle F_{m, \varphi} , \, \Theta(\cdot, z) \rangle^{\reg}. 
\]
By analyzing the singularities of this function, as in \cite[Chapter 2]{brhabil}, one finds that it defines a Green function for the cycle $Z(m,\varphi)$; this construction generalizes that of \cite{boautgra}, where weakly holomorphic forms were used.
Note that we are extending the definition to the case that $m \leq 0$, where $\BGrO{m,\varphi}$ is zero for 
all but finitely many $m \leq 0$ and is moreover  a smooth function in this case, i.e. a Green function for the zero cycle.

As usual, there is an $S(L)^{\vee}$-valued version
\[ 
  \BGrO{m}(z) \colon \varphi \mapsto \BGrO{m, \varphi}(z), 
\]
defining an $S(L)^{\vee}$-valued Green function for $Z(m)$. 

The  statement analogous to \Cref{lem:xireg} holds for $\BGrO{m,\mu}$, and can be proved in much the same manner; 
see also \cite[Lemma 4.5]{bryfaltings} and \cite[Lemma 2.19]{schofer} for the corresponding statement at a CM point.
\begin{lemma}
 \label{lem:Breg}
  The regularized integral $\BGrO{m}(z) = \langle F_m, \Theta_L(\cdot, z) \rangle^{\reg}$ exists for all $z \in \domain^o(V)$, and
  \begin{equation}
    \label{eq:Breg}
     \BGrO{m}(z) = \lim_{T \to \infty} \left(\int_{\calF_T}  F_{m}(\tau)\, \Theta_L(\tau,z) \,  d\mu(\tau) \ - \  [S_{m}(z)-\delta_{m,0}\ \varphi_0^{\vee} \ +  \ \varphi_0^{\vee} \cdot c_{F_m}(0)]  \log(T)\right).
  \end{equation}
  Here $c_{F_m}(0) \in S(L)^{\vee} \otimes_{\C} S(L)$ is the zeroth Fourier coefficient of $F_m$, so that
  \[  \varphi_0^{\vee} \cdot c_{F_m}(0) \in S(L)^{\vee} \qquad \text{ is the functional }  \qquad \varphi \ \mapsto \   c_{F_{m,\varphi}}(0)(0). \] \qed
\end{lemma}

\begin{corollary} \label{cor:ArchDiffRegLift}
 For each $m \in \Q$ and parameter $v\in \R_{>0}$, 
 the function 
\[
  \langle  P_{m,v}, \Theta(\cdot,z) \rangle^{\reg} - \langle F_{m}, \Theta(\cdot,z) \rangle^{\reg}
\]
extends to a smooth function on $\domain^o(V)$, given by
	\begin{align*}
		\langle  &P_{m,v} -   F_{m}, \Theta(\cdot,z) \rangle^{\reg} \\
		&=  \  \lim_{T \to \infty} \left(\int_{\calF_T} \left( P_{m,v}(\tau') -  F_{m}(\tau')\right) \Theta(\tau',z) \,  d\mu(\tau') \ -(\delta_{m,0}\ \varphi_0^{\vee} - \   \varphi_0^{\vee} \cdot c_{F_m}(0)) \, \log(T)\right).
	\end{align*}
	\begin{proof} 
                Up to a constant (when $m=0$), $\langle  P_{m,v}, \Theta(\cdot,z) \rangle^{\reg} - \langle F_{m}, \Theta(\cdot,z) \rangle^{\reg}$ 
                equals the difference $\KGrO(m,v) - \BGrO{m}$ by \Cref{thm:Xitheta},
		which is a Green function for the zero cycle. Hence, it is smooth by regularity for the elliptic differential $\ddc$. 
                The given expression for the difference follows from \Cref{lem:xireg} and \Cref{lem:Breg}.
        \end{proof}
\end{corollary}

Our next aim is to identify the differences of these Green functions as Fourier coefficients of a (non-holomorphic) modular form.
\begin{theorem} \label{thm:OrthModGenSeries}
  For each $z \in \mathbb D^o(V)$, the $q$-series 
\[
 - \log v \, \varphi_0^{\vee} \ + \   \sum_{m } \, (\KGrO(m,v) - \BGrO{m})(z) \, q^m
\]
is the $q$-expansion of an element of $\AkappaLexp(\rho_L^{\vee})$ of weight $\kappa = p/2+1$. Here $q = e^{2 \pi i \tau}$ 
with $\tau = u + iv \in \uhp$. 

\begin{proof}
By \Cref{thm:Xitheta} and the definition of $\BGrO{m}$, this generating series is
\[
	\sum_{m} \langle P_{m,v} - F_m, \Theta_L(\cdot, z) \rangle^{\reg} \, q^m. 
\]
	Since $\Theta_L(\cdot,z) \in \ALmod{p/2-1}(\rho_L^{\vee})$, the result follows immediately from \Cref{cor:ArchDiffRegLift}.
\end{proof}
\end{theorem}

\section{Arithmetic theta functions on unitary Shimura varieties} \label{sec:arithUnitary}
In this section, we consider applications of our analytic calculations to the arithmetic geometry of unitary Shimura varieties. 
The two families of Green functions considered in the previous section
turn out to be Green functions for special cycles on unitary Shimura varieties, which we will consider in this section.
We then define \emph{arithmetic theta series} attached to each of the two families of Green functions, 
and show that the difference of these two series transforms as a modular form. 
We also prove that this latter form is orthogonal to cusp forms under the Petersson inner product; this can be viewed as a holomorphic projection formula. Finally, we show that pairing the difference against a certain `arithmetic volume form' gives the Fourier coefficients of derivatives of Eisenstein series, providing further evidence for Kudla's conjecture in this context. We also describe a refinement of a theorem of Bruinier-Howard-Yang on intersection numbers with CM cycles. 

Throughout, we fix an imaginary quadratic field $\kb$ with ring of integers $o_\kb$, and assume that the discriminant $d_\kb$ is odd. Let $\partial_\kb$ be the different, and fix a generator $\delta_\kb$. The non-trivial Galois automorphism of $\kb$ is denoted by $a \mapsto a'$. Fix once and for all an embedding $o_\kb \hookrightarrow \C$, which allows us in particular to view $\Spec(\C)$ as an object in the category of $o_\kb$-schemes; on those rare occasions in which we need to distinguish the two choices, we write $\C_{\sigma}$ for the field of complex numbers viewed as an $o_\kb$-algebra by an embedding $\sigma \colon o_\kb \to \C$. 

\subsection{Unitary Shimura varieties} \label{sec:UModProb}

For any integer $n\geq 1$, consider the functor $\M^{\Kra} (n-1,1)$ over $\Spec o_\kb$ defined by the following moduli problem: for any scheme $S$  over $o_\kb$ with  structure map $\tau_S \colon o_\kb \to \mathcal O_S$, the $S$-points $\M^{\Kra}(n-1,1)(S)$  comprise the category
\[ \mathcal M^{\Kra}(n-1,1) (S) \ = \ \{ \underline A = (A, i, \lambda, \lie{F}) \}, \]
where \begin{enumerate}
\item $A$ is an abelian scheme over $S$ of relative dimension $n$;
\item $i \colon o_\kb \to \End(A)$ is an $o_\kb$-action;
\item $\lambda$ is a principal polarization such that $\lambda \circ i(a) \ = \ i(a')^{\vee} \circ \lambda$ for all $a \in o_\kb$; and
\item (\emph{Kr\"amer's condition} \cite{kramer}) $\lie{F}$ is a locally free subsheaf of $\mathsf{Lie}(A) $ of rank $n-1$ such that for all $a \in o_\kb$, the induced map $\mathsf{Lie}(i(a))$ agrees with $\tau_S(a) $ on $\lie{F}$, and with $\tau_S(a')$ on $\mathsf{Lie}(A) / \lie{F}$.
\end{enumerate}
This moduli problem is represented by a Deligne-Mumford stack, which we also denote by ${\mathcal{M}^{\Kra}(n-1,1)}$, which is regular, flat over $\Spec(o_\kb)$, and smooth over $\Spec o_\kb[1/d_\kb]$.

\begin{remark} Replacing the condition (iv) above with the perhaps more familiar \emph{determinant condition} as in e.g.\ \cite[\S 2.1]{KRunn2} yields a space $\mathcal M^{\mathrm{naive}}(n-1,1)$ that is in general neither flat nor regular at primes dividing $d_\kb$; thus, for the purposes of arithmetic intersection theory, Kr\"amer's model $\mathcal M^{\Kra}$ is more suitable.	
\end{remark}

Similarly, we may consider the moduli space $\M(1,0)$ parametrizing principally polarized elliptic curves $E$ with complex multiplication by $o_\kb$, where the action is normalized to coincide with the structural morphism on $\mathsf{Lie}(E)$. This stack is smooth and proper of relative dimension 0 over $\Spec(o_\kb)$ \cite[Proposition 2.1.2]{howard-unitary-2}.

Finally, let
\[ \M \ = \ \M(1,0) \times_{\Spec(o_\kb)} \M^{\Kra}(n-1, 1). \]

One of the main results of \cite{howard-unitary-2} is the construction of a canonical toroidal compactification of $\M$, obtained by extending the moduli problem to generalized abelian varieties; its  properties are summarized in the following proposition. 

\begin{proposition}[{\cite[Theorem A]{howard-unitary-2}}] The canonical toroidal compactification of $\mathcal M$ is a DM stack $\mathcal M^*$ that is regular of dimension $n$, proper and flat over $\Spec o_\kb$, and which contains $\M$ as an open dense substack. Moreover, the boundary $\partial \M^* = \M^* \setminus \M$ is a divisor on $\M^*$.
\end{proposition}

The stack $\M$ admits a decomposition
	\begin{equation}  \label{eqn:MDecomp}
		\M \ = \ \coprod_{[\calV]}  \ \M_\calV 
	\end{equation}
where $[\calV]$ runs over the (finite) set of isomorphism classes of Hermitian vector spaces over $\kb$ of signature $(n-1,1)$ that contain a self-dual lattice, see \cite{KRunn2, bruinier-howard-yang-unitary}. More precisely, suppose $z \in \M(\C)$ corresponds to a pair of complex abelian varieties $(\underline E, \underline A)$, equipped with $o_\kb$-actions and polarizations. These additional structures endow the homology groups $H_1(E(\C), \Q)$ and $H_1(A(\C),\Q)$ with $\kb$-Hermitian forms. 
The component $\M_\calV$ is then characterized by the property that for any complex point $z \in \M_\calV(\C)$ corresponding to  $(\underline E, \underline A)$, there is an isomorphism
\[ \Hom_k \big( H_1(E(\C), \Q) , \, H_1(A(\C), \Q)\big)  \   \simeq  \ \calV  \]
of Hermitian vector spaces.

As a consequence of \eqref{eqn:MDecomp}, there is a decomposition
	\[ \M^* \ =  \ \coprod_{\calV} \M_\calV^*, \]
where  $\M^*_\calV $ is the Zariski closure of $\M_\calV $ in $\M^*$.

Next, we recall the complex uniformizations of $\M_\calV$ and $\M_\calV^*$. For any Hermitian space $\calV$ of signature $(n-1,1)$, let
\[ \mathbb D (\calV) \ := \ \{ z \subset \calV \otimes_{\Q} \R \ \text{ a negative definite } \kb \otimes_{\Q} \R \text{-line}\}. \]
This is a complex manifold of dimension $n-1$, and is a model for the locally symmetric space attached to $GU(\calV)$.

If $\calL_0$ and $\calL_1$ are self-dual Hermitian $o_\kb$-lattices of signature $(1,0)$ and $(n-1,1)$ respectively, and 
\[ \calV_{\calL_0, \calL_1} \ := \ \Hom_{o_\kb}(\calL_0, \calL_1)\otimes_{\Z} \Q, \]
then the group 
\[ \Gamma_{\calL_0,\calL_1} \ := \ \Aut(\calL_0) \times \Aut(\calL_1) \]
 acts on $\calV_{\calL_0, \calL_1}$ by unitary transformations, and hence on $\mathbb D(\calV_{\calL_0,\calL_1})$. We obtain a \emph{complex uniformization} as in \cite[\S 3.2]{bruinier-howard-yang-unitary}, see also \cite[\S 3]{KRunn2},
\begin{equation}	\label{eqn:MCpxUnif}
\M_\calV(\C) \ \simeq \ \coprod_{[\calL_0, \calL_1]} \left[ \Gamma_{\calL_0,\calL_1} \big\backslash \mathbb D(\calV) \right]
\end{equation}
where the disjoint union is taken over isomorphism classes of pairs of lattices $\calL_0$ and $\calL_1$ such that $\calV_{\calL_0,\calL_1} \simeq \calV$.
Implicitly in this statement, we have fixed a set of representatives $\{ (\calL_0, \calL_1) \}$ and for each such pair, we view 
\[
	\Hom_{o_{\kb}}(\calL_0, \calL_1) \ \subset \ \calV 
\]
as a lattice in $\calV$ via a fixed isomorphism $\calV_{\calL_0,\calL_1} \simeq \calV$; in particular, the group $\Gamma_{\calL_0,\calL_1}$ acts on $\domain(\calV)$ via this fixed isomorphism.

\subsection{Kudla-Rapoport divisors}

We now turn to the definition of the Kudla-Rapoport divisors, following \cite{KRunn2, bruinier-howard-yang-unitary}.
Suppose $(\underline E, \underline A) \in \mathcal M(S)$ for some base scheme $S$; then the $o_\kb$-module
\[ L(E,A) \ := \ \Hom_{o_\kb,S}(E,A) \]
admits an $o_\kb$-Hermitian form, defined by the formula
\begin{equation}  \label{eqn:defHomHerm}
(y_1,y_2) \ := \ \ \lambda_E^{-1} \circ y_2^{\vee} \circ \lambda_A \circ y_1 \ \in \ \End_{o_\kb}(E) \ \stackrel{i_E^{-1}}{\simeq} o_\kb. 
\end{equation}

For each $m \in \Q_{>0}$ and ideal $\lie{a} \subset o_\kb $ dividing $\partial_\kb$, let $\Zed(m, \lie a)$ be the Deligne-Mumford stack over $\Spec(o_\kb)$ representing the following moduli problem:  for a scheme $S / \Spec(o_\kb)$, the points of $\Zed(m, \lie a)$ comprise the category
\[ \Zed(m, \lie a)(S) \ = \ \left\{ (\underline E, \underline A, y) \right\}, \]
where
\begin{itemize}
	\item $\underline E = (E, i_E, \lambda_E) \in \mathcal M(1,0)(S)$ and $\underline A = (A, i_A, \lambda_A, \lie{F}) \in \mathcal M^{\mathrm{ Kra}}(n-1,1)(S) $
	\item $y \in \lie a^{-1} \, L(E, A)$ with $(y,y) = m$, and such that
    \[ \Lie( i(\delta_k) \circ y ) \colon \Lie(E) \ \to \ \Lie(A) \]
    induces the trivial map $\Lie(E) \to \Lie(A) / \lie{F} $.
\end{itemize}
The forgetful map $\Zed(m, \lie a) \to \M$ defines a divisor on $\M$, cf.\ \cite[\S 3.1]{bruinier-howard-yang-unitary}, which we denote by the same symbol $\Zed(m, \lie a)$ in a hopefully mild act of violence against notation.

Setting $\Zed_\calV(m, \lie a) = \Zed(m, \lie a) \times_\M \M_\calV$, there is a complex uniformization
\begin{equation}  \label{eqn:ZmCpxUnif}
	\Zed_\calV(m,\lie a)(\C) \ \simeq \ \coprod_{[\calL_0], [\calL_1] } \, \left[ \Gamma_{\calL_0,\calL_1} \Bigg\backslash 
	\coprod_{ \substack{y \in \lie a^{-1} \Hom(\calL_0,\calL_1) \\ (y,y)=m }} \mathbb D_y  \right];
\end{equation}
here
\[
	\mathbb D_y := \{z \in \mathbb D( \calV)  \ |\   z \perp y \}.
\]	
 This uniformization is compatible with \eqref{eqn:MCpxUnif}, in the sense that the map $\Zed_\calV(m,\lie a)(\C) \to \M_\calV(\C)$ is induced by the inclusions $\mathbb D_y \hookrightarrow \mathbb D(\calV)$.

\begin{definition}
Let $\Zed^*(m, \lie{a})$ be the Zariski closure of $\Zed(m, \lie{a})$ in $\M^*$, and 
\[ \Zed_\calV^*(m,\lie{a}) \ := \ \Zed^*(m, \lie{a}) \times_{\M^*} \M^*_\calV .\]
\end{definition}

It will be useful to take a `vector-valued' approach, as follows.
Suppose $\calL \subset \calV$ is a Hermitian self-dual lattice, 
and $\calL_{[\Z]}$ the $\Z$-lattice of signature $(2n-2,2)$ obtained by taking the trace of the Hermitian form. Then the $\Z$-dual $\calL_{[\Z]}^{\vee}$ satisfies 
\[ 
  \mathcal L_{[\Z]}^{\vee} = \ \diffinv{\kb} \mathcal L 
\]
and in particular there is an action of $\SL_2(\Z)$ on 	
\[ 
 S(\calL) := S(\calL_{[\Z]}) \ \cong \ \C[ \diffinv{\kb} \mathcal L / \mathcal L] 
\]
via the Weil representation, as in  \Cref{sec:WeilRep}.

For each $\ov{m} \in \Q / \Z$ and $ \lie a | \partial_k$, define
\[ 
  \varphi_{\ov{m},\lie a} \ 
   = \ \text{characteristic function of } \left\{ x \in \fraka^{-1} \calL / \calL , \, (x,x) \equiv \ov{m} \bmod \Z \right\} \in S(\calL);
\]
note that $\varphi_{\ov{m}, \lie a} = 0$ if $m \notin d_\kb^{-1} \Z / \Z$.

By \cite[Remark 3.9]{bruinier-howard-yang-unitary}, the set $\{\varphi_{\ov{m},\lie a}  \ |  \ \delta_k \subset \lie a, \, \ov{m} \in N(\lie a)^{-1} \Z / \Z \}$ forms a basis for the space 
\begin{equation} \label{eqn:SDefn}
\calS \ := \ S(\calL)^{ \Aut(\diffinv{\kb} \mathcal L / \mathcal L)} 
\end{equation}
of $\Aut(\diffinv{\kb} \calL / \calL)$ invariant functions.
The Weil representation commutes with $\Aut(\diffinv{\kb} \mathcal L / \mathcal L)$, and so $\mathcal S$ inherits an action of $\SL_2(\Z)$; furthermore, this action depends only on $\calV$ and not the choice of lattice $\calL$, since any two self-dual lattices in $\calV$ are in the same genus.  

We define \emph{vector-valued special cycles} 
\[ \Zed_\calV(m) \in   \Div_{\C} \M_\calV \otimes_{\C}\mathcal S^{\vee}\qquad \text{and} \qquad \Zed^*_\calV(m) \in  \Div_{\C} \M^*_\calV \otimes_{\C} \mathcal S^{\vee}   \]
by the formulas
\begin{equation} \label{eqn:vectorValDivs}
 \Zed_\calV(m) \ = \ \sum_{\lie a | \partial_k} \, \Zed_\calV(m, \lie a) \otimes \varphi_{\ov{m},\lie a}^{\vee} 
\qquad
\text{and}
\qquad
 \Zed_\calV^*(m) \ = \ \sum_{\lie a | \partial_k} \, \Zed_\calV^*(m, \lie a) \otimes \varphi_{\ov{m}, \lie a}^{\vee}.
 \end{equation}
For future use, we also define $\Zed_\calV(m) = \Zed_\calV^*(m) = 0$ whenever $m \leq 0$.  

 \begin{remark} \label{rem:cycleorth}The complex cycles $\Zed(m)(\C)$ are closely related to the cycles introduced in \Cref{sec:arch}. Let $\calV_{[\Q]}$ denote the quadratic space of signature $(2n-2,2)$ over $\Q$ obtained by viewing $\calV$ as a vector space over $\Q$ with quadratic form $Q(x) = (x,x)$. Then, the set
\[
	\mathbb D^o(\calV_{[\Q]}) \ := \ \left\{ \zeta \subset \calV \otimes_{\Q} \R \text{ a negative definite } \R \text{-plane} \right\}
\]
is a model for the symmetric space attached to $O(\calV_{[\Q]})$. There is a natural embedding
\[ 
	\mathbb D(\calV) \ \hookrightarrow \ \mathbb D^o(\calV_{[\Q]}).
\]
given by viewing a negative definite $\kb_{\R}$-line $z \in \domain(\calV)$ as a real plane. Let $ \calL = \Hom(\calL_0, \calL_1)$ and assume $\calV = \calL_{\Q}$.  Recall that for $m \in \Q$, we had defined $S(\calL)^{\vee}$-valued cycles $Z(m)$ on $\domain^o(\calV_{[\Q]})$ by the formula
\[ Z(m)(\varphi) \ = \ \sum_{\substack{x \in \diffinv{\kb}\calL \\ Q(x) = m}} \, \varphi(x) \, Z(x), \]
where $Z(x) := \{ z \in \domain^o(\calV_{[\Q]}) \, | \, z \perp x \} $. Note by construction, if $\varphi \in \calS$, then this sum is invariant under automorphisms of $\calL$, and hence under $\Gamma_{\calL_0, \calL_1}$. 

Thus, the complex uniformization \eqref{eqn:ZmCpxUnif} can be rephrased as saying that the restriction of the cycle $\Zed_{\calV}(m)(\C)$ to the component $[ \Gamma_{\calL_0, \calL_1} \bs \domain(\calV)]$ is given by: first restricting $Z(m)$ to $\calS^{\vee}$, then pulling back to the unitary symmetric space $\domain(\calV)$, and finally taking the image in the quotient  $[\Gamma_{\calL_0, \calL_1}\bs \domain(\calV)]$.
\end{remark}
\subsection{Classes in arithmetic Chow groups} \label{sec:arithChow}
In this section, we recall two ways of equipping the divisors $\Zed_\calV^*(m,\lie a)$ with Green functions to obtain classes in the arithmetic Chow group $\ChowHatC(\M_\calV^*)$. Here and throughout we work with the $\log\text{-}\log$ singular version due to Burgos Gil-Kramer-K\"uhn \cite{burgos-kramer-kuhn}, or more precisely, the `stacky' extension described in \cite[\S 3.1]{howard-unitary-2}.

Roughly speaking, the extended Chow group  $\ChowHatC(\M_\calV^*)$ is a complex vector space spanned by elements of the form
\[ (Z, \, g_Z) \]
where $Z$ is a $\C$-divisor on $\M^*_\calV$ (i.e.\ a formal $\C$-linear combination of closed substacks, each of which is \'etale locally cut out by a single non-zero equation), and $g_Z$ is a current on $\M^*_\calV(\C)$ that is 
\begin{enumerate}
\item smooth outside the support of $Z(\C) \cap \M_{\calV}(\C)$ with logarithmic singularities along this support;
\item has, along with its derivatives, at worst `$\log$-$\log$' singularities along the boundary $\partial \M^*_{\calV}(\C)$, cf.\ \cite[Definition 7.1]{burgos-kramer-kuhn}; and 
\item satisfies Green's equation
  \[ 
     \ddc[g_Z] \ + \ \delta_{Z(\C)} \ = \ [\omega] 
  \]
for some $(1,1)$ differential form $\omega$ that is smooth on $\M_{\calV}(\C)$.
\end{enumerate}
\begin{remark}
Strictly speaking, the analytic aspects of \cite{burgos-kramer-kuhn} are applicable to manifolds, and not orbifolds like $\calM_{\calV}^*(\C)$. To circumvent this technicality, note that by fixing additional level structure in the moduli problem, one can find a manifold  $M$ and a finite group $K$ acting on it, such that $\calM^*_{\calV}(\C) = [K \backslash M]$. By pulling back the cycles $\calZ_{\calV}(m)(\C)$ along the projection
\[ M \ \to \ [K\backslash M ] \ = \ \calM_{\calV}^*(\C) \]
we may interpret $(i)-(iii)$ in terms of $K$-invariant currents on $M$. 
 In the sequel we will gloss over this technicality, and refer the reader to \cite[Section 3.1]{howard-unitary-2} for a more careful treatment of this issue.
\end{remark}
 As with ordinary Chow groups, the principal divisors are deemed equivalent to zero: more precisely, if $f \in \Q(\M_{\calV}^*)^{\times}$ is a global rational function, then
the divisor
\[ 
   \divhat(f) \ := \ (\mathrm{div} f, \,  - \log|f|^2)
\]
is called principal. The group $\ChowHatC(\M_{\calV}^*)$ is then defined to be the quotient of the space of arithmetic divisors by the subspace spanned by principal divisors.

We begin with the construction of Kudla's Green functions $\KGr(m,v)$. In light of the complex uniformization \eqref{eqn:MCpxUnif}, it suffices to specify $\KGr(m,v)$ on each component $[\Gamma_{\calL_0, \calL_1} \backslash \mathbb D(\calV)]$ where, as before, 
\[ \calL = \Hom(\calL_0, \calL_1) \hookrightarrow \calV . \]
In \Cref{sec:kg-theta}, we defined an $S(\calL)^{\vee}$-valued current $\KGrO(m,v)$ on $\domain^o(\calV_{[\Q]})$, with singularities along the cycle $Z(m)$; here $\calV_{[\Q]}$ is the space $\calV$ viewed as a quadratic space over $\Q$ of signature $(2n-2,2)$.  In parallel with \Cref{rem:cycleorth}, we may restrict $\KGrO(m,v)$ to $\calS^{\vee}$, and pull back to the unitary Grassmannian $\domain(\calV)$ to obtain a $\Gamma_{\calL_0, \calL_1}$-invariant function on $\domain(\calV)$; descending to the quotient $[\Gamma_{\calL_0, \calL_1} \bs \domain(\calV)]$ gives a function $\KGr(m,v)$ with singularities along the pullback of $\Zed_{\calV}(m)$ to $[\Gamma_{\calL_0, \calL_1} \bs \domain(\calV)]$. 

Putting the components together in the complex uniformization \eqref{eqn:MCpxUnif}, we obtain a function  $\KGr(m,v)$  on $\M_{\calV}(\C)$ with logarithmic singularities along $\Zed_{\calV}(m)(\C)$.
 
Note that if $m\leq0$, then $\KGr(m,v)$ is smooth. 

As we wish to work with the toroidal compactification $\M_{\calV}^*$, we also need to understand the behaviour of $\KGr(m,v)$ at the boundary. 
By the discussion in \cite[\S 2.6]{howard-unitary-2}, the  components of the boundary $ \partial \M^*_{\calV} = \M_{\calV}^* - \M_{\calV}$ are parametrized by the set of isomorphism classes of triples
\[ \mathcal B_{\calV}  \ : = \  \left\{  \ ( \calL_0, \, \lie{m} \subset \calL_1) \ \Big| \
\begin{matrix}
 \lie{m} \text{ is an isotropic rank-1 direct summand of } \calL_1 \text{ and }   \\ \calV_{\calL_0,\calL_1} \simeq \calV  
 \end{matrix} \right\} /_{\simeq}\]
If $B=(\calL_0, \lie{m}\subset \calL_1) \in \mathcal B_{\calV}$, then
\[ 
  \lie{n} := \Hom_{o_\kb} ( \calL_0, \lie{m})  
\]
is an isotropic rank-1 submodule of $\mathcal L = \Hom_{o_\kb}(\calL_0, \calL_1)$, and
\[ 
  \Lambda_B \ := \ \lie{n}^{\perp} / \lie{n} 
\]
is a self-dual Hermitian $o_\kb$-lattice of signature $(n-2,0)$. 
There is a contraction map 
\[ 
  \mathbf c_{B} \colon S(\mathcal L) \to S(\Lambda_B)
\] 
defined by the formula
\begin{equation}  \label{eqn:ContractionMap}
\mathbf c_{B} (\varphi)(\nu) \ = \ \sum_{\substack{\mu \in \diffinv{\kb} \mathcal L/ \mathcal L \\ \mu|_{\lie{n}^{\perp}} = \nu } } \, \varphi(\mu), \qquad \varphi \in S(\mathcal L) ,
\end{equation}
and is equivariant for the action of the Weil representation on $S(\mathcal L)$ and $S(\Lambda_B)$, cf.\ \cite[\S5]{boautgra}.  In the above sum, the equality $\mu|_{\lie n^{\perp}} = \nu$ is to be interpreted as follows: a coset $ \mu \in \diffinv{\kb} \calL / \calL$ determines a map
\begin{equation} \label{eqn:contractionMapDef} {\mu}|_{\lie n^{\perp}} \in \Hom_{o_\kb} \left( \lie n^{\perp}, \ \diffinv{\kb} / o_\kb \right) \end{equation}
by choosing any representative $\tilde \mu \in \diffinv{\kb} \calL$ and sending $x \in \lie n^{\perp}$ to the image of $\langle x, \tilde\mu \rangle$ in $ \diffinv{\kb}  / o_\kb$. This is clearly independent of the choice of $\tilde \mu$. 
On the other hand, a coset $\nu \in \diffinv{\kb} \Lambda_B / \Lambda_B$ determines an element of $\Hom(\Lambda_B, \diffinv{\kb}/o_\kb)$  by sending $\lambda \in \Lambda_B$ to the image of $\langle \lambda, \tilde \nu \rangle$ in $\diffinv{\kb}/o_\kb$, for any representative $\tilde \nu$ of $\nu$. Pulling back along the projection $\lie n^{\perp} \to \Lambda_B$ defines an element of $\Hom(\lie n^{\perp}, \diffinv{\kb} / o_\kb)$, and the equation $\mu|_{\lie n^{\perp}} = \nu$ is interpreted as the equality of this element with the map \eqref{eqn:contractionMapDef}.

Let $\vartheta_{\Lambda_B} \in M_{n-2}(S(\Lambda_B)^{\vee})$ denote the theta function attached to $\Lambda_B$; it
is a vector-valued modular form of weight $n-2$ defined by the formula
\[ \vartheta_{\Lambda_B}(\tau)(\varphi) \ = \ \sum_{\lambda \in \diffinv{\kb} \Lambda_B } \varphi (\lambda)  \, q^{(\lambda, \lambda)}.
\]
Pulling back along the map $\mathbf c_{B} \colon S(\mathcal L) \to S(\Lambda_B)$ and restricting to $\calS \subset S(\mathcal L)$, yields a theta function
\[  \Theta_{\Lambda_B}(\tau) \ \in \ M_{n-2}(\mathcal S^{\vee}), \]
whose Fourier expansion we write as
\begin{equation} \label{eqn:bdyThetaFn}
 \Theta_{\Lambda_B}(\tau) \ = \ \sum_{m \in \Q_{\geq 0}} \, \mu_B(m) \, q^m 
\end{equation}
with coefficients $\mu_B(m) \in \mathcal S^{\vee}$. Concretely, a straightforward computation
yields
\[ \mu_B(m) \ = \ \sum_{\lie a} \, N(\lie a) \cdot  \# \{ \lambda \in \lie{a}^{-1} \Lambda \ | \ (\lambda,\lambda) = m  \} \cdot \varphi_{\ov{m},\lie a}^{\vee}. \]

\begin{remark} \label{rmk:nEquals2}
Suppose $n = 2$, so that $\calV$ is either anisotropic or has Witt rank equal to one (i.e. is split). In the first case $\mathcal M_{\calV} = \M_{\calV}^*$ is proper over $\Spec o_\kb$, and in particular there are no boundary components ($\mathcal B_{\calV} = \emptyset$). In the second case, the lattice $ \Lambda_B$ attached to a boundary component $B \in \mathcal B_{\calV}$ is trivial, and so concretely
\[ \mu_B(m) \ = \ 
	\begin{cases} 0 & m \neq 0 \\ 
				 \sum_{\lie a} N(\lie a) \cdot \varphi_{0,\lie a}^{\vee} & m = 0
	\end{cases}
\]
in this case.
\end{remark}

\begin{proposition} 
	 For any $v \in \R_{>0}$, $m \in \Q$, the current $\KGr(m, v)$ extends to an $\mathcal S^{\vee}$-valued current -- also denoted by $\KGr(m,v)$ -- with logarithmic singularities on $\mathcal M_{\calV}^*$, and is a Green current 
	 for the cycle
	\[ \Zed_{\calV}^*(m) \ + \ \frac{1}{4 \pi  v} \ \sum_{B \in \mathcal B_V} \ \mu_B(m) \ [B], \]
	where for $B \in \mathcal B_{\calV}$, we write $[B]$ for the corresponding boundary component.
	
	Recall that $\Zed_{\calV}^*(m) = 0$ whenever $m \leq 0$ by definition; in particular, if $m<0$ then $\KGr(m,v)$ is a Green function for the zero cycle on $\M_{\calV}^*$.
\begin{proof}
	It suffices to check the claim after evaluating at each of the basis elements $\{\varphi_{\ov{m}, \lie a} \}$ of $\mathcal S$. The case where $\ov{m} \equiv 0$ and $\lie a = o_\kb$ is proved in \cite[Theorem 3.4.7]{howard-unitary-2}, and the same proof works with only minor modifications in general.
	 \end{proof}
\end{proposition}

As a consequence, for $v \in \R_{>0}$ and $ m \in \Q_{\neq 0}$ we obtain classes
\[ \ZedVKud(m, v) \ = \left( \Zed_{\calV}^*(m) \ + \ \frac1{4 \pi v } \sum_{B \in \mathcal{B}_{\calV}}  \, \mu_B(m)\, [B], \, \KGr(m, v) \right) \ \in \ \ChowHatC(\mathcal M_{\calV}^*) \otimes_{\C} \mathcal S^{\vee}. \]
Here the superscript $\mathsf{K}$ serves as a reminder that Kudla's Green functions are being used.
Note if $m \notin d_\kb^{-1} \Z$ then $\ZedVKud(m,v) = 0$.
It remains to define the `constant term' $\ZedVKud(0,v)$, a task we will return to shortly.

Turning to Bruinier's automorphic Green functions, suppose $m \in \Q_{\neq 0}$. Recall that in \Cref{sec:relat-borch-lift}, we had considered the $S(\calL)^{\vee}$-valued current $\BGrO{m}$ on $\mathbb D^o(\calV_{[\Q]})$, defined via the regularized theta lift against the Siegel theta function.
\[ \BGrO{m}(z) \ = \ \langle F_m, \, \Theta_L(\cdot, z) \rangle^{\mathrm{reg}}; \]
here $F_m(\tau)$ is the unique weak harmonic Maa\ss\, form of weight $2-n$ specified in \Cref{sec:relat-non-holom}. 

Restricting $\BGrO{m}$ to $\calS^{\vee} \subset S(\calL)^{\vee}$, and then to the unitary Grassmannian 
\[ \mathbb D(\calV) \subset \mathbb D^o(\calV_{[\Q]}), \]
yields a $\Gamma_{\calL_0,\calL}$-invariant current on $\mathbb D(\calV)$, which can be viewed as living on the component $[\Gamma_{\calL_0,\calL} \backslash \mathbb D(\calV)]$; repeating this construction for all the components appearing in \eqref{eqn:MCpxUnif} yields an $\mathcal S^{\vee}$-valued current on $\M_{\calV}(\C)$ that we denote by $\BGr{m}$, with logarithmic singularities along $\Zed_{\calV}(m)(\C)$. Included in the following proposition is a description of the behaviour of $\BGr{m}$ at the boundary.

\begin{proposition}[{\cite[Theorem 4.10]{bruinier-howard-yang-unitary}}] \label{prop:BHYbdy}
 Suppose $m \in \Q_{\neq 0}$. Then $\BGr{m}$ is a Green function for the cycle 
		\[ \Zed^*_{\calV}(m) \ + \ \sum_{B \in \mathcal B_{\calV}} \eta_B(m) \, [B] \]
		where 
		\[ \eta_B(m) \ := \ \frac{1}{4 \pi} \, \langle F_m, \, \Theta_{\Lambda_B} \rangle^{\mathrm {reg}} . \]
\qed
\end{proposition}
\begin{remark}
Via \cite[Remark 4.11]{bruinier-howard-yang-unitary}, we may express $\eta_B(m)$ more explicitly: if $n>2$, then 
	\[  \eta_B(m) =  \frac{ m}{n-2} \, \mu_B(m). \]
	If instead $n=2$ and $\calV$ is isotropic, then
	\[ \eta_B(m) \ = \ -2 \left( \sum_{\lie a | \partial_k}  N(\lie a) \,  \sigma_1(m)   \, \varphi_{\overline m, \lie a}^{\vee}  \ + \ \mathbf c_{B}\left( c_{F_m}^+(0) \right) \right)
	\]
	If $n=2$ and $\calV$ is anisotropic, then $\mathcal B_{\calV} = \emptyset$ and $\BGr{m}$ is a Green function for the cycle $\Zed^*(m) = \Zed(m)$.
\end{remark}

Thus for each $m \in \Q$, we may define an $\mathcal S^{\vee}$-valued arithmetic cycle
\[ \ZedVBru(m) \ = \ \left( \Zed^*_{\calV}(m) \, + \, \sum_{B} \eta_B(m) \, [B], \, \BGr{m}  \right) \ \in \ \ChowHatC(\M_V^*) \otimes_{\C} \mathcal S^{\vee} \]
where the superscript $\mathsf{B}$ reminds us that we are using Bruinier's automorphic Green functions.

Finally, we turn to the constant terms. Let $\otauthat$ denote the \emph{tautological bundle} on $\calM_{\calV}^*$ viewed as an element of $\ChowHatC(\calM_{\calV})$. As this bundle plays only a marginal role in our present work, we refer the reader to \cite[\S6]{bruinier-howard-yang-unitary} for its construction, and content ourselves with the remark that when restricted to the open part $\calM_{\calV}(\C)$, the first Chern form $c_1({\otauthat}|_{\calM_{\calV}(\C)})$ is a K\"ahler form. 

Set
\[ 
  \ZedVBru(0) \ := \ - \otauthat \otimes \varphi_{\ov{0}, o_\kb}^{\vee} + \, \left(\sum_{B} \eta_B(0) \, [B], \, \BGr{0}\right),
\]
and similarly, for $v \in \R_{>0}$, let
\[ \ZedVKud(0,v) \ := \  - \otauthat \otimes \varphi_{\ov{0}, o_\kb}^{\vee} \ + \ \left( \frac{1}{4 \pi v} \sum_{B \in \mathcal B_{\calV}} \mu_B(0) \, [B], \, \KGr(0,v) \right) \ -\ (0, \log v) \otimes \varphi_{\ov{0}, o_\kb}^{\vee}.
\]

\begin{definition} For $\tau = u + i v \in \uhp$ and $q = e^{2 \pi i \tau}$, define formal $q$-expansions
\[ \ThetaVKud(\tau) \ := \ \sum_{m \in \Q} \, \ZedVKud(m,v) \, q^m 
 \qquad \text{and} \qquad 
  \ThetaVBru(\tau) \ := \ \sum_{m \in \Q} \, \ZedVBru(m) \, q^m, \]
  with coefficients valued in $\ChowHatC(\M_{\calV}^*) \otimes_{\C} \mathcal S^{\vee}$; these are the Kudla and Bruinier \emph{arithmetic theta functions}, respectively.  
\end{definition}

\subsection{Modularity results}
Our main results involve viewing the generating series $\ThetaVKud$ and $\ThetaVBru$ as $\ChowHatC(\M_{\calV}^*)\otimes_{\C} \mathcal S^{\vee}$-valued functions; it may be helpful at this juncture to elaborate on what this notion might mean
and what it means for such a function to be modular.

Let $(W, \rho)$ be a finite-dimensional complex representation of $\SL_2(\Z)$. Recall that a smooth function $f \colon \uhp \to W$ is said to \emph{transform as a modular form} of weight $k$ and representation $(W, \rho)$ if it is invariant under the weight $k$ slash operator
\[ f\mid_{k}[\gamma](\tau)  \ := \ (c \tau + d)^{-k} \cdot \rho(\gamma^{-1}) f( \gamma \tau)  \]
for each $\gamma = (\begin{smallmatrix} a&b\\c&d \end{smallmatrix}) \in \SL_2(\Z)$. Denote the space of such functions by $A_k(\rho)$, which is typically infinite-dimensional. 

Now suppose $\{ \widehat\phi_m(\tau) \}_{m \in \Q}$ is a collection of functions 
\[ \widehat{\phi}_m \colon \uhp \ \to  \ \ChowHatC(M^*_{\calV}) \otimes_{\C} W \]
and consider the formal series
\[ \widehat{\phi}(\tau) \  = \  \sum_{m \in \Q} \, \widehat{\phi}_m(\tau) \, q^m. \]

\begin{definition} \label{def:ModularityDef}
We say that $\widehat{\phi}(\tau)$ \emph{transforms as a modular form} of weight $k$ and representation $(W, \rho)$ if for each $m$, there is a decomposition
\[ \widehat{\phi}_m(\tau) \ = \ \widehat{\phi}_m^{\natural}(\tau) \ + \ (0, g_m(\tau) ) \]
such that the following two conditions hold.
\begin{enumerate}
	\item Roughly speaking, the formal generating series $\sum_{m} \widehat{\phi}_m^{\natural} q^m$ lies in $\ChowHatC(\M_{\calV}^*) \otimes_{\C} A_k(\rho)$. More precisely, there are finitely many elements
	\[ \widehat\Zed_1, \dots, \widehat\Zed_r \in \ChowHatC(\M_{\calV}^*)  \]
	and functions 
	\[ c_{i,m}(\tau) \colon \uhp \to W, \qquad \text{for all }i =1, \dots, r \ \text{and} \ m \in \Q, \]
	such that
	\[ \widehat{\phi}_m (\tau) \ = \ \widehat\Zed_1 \otimes c_{1,m} (\tau)  \ + \ \cdots \ + \ \widehat\Zed_r \otimes c_{r,m}(\tau) \]
	and for each $i = 1, \dots, r$, the formal series
	\[ f_i(\tau) := \sum_m c_{i,m}(\tau) \, q^m \]
	converges to a function $f_i \in A_k(\rho)$. For convenience, we abuse notation and write
	\[ \sum_{m} \widehat{\phi}_m^{\natural} q^m \ = \ \widehat \Zed_1 \otimes f_1(\tau) \ + \ \cdots \ + \ \widehat \Zed_r \otimes f_r(\tau). \]
	
	\item For each $m$ and $\tau$, the function $g_m(\tau,z)$ is a $W$-valued Green function for the zero cycle; i.e.\ it is a smooth function
	\[ g_m(\tau,z) \colon \uhp \times \M_{\calV}(\C) \ \to \ W \]
	that extends, pointwise in $\tau$, to a function with at worst $\log\text{-}\log$ singularities along the boundary of $\M_{\calV}^*$.
	
	We then require that there is a function $s(\tau,z)$ on $\uhp \times \M_{\calV}(\C)$ such that:
		\begin{itemize}
			\item for each fixed $\tau$, the function $s(\tau,z)$ and its derivatives in $z$ are smooth on $\M_{\calV}(\C)$ and have at worst $\log-\log$ singularities at the boundary;
			\item for any smooth differential form $\eta \in A^{2n-2}(\M_{\calV}^*(\C))$, the value of the current 
			\[ [s(\tau, z)](\eta) \ := \ \int_{\calM_{\calV}^*(\C)} s(\tau, z) \wedge \eta  \]
			defines a function $[s(\tau, z)](\eta) \in  A_k(\rho)$. 
			\item for each fixed $\tau$, the sum $\sum_{m} [g_m(\tau,z)] q^m$ converges weakly to $[s(\tau, z)]$, i.e.\ 
			\[
				\lim_{N \to \infty}	\sum_{|m| \leq N} \, [g_m(\tau,z)](\eta) \ q^m \ = \ [s(\tau,z)](\eta)
			\]
			for every $\eta \in A^{2n-2}(\M_{\calV}^*(\C))$ as above.
		\end{itemize}
\end{enumerate}
\end{definition}

\begin{remark} (i) If we so desire, we may impose further analytic conditions in the previous definition (e.g.\ holomorphicity, real analyticity, etc.) by replacing $A_k(\rho)$ with the corresponding spaces of modular forms, which we then call the type of $\widehat{\phi}$; thus we may, for example, speak of spaces of holomorphic $W \otimes \ChowHatC(\M_V^*)$-valued modular forms
(i.e.\ of type $M_k(\rho)$), or with a type given by one of the spaces considered in Section 2.

(ii) If a generating  series $\widehat{\phi}(\tau)$ is modular as in the above definition, it defines a map (denoted by the same symbol)
\[ \widehat{\phi} \colon \uhp \to \ChowHatC(\M_{\calV}^*) \otimes_{\C} W, \qquad \widehat{\phi}(\tau) := \sum_{i=1}^r  \widehat \Zed_i  \otimes f_i(\tau)  \ + \ (0, s(\tau,z))  \]
which satisfies
\[ \widehat{\phi}(\tau)\mid_{k}[\gamma] \ = \ (c \tau + d)^{-k}  \ \left(1 \otimes \rho(\gamma^{-1})\right) \widehat{\phi}(\tau) \]
for all $\gamma = (\begin{smallmatrix} a & b \\ c & d \end{smallmatrix})$, and is independent of the choices of $\widehat\Zed_i, f_{m,i}$ and $g_m$ above. 

The upshot of this definition is that given a ``reasonable" pairing 
\[ \ChowHatC(\M_{\calV}^*) \times {\widehat{\mathsf{CH}}{}^{n-1}_{\C}}(\M_{\calV}^*) \to \C, \qquad [\widehat\Zed_1 : \widehat\Zed_2] \ := \ \widehat\deg \, \left(\widehat\Zed_1 \cdot \widehat\Zed_2 \right) \]
cf.\ \Cref{hyp:arithInt} below, and any fixed class $\widehat\calY \in{\widehat{\mathsf{CH}}{}^{n-1}_{\C}}(\M_{\calV}^*)$, the expression $[\widehat{\phi}(\tau) \colon \widehat \calY]$ defines an element of $A_{k}(\rho)$ with $q$-expansion
\[
	[\widehat{\phi}(\tau) : \widehat \calY] \ = \ \sum_{m} \, [ \widehat \phi_m(\tau) : \widehat \calY] \, q^m. 
\]

\end{remark}

We now come to the main theorem of this section.
\begin{theorem} \label{thm:ArithModularity} The difference
$\ThetaVKud(\tau) - \ThetaVBru(\tau)$ transforms as a $\ChowHatC(\M_{\calV}^*)$-valued  form of type $\calA_n^!(\calS^{\vee})$. 
\end{theorem}

For the proof we first set
\[  \ThetaVKud(\tau) - \ThetaVBru(\tau)  \ =: \ \widehat\phi(\tau)\ = \ \sum_{m} \widehat\phi_m(\tau) \, q^m  \]
where, for $m \neq 0$,
\begin{align*}
 \widehat\phi_m(\tau) \ = \  \left( \sum_{B \in \mathcal B_V} \left( \frac{1}{4 \pi v} \mu_B(m) - \eta_B(m) \right) \cdot [B],
  \ \KGr(m,v) - \BGr{m}  \right).
 \end{align*}
Since $\M_\calV^*$ is projective, there is a Green function $\lie g_B$ for each boundary component $[B$] that has $C^{\infty}$ regularity on $\M_\calV(\C)$, and may be normalized so that
\[ \int\limits_{\M_\calV^*(\C)} \lie g_B  \, \mathrm d\Omega \ = \ 0 \]
for the measure $\mathrm d \Omega = \wedge^{n-1} c_1(\widehat\omega_{/ \C} ) $ induced by top wedge power of the Chern form of the tautological bundle. These Green functions define classes
\[ \widehat{[ B ]} \ := \ (B, \lie g_B) \in \ChowHatC(\M_{\calV}^*). \]
Thus for each $m \neq 0$, we may write
\begin{equation}
  \label{eq:PhiHatDecomp}
  \widehat\phi_m(\tau) \ = \ \widehat\phi_{m}^{\natural}(\tau) \ + \ (0,  g_{m}(\tau,z) ),
\end{equation}
where
\[ \widehat\phi_{m}^{\natural}(\tau) \ = \   \sum_{B \in \mathcal B_{\calV}} \ \widehat{[ B]} \otimes   \left( \frac{\mu_B(m)}{4 \pi v} - \eta_B(m) \right) \ \in \   \ChowHatC(\M_{\calV}^*) \otimes \mathcal S^{\vee} \]
and
\begin{equation} \label{eqn:gmDef}
 g_{m}(\tau,z) \ = \ \KGr(m, v)(z) -  \BGr{m}(z) - \sum_{B \in \mathcal B_{\calV}}  \left( \frac{\mu_B(m)}{4 \pi v} - \eta_B(m) \right) \lie g_B(z) 
\end{equation} 

Similarly, when $m=0$,
\begin{align}
\notag \widehat\phi_0(\tau) \ =&  \ \left( \sum_{B \in \mathcal B_V} \left(\frac{\mu_B(0)}{4 \pi v}-\eta_B(0)\right) \, [B]\ , \ \KGr(0,v)-\BGr{0} \right) \ - \  (0, \log v) \otimes \varphi_{\ov{0}, o_\kb}^{\vee}   \\
 \notag =&  \  \sum_B  \, \widehat{[B]}  \otimes \left( \frac{\mu_B(0)}{4 \pi v} -\eta_B(0)\right)\  \\
&\quad + \ \left( 0, \, \KGr(0,v)(z) - \BGr{0} \ - \ \sum_B \left( \frac{\mu_B(0)}{4\pi v} - \eta_B(0) \right) \frakg_B(z)  \ - \  \log v \cdot \varphi_{\ov{0}, o_\kb}^{\vee} \right) \notag\\
 \label{eqn:g0def} =&: \ \widehat\phi^{\natural}_0(v) \ + \ (0, g_0(v,z)) .
 \end{align}
We first show the coefficients of the boundary components are already the coefficients of modular forms.
\begin{lemma}  \label{lem:bdyFourierMod}
For any boundary component $B$,  
\[  
    \sum_{m \geq 0} \left( \frac{\mu_B(m)}{4 \pi v }  \ -  \ \eta_B(m) \right) \, q^m  \ 
     = \ - \frac{1}{4 \pi} \Lsharp\left(\Theta_{\Lambda_B} \right), 
\]
where $\Lsharp\left(\Theta_{\Lambda_B} \right)$ denotes the normalized preimage under the lowering operator
as defined in \Cref{prop:uniquepreimage}.
In particular, the left hand side defines a  form in $\calA_n^{\mathrm{mod}}(\calS^{\vee}) \subset\calA_n^!(\calS^{\vee}) $.
\begin{proof}
We apply \Cref{thm:MaassSection} to $\Theta_{\Lambda_B}(\tau)$, so that the $m$'th Fourier coefficient of $-\Lsharp(\Theta_{\Lambda_B})$ is given by
\[ \langle P_{m,v} - F_m , \, \Theta_{\Lambda_B} \rangle^{\mathrm{reg}} \]
Since $\eta_B(m) = \frac{1}{4 \pi} \langle F_m, \, \Theta_{\Lambda_B} \rangle^{\mathrm{reg}}$, it will suffice to show that 
\[ \langle P_{m,v}, \, \Theta_{\Lambda_B} \rangle^{\mathrm{reg}} \ \stackrel{?}{=} \  \frac{\mu_B(m)}{v} \]
for all $m \in \Q$; note that $\mu_B(m) = 0$ when $m < 0$. 

For convenience, set 
\[P_{m,v,\lie a}(\tau') = P_{m,v}(\tau')(\varphi_{\overline m, \lie a}) \]
for a  (non-zero)  element $\varphi_{\overline m, \lie a}$ and $\tau' \in \mathbb H$. Recall that if $ v_0 := \max(v,1/v)$, then
\[ P_{m,v,\lie a}(\tau') \ = \ (q')^{-m} \cdot \varphi_{\ov{m}, \lie a} \qquad \ \text{ whenever } \Im(\tau') > v_0. \]
Therefore, on the set $v' > v_0$, the integral
\[ \int_{-1/2}^{1/2} P_{m,v,\lie a}(u' + i v') \cdot \Theta_{\Lambda_B}(u'+iv') d u' \ = \ \mu_B(m)(\varphi_{\ov{m}, \lie a})  \]
is bounded uniformly.
It follows that the meromorphic function in $s$ whose constant term at $s=0$ defines the regularized integral
	\[
		 \langle P_{m,v,\lie a}, \, \Theta_{\Lambda_B} \rangle^{\mathrm{reg}} \ = \ \CT_{s=0}  \, \lim_{T \to \infty} \, \int_{\calF_T} P_{m,v,\lie a}(\tau') \cdot \Theta_{\Lambda_B}(\tau') \, \frac{ d \mu(\tau')}{(v')^s} 
	\]
is holomorphic at $s=0$, i.e.\ 
	\[
		\langle P_{m,v,\lie a}, \, \Theta_{\Lambda_B} \rangle^{\mathrm{reg}} \ = \ \lim_{T \to \infty} \int_{\calF_T} P_{m,v,\lie a}(\tau') \cdot \Theta_{\Lambda_B}(\tau') \,  d \mu(\tau') .
	\]
Unfolding the Poincar\'e series, as in the proof of \Cref{thm:MaassSection}, then gives
	\begin{align*}
		\lim_{T\to\infty} \int_{\calF_T} P_{m,v,\lie a}(\tau') \cdot \Theta_{\Lambda_B}(\tau') \,  d \mu(\tau') \ =& \ \int_{0}^{\infty} \int_{-1/2}^{1/2} \  \sigma_v(\tau') \ (q')^{-m}  \ \Theta_{\Lambda_B}(\tau')(\varphi_{\ov{m}, \lie a})  \ \frac{du' \, dv' }{(v')^2} \\
		=& \ \int_v^{\infty} \int_{-1/2}^{1/2} \ (q')^{-m}  \ \Theta_{\Lambda_B}(\tau')(\varphi_{\ov{m}, \lie a})  \ \frac{du' \, dv' }{(v')^2} \\
		=& \ \mu_{B}(m)(\varphi_{\ov{m}, \lie a}) \cdot \int_v^{\infty} \frac{dv'}{(v')^2} \\
		=& \ \frac{\mu_{B}(m) (\varphi_{\ov{m}, \lie a}) }{v}
	\end{align*}
as required.
The fact that $\Lsharp(\Theta_{\Lambda_B})$ has moderate growth can be easily deduced from the Fourier expansion.
\end{proof}
\end{lemma}

\begin{remark}
  Observe that we can also identify the sum on the left-hand side in Lemma \ref{lem:bdyFourierMod}
  as (up to a non-zero constant) the image of $\Theta_{\Lambda_B}$ under the weight $n-2$ raising operator.
\end{remark}

\begin{corollary} \label{cor:ArithBdyMod} The formal generating series $ \widehat\phi^{\natural}(\tau)$ is an element of  $ \ChowHatC(\M_{\calV}^*) \otimes_{\C} \calA^{\mathrm{mod}}_n(\mathcal S^{\vee})$ in the sense of \Cref{def:ModularityDef}(i). 
\begin{proof} This follows immediately from the previous lemma, since
\[ \widehat\phi^{\natural}(\tau) \ = \ - \frac1{4 \pi} \sum_{B} \, [\widehat{B}] \otimes \Lsharp(\Theta_{\Lambda_B})  . \]
\end{proof}
\end{corollary}

\begin{proof}[Proof of \Cref{thm:ArithModularity}]
In view of \Cref{cor:ArithBdyMod} and the decomposition \eqref{eq:PhiHatDecomp}, it remains to prove the modularity, in the sense of \Cref{def:ModularityDef}(ii), of the ``archimedean part"
\[
	\widehat{\phi}(\tau) \ - \ \widehat\phi^{\natural}(\tau) \ = \ \sum_m ( 0, g_m(v,z)) \, q^m
\]
where
\[
		g_m(v,z) \ =\  \KGr(m, v)(z) -  \BGr{m}(z) - \sum_{B \in \mathcal B_{\calV}}  \left( \frac{\mu_B(m)}{4 \pi v} - \eta_B(m) \right) \lie g_B(z)  - \delta_{m,0} \log v \cdot \varphi_{\ov{0}, o_\kb}^{\vee}.
\]
Note that for fixed $z$, we may consider the $q$-series
\begin{align*}
	g(\tau,z) \ := \ \sum_m \, g_m(v,z) q^m \ = \ - \log(v) \, \varphi_{\bar0, o_\kb}^{\vee} & \ + \ \sum_{m} \left( \KGr(m,v) - \BGr{m} \right)(z) \, q^m \\
	&  - \frac{1}{4 \pi} \sum_B  \Lsharp(\Theta_{\Lambda_B})(\tau) \cdot \lie g_B(z);
\end{align*}
since the Green functions $\KGr(m,v)$ and $\BGr{m}$ are obtained by restricting their orthogonal counterparts along $\domain(\calV) \hookrightarrow \domain^o([\calV]_{\Q})$, the pointwise-in-$z$ modularity of $g(\tau,z)$ follows immediately from \Cref{thm:OrthModGenSeries}, which in turn relies on our abstract characterization of $\Lsharp$. Unfortunately, the methods of \Cref{sec:invert-xi} do not give us any information for its behaviour as $z$ varies, while  \Cref{def:ModularityDef}(ii) requires more control in this aspect; we will therefore need to be somewhat indirect in our approach.

Consider the Kudla-Millson theta function \cite{kudlamillsonintersection}
\[ 
	\Theta_{KM}(\tau) \colon \uhp \ \to \ Z^{(1,1)}(\calM_{\calV}(\C)) \otimes_{\C} \calS^{\vee}
\]
which we view as a non-holomorphic form of weight $n$ in $\tau$, valued in the tensor product of the space of closed (smooth) differential forms of degree $(1,1)$ on $\calM_{\calV}(\C)$ and $\calS^{\vee}$. Its Fourier expansion is
\begin{equation} \label{eqn:KMThetaFourier}
	\Theta_{KM}(\tau) \ = \ - \ \varphi_0^\vee \cdot c_1(\omega^{\mathrm{taut}}) \ + \  \sum_{m \in \Q} \, \omega_{\KGr(m,v)} \, q^m,
\end{equation}
where 
\[ 
	[\omega_{\KGr(m,v)}] \ = \ \ddc [\KGr(m,v)] \ + \ \delta_{\calZ(m)(\C)} \qquad \text{on }  \calM_{\calV}(\C).
\]
Moreover, 
\begin{equation} \label{eqn:LowerThetaKM}
	\Low(\Theta_{KM})(\tau) \ = \ \ddc \Theta_L(\tau)
\end{equation}
where $\Theta_L(\tau)$ is the Siegel theta function, as in \eqref{eq:siegeltheta}; this relation can be extracted from the Fock model construction of $\Theta_{KM}(\tau)$ in \cite{kudlamillsonintersection}, see also \cite[\S 7]{brfugeom}. 
More precisely, if $\mathbf z = ( z_1, \dots, z_{n-1})$ is a local  coordinate on some small open set $U \subset \calM_{\calV}(\C)$, then we may write
\[ 
	\Theta_{KM}(\tau) \ = \ \sum_{i,j} f_{ij}(\tau, \mathbf z) \, d z_i \wedge d \overline{z_j} 
\]
for a family of forms $f_{ij}(\cdot, \mathbf z) \in A_n^{!}(\calS^{\vee})$ varying smoothly in $\mathbf z$. 
Then \eqref{eqn:LowerThetaKM} asserts that at a given fixed point $\mathbf z_0 \in U$,  
\[ 
	\Low(\Theta_{KM})(\tau)|_{\mathbf z_0} \ = \  \sum_{i,j}  \  \Low  \left( f_{ij}(\cdot, \mathbf z_0) \right)(\tau)   \ d z_i \wedge d \overline{z_j}  \ = \  \ddc \Theta_L(\tau)|_{\mathbf z_0}.
\]
Moreover, it follows easily from \eqref{eqn:KMThetaFourier} that
\[
	\kappa_{\Theta_{KM}}(m) \ = \lim_{T \to \infty} c_{\Theta_{KM}}(m, T) \ = \ \begin{cases} - \ \varphi_0^\vee \cdot \Omega  & \text{ if } m = 0 \\ 0 & \text{ if } m < 0, \end{cases}
\]
where we have abbreviated $\Omega  = c_1(\widehat \omega^{\mathrm{ taut}})$.
Thus, there is a holomorphic modular form $G(\tau) \in M_{n}(\calS^{\vee})$  and an orthonormal basis of cusp forms $f_1, \dots f_t \in S_n(\calS^{\vee})$ such that
\begin{align*}
	H(\tau) \ :=& \ \Theta_{KM}(\tau) \  - \ G(\tau) \cdot \Omega \ - \ \sum_i \langle \Theta_{KM} \, - \, G \cdot \Omega, \ f_i  \rangle_{\mathrm{Pet}}^{\reg} \ f_i(\tau)  \\
	=&  \ \Low^{\sharp} \left( \ddc \Theta_L \right).
\end{align*}
Writing the Fourier expansion 
\[ 
	H(\tau) = \sum_{m} c_H(m,v) \, q^m, \qquad \text{  where } c_H(m,v) \in Z^{(1,1)}(\calM_{\calV}(\C)) \otimes_{\C} \calS^{\vee},
\]
\Cref{thm:MaassSection} implies that
\[
		c_H(m,v) \ = \ \langle P_{m,v} \, - \, F_m, \ \ddc \Theta_L \rangle^{\reg}.
\]
On the other hand, by using \Cref{cor:ArchDiffRegLift}, it can be easily shown that the $\ddc$ operator commutes with taking the regularized integral, so
\[
	c_H(m,v) \ = \ \ddc  \, \langle P_{m,v} \, - \, F_m, \ \Theta_L \rangle^{\reg} \ = \ \ddc \left( \KGr(m,v) \ - \BGr{m}) \right),
\]
so 
\[
	H(\tau) \ = \ \sum_m \ddc \left(  \KGr(m,v) \ - \ \BGr{m} \right) \ q^m
\]
on $\calM_{\calV}(\C)$; the key point here is that for  a fixed $\tau$, the generating series on the right defines a smooth $(1,1)$ form on $\calM_{\calV}(\C)$.

Now consider
\[
	H^*(\tau)  \ := \ H(\tau)  \ - \ \frac{1}{4 \pi} \sum_{B} \, \Lsharp(\Theta_{\Lambda_B})(\tau) \cdot \omega_{B}
\]
where $\omega_{B} = \ddc [\lie g_B] \ - \ \delta_{B(\C)}$ is a smooth $(1,1)$ form on the compactification $\calM^*_{\calV}(\C)$. Note that $H^*(\tau)$ is a smooth closed form on $\calM_{\calV}(\C)$. 

A rather tedious but straightforward calculation reveals that $\Theta_{KM}(\tau)$, and hence $H^*(\tau)$, is $\log\text{-}\log$ singular at the boundary, and so defines a closed current
\[
	[H^*(\tau)] \ \in \ \calD^{(1,1)}(\calM^*_{\calV}(\C))
\]
cf.\ \cite[Proposition 2.26]{bkk}; moreover, for any smooth differential form $\phi$ on $\calM_{\calV}^*(\C)$ of degree $(n-2, n-2)$, the integral 
\[
  [H^*(\tau)](\phi) = 	\int_{\calM_{\calV}^*(\C)} H^*(\tau) \wedge \phi
\]
defines a moderate growth form $[H^*(\tau)](\phi) \in \ALmod{n}(\calS^{\vee})$. 

By construction, the Fourier coefficients of $[H^*(\tau)]$ are given by
\[
		c_{H^*}(m,v) \ = \  \left[ \ddc \left( \KGr(m,v)  \  - \ \BGr{m}\ - \  \sum_B \left( \frac{\mu_B(m)}{4 \pi v} - \eta_B(m)  \right) \lie g_B \right) \right] \ = \ [\ddc g_m(v)].
\]
Since $\ddc g_m(v,z)$ is cohomologically trivial (i.e.\ $g_m(v,z)$ is a Green function for the zero cycle), it follows that for any closed differential form $\phi$ on $\calM^*_{\calV}(\C)$,
\begin{align*}
 	[H^*(\tau)](\phi) \ =  \ \int_{\calM^*_{\calV}(\C)} H^*(\tau) \wedge \phi \ =& \ \sum_m \ \int_{\calM^*_{\calV}(\C)} \ddc g_m(v,z) \wedge \phi  \ q^m \\ 
 	=& \ \sum_m \int_{\calM^*_{\calV}(\C)} g_m(v,z) \wedge \ddc \phi \ \ q^m  \\ =& \ 0.
\end{align*}

Thus, for each fixed $\tau$, the current $[H^*(\tau)]$ is closed and exact. Hence, by the $\ddc$ lemma and \cite[Theorem 2.23]{bkk}, for each fixed $\tau$ there exists a smooth function $S_0(\tau,z)$ on $\calM_{\calV}(\C)$ with $\log\text{-}\log$ singularities at the boundary, such that
\[
 \ddc [S_0(\tau,z)] \ = \ 	[\ddc S_0(\tau,z)] \ = \ [H^*(\tau)].
\]
For each connected component $X_i \subset \calM^*_{\calV}(\C)$, fix a K\"ahler form $\omega_i$ such that $\vol(X_i, \omega_i^{n-1})= 1$; if we further impose the normalization
\begin{equation} \label{eqn:StauNormalize}
	\int_{ X_i} \, S_0(\tau,z) \ \omega_i^{ n-1}  \ = \ 0 
\end{equation}
for all $i$, then $S_0(\tau,z)$ is unique. 

This normalization also forces the current $[S_0(\tau,z)]$ to transform as a modular form in $\tau$. More precisely, given a smooth form $\eta \in A^{n-1, n-1}(\calM_{\calV}^*(\C))$, use the Hodge decomposition and the $\ddc$-lemma to write $\eta = \sum a_i \omega_i^{n-1} + \ddc \phi$ for some scalars $a_i$ and a smooth form $\phi$, so that 
\begin{equation} \label{eqn:SisModGrowth}
[S_0(\tau,z)](\eta) \ = \ [H^*(\tau)](\phi) \ \in \ \ALmod{n}(\calS^{\vee}).
\end{equation}

In particular, taking the Fourier expansion $[S_0(\tau,z)] = \sum_m [c_{S_0}(m,v)] q^m$, we find
\[
	\ddc [c_{S_0}(m,v)] \ = \ [c_{H^*}(m,v)] \ = \ \ddc [g_m(v,z)] ,
\]
so for each fixed $\tau$, there is  a locally constant function  $a_m(\tau)$ on $\calM_{\calV}^*(\C)$, valued in $ \calS^{\vee}$,  such that
\[
	 [g_m(v,z)] \ = \ [c_{S_0}(m,v)]  \  + \ [ a_m(\tau)].
\]

We now show that for fixed $\tau$, the sum $ \sum_m a_m(\tau) q^m $ converges to a form of at worst exponential growth (viewed as a locally constant function). Fix a component $X_i \subset \calM_{\calV}^*(\C)$, and a smooth differential form $\eta_0$ on $X_i$ of degree $(n-1,n-1)$. For convenience, we may choose $\eta_0$ to have compact support contained in the interior $\calM_{\calV}(\C) \cap X_i$, and such that
\[
	\int_{X_i} \eta_0 \neq 0.
\]
Therefore
\[
a_m(\tau)|_{X_i} \cdot \int_{X_i} \eta_0 \ = \ [g_m(v,z)](\eta_0) - [c_{S_0}(m,v)](\eta_0).
\]
We may write (using \Cref{thm:Xitheta})
\begin{align*}
 [g_m(v,z)](\eta_0 ) =  \int_{X_i}    \langle P_{m,v} - F_m  , \Theta_L \rangle^{\reg}\, \eta_0 \ - \ \sum_B \left( \frac{\mu_B(m)}{4 \pi v} - \eta_B(m) \right) \int_{X_i} \lie g_B \, \eta_0;
\end{align*}
recall that the second factor $\frac{\mu_B(m)}{4 \pi v} - \eta_B(m)$ is the $m$'th coefficient of $\Lsharp(\Theta_{B_{\Lambda}})$, so the corresponding sum on $m$ of these terms converges.

On the other hand, as $\eta_0$ is compactly supported on $\calM_{\calV}(\C)$, an easy estimate using \Cref{cor:ArchDiffRegLift} allows us to interchange the regularization with the integral over $X_i$. 
As a consequence,  
\[
\int_{X_i}  \langle P_{m,v} - F_m , \Theta_L \rangle^{\reg}  \, \eta_0 \ = \  \langle P_{m,v} - F_m, \calI_{\eta_0}(\tau) \rangle^{\reg} \ = \ m\text{'th coefficient of } \Lsharp\left(\calI_{\eta_0}(\tau) \right),
\]
where 
\[
\calI_{\eta_0} (\tau) :=	\int_{X_i} \Theta_L(\tau,z) \, \eta_0 \ \in \ \ALmod{n}(\calS^{\vee}).
\] 
Finally, by \eqref{eqn:SisModGrowth}, we have 
\[ \sum [c_{S_0}(m,v)](\eta_0) q^m = [S_0(\tau,z)](\eta_0) \in \ALmod{n}(\calS^{\vee}), \]
and so, for our choice of $\eta_0$ as above, we may write
\[
	\sum_m \, a_m(\tau)|_{X_i} \, q^m \ = \ \left( \int_{X_i} \eta_0 \right)^{-1} \cdot \left[ \Lsharp(\calI_{\eta_0})(\tau) - \sum_{B} \Lsharp(\Theta_{B_{\Lambda}}) \, \int_{X_i} \lie g_B \, \eta_0 \ - \ [S_0(\tau,z)](\eta_0) \right].
\]
In particular, $a(\tau) = \sum_m a_m(\tau) q^m \in \ALexp{n}(\calS^{\vee})$, viewed as a locally constant function on $\calM^*_{\calV}(\C)$. 

Thus, setting
\[ S(\tau,z) \ :=\ S_0(\tau,z) + a(\tau), \]
it follows that for any smooth differential form $\eta$ on $\calM_{\calV}^*(\C)$,
\[
	\lim_{N \to \infty} \  \sum_{|m| \leq N} [g_m(\tau, v)](\eta) = [S(\tau,z)](\eta),
\]
and $S(\tau,z)$ satisfies the conditions in \Cref{def:ModularityDef}(ii).
This  concludes the proof of modularity.
\end{proof}

\begin{remark} \label{rmk:FormsOnSL2}
	
	(i) In very recent work, Bruinier-Howard-Kudla-Rapoport-Yang \cite{BHKRY1}  establish the modularity of $\ThetaVBru(\tau)$ when $n>2$; in conjunction with \Cref{thm:ArithModularity}, this implies the modularity of $\ThetaVKud(\tau)$ (see Theorem 7.4.1 of \cite{BHKRY1}). 
	
	An analogous modularity statement for divisors on (the interior of) orthogonal Shimura varieties, equipped with Bruinier's Green functions, has been established by Howard-Madapusi Pera \cite{HMP}. We expect, but have not checked the details, that an analogue of \Cref{thm:ArithModularity} holds on the open Shimura variety. A major obstacle in formulating an extension to the compactification is that the analytic behaviour of both families of Green functions near the boundary is substantially more delicate than in the unitary setting; these complications are already evident in work by Berndt-K\"uhn \cite{kuehn-berndt-split}, which studied Kudla's Green functions on the Shimura variety attached to the quadratic space $(M_2(\Q),\det)$.
	
	(ii) Strictly speaking, we should view $\ThetaVKud(\tau) - \ThetaVBru(\tau)$ as a  form for the group $ U(1,1)$, in order to be consistent with the philosophy that it defines an arithmetic analogue of the theta correspondence for the dual pair $(U(1,1), U(V))$. 
	
	Let $V_0$ be the standard split Hermitian space over $k$ of signature (1,1) and let $\mathbf G = U(V_0)$ viewed as a group over $\Q$. Upon choosing a basis so that $V_0$ has the Hermitian form $(\begin{smallmatrix} & \delta_\kb \\ - \delta_\kb & \end{smallmatrix})$, it is easily seen that 
	\[ 1 \ \to \ \SL_2 \ \to \ \mathbf G \ \to \ U(1) \ \to \ 1 \]
	is an exact sequence of algebraic groups over $\Q$, where $U(1)(R) \ = \ \{ x \in R \otimes_\Q \kb \, | \, x \bar x = 1 \}$ for any $\Q$-algebra $R$. As we now explain, modular forms on $\mathbf G$ of the type we are considering are determined by their restriction to $\SL_2$.
	
	The choice of an idele class character $\eta \colon \mathbb A_k^{\times} / \kb^{\times} \ \to \ \C^{\times}$ 
	induces  a splitting of the metaplectic cover of $\mathbf G$, and in particular determines a unitary Weil representation $\rho_{\mathbf G} = \rho_{\mathbf G, \eta}$ of $\mathbf G(\mathbb A)$ on $S(V(\mathbb A))$, see \cite{Harris-Kudla-Sweet}. If we further assume that $\eta|_{\mathbb A_{\Q}^{\times}} = (\chi_k)^{\dim V}$ where $\chi_k$ is the character attached to $\kb/\Q$, then the restriction of $\rho_{\mathbf G}$ to $\SL_2(\mathbb A)$ coincides with the Weil representation for the dual pair $(\SL_2, O(V))$.
	Also note that the choice of basis for $V_0$ above gives an integral structure for $\mathbf G$, and $\mathcal S \subset S(V(\mathbb A_f))$ is stable under the action of $\mathbf G(\widehat\Z)$.

	Consider the maximal compact subgroup
	\[ K_{\infty} \ = \  \left( U(1) \times U(1) \right) (\R) \ \hookrightarrow \ \mathbf G(\R), \qquad \ (e^{i \theta}, e^{i \varphi}) \ \mapsto \ e^{i\theta} \begin{pmatrix} \cos \varphi & - \sin \varphi \\ \sin \varphi & \cos \varphi \end{pmatrix}; \]
	there is a bijection $\uhp \isomto \mathbf G(\R) / K_{\infty}$ that identifies $\tau = u + i v$ with 
	\[ g_{\tau}\ : = \ \begin{pmatrix} v^{1/2} & v^{-1/2} u \\ & v^{-1/2} \end{pmatrix}. \]
	
	Now suppose $F$ is an automorphic form for $\mathbf G$ of weight $n$, whose $K_f$-type is $(\rho_{\mathbf G}^{\vee}, \mathcal S^{\vee})$,  and with central character $\eta$. In other words, $F \colon \mathbf G(\Q) \backslash \mathbf G(\mathbb A) \to \mathcal S^{\vee}$ is a function satisfying
	\begin{enumerate}
		\item $ F(z \cdot g)  \ = \ \eta(z) F(g)$ for all $z \in Z(\mathbb A)$;
		\item  $F (g \cdot k_{\infty} )  \ =\ \eta_{\infty}(e^{i\theta})  \ e^{-i n \varphi} \ F(g)$ for all $k_{\infty} = (e^{i \theta}, e^{i \varphi}) \in K_{\infty}$; and
		\item $ F(g\cdot k_f) \ = \ \rho_{\mathbf G}^{\vee}(k_f^{-1}) F(g) $ for all $k_f \in \mathbf G(\widehat\Z)$,
	\end{enumerate}
	along with certain analytic conditions that we ignore. 
	If we define a map $f \colon \uhp \to \mathcal S^{\vee}$ by the formula 
	\[ f(\tau) \ =  v^{-n/2} F(g_{\tau}), \]
	then the conditions (i)--(iii) imply that $f$ satisfies the usual transformation law
	\begin{equation} \label{eqn:SL2translaw} f(\gamma \tau) \ = (c \tau + d)^n \, \rho^{\vee}(\gamma) \, f (\tau) 
	\qquad \text{for all } \gamma \in \mathbf G(\widehat \Z) \cap \mathbf G(\Q) \ = \ \SL_2(\Z).
	\end{equation}
	In other words, $f$ transforms as  a vector-valued modular form for $\SL_2(\Z)$ in the usual sense.
	
	Conversely, the fact that there is a single genus of self-dual lattices in $V_0$ implies that
	\[ \mathbf G(\mathbb A) \ =  \ \mathbf G(\Q) \cdot \left( \mathbf G(\R) \times \mathbf G(\widehat\Z)\right); \]
	a straightforward, though somewhat tedious, verification then implies that any function $f$ for $\SL_2$ satisfying \eqref{eqn:SL2translaw} determines an automorphic function $F$ on $\mathbf G$ satisfying (i) - (iii) above.

\end{remark}

\subsection{Holomorphic projection}
\label{sec:holom-proj}
Suppose $\widehat\phi(\tau)$ is modular of weight $k$ in the sense of \Cref{def:ModularityDef}, so that we may write
\[ \widehat\phi(\tau) \ = \ \sum_i f_i(\tau) \otimes \widehat\Zed_i  \ + \ (0, g(\tau, z)) \]
for some  forms $f_i(\tau)$ and $g(\tau, z)$ and arithmetic classes $\widehat\Zed_i$. If $G \in S_k( \mathcal S^{\vee})$, we define a Petersson pairing
\[ \langle \widehat\phi(\tau), G \rangle_{\mathrm{Pet}} \ := \  \sum_i \langle f_i, G \rangle_{\mathrm{Pet}}  \cdot \widehat\Zed_i \ + \ \left(0, \, \langle g(\cdot, z), G \rangle_{\mathrm{Pet}}  \right) \ \in \ \ChowHatC(\M_V^*) \]
where on the right hand side, we have the usual Petersson pairing between modular forms, and provided that all these latter pairings exist. We may relax this definition to the case where the Petersson pairings exist in the regularized sense, as in \Cref{lem:regPet}. 
\begin{theorem} \label{thm:arithHolProj} For every $G \in S_n(\mathcal S^{\vee})$, 
\[ \langle \ThetaVKud - \ThetaVBru, \, G \rangle^{\reg}_{\mathrm{Pet}} \ = \ 0 .\]
In particular, the cuspidal holomorphic projections of $\ThetaVKud(\tau)$ and $\ThetaVBru(\tau)$ coincide.
\begin{proof}
Write 
\[ \ThetaVKud(\tau) - \ThetaVBru(\tau) \ = \ - \frac{1}{4\pi} \sum_B \, [\widehat B] \otimes  \Lsharp( \Theta_{\Lambda_B})(\tau)     \ + \ (0, g(\tau,z)). \]
By construction of the section $\Lsharp$, the terms $\Lsharp(\Theta_{\Lambda_B})$ are orthogonal to cusp forms, cf. \Cref{prop:uniquepreimage}.
On the other hand, for a fixed $z \in \domain(\calV)$, we have already seen that 
\[ g(\tau,z) =  - \Lsharp\left(\Theta_L(\cdot,z)\right) (\tau) \]
for the Siegel theta function $\Theta_L(\tau,z)$, and is therefore also orthogonal to cusp forms.
\end{proof}

\end{theorem}

\subsection{The arithmetic height generating series} \label{sec:arithHeights}
A particularly striking aspect of Kudla's programme is a set of conjectures relating certain  height pairings involving  arithmetic special cycles and special values of derivatives of Eisenstein series, see \cite{kudla-special-eisenstein-imrn} for an overview; these conjectures are generalizations of the Gross-Zagier theorem \cite{grosszagier} to higher dimensional Shimura varieties.

Suppose $\calX^*$ is an $n$-dimensional arithmetic variety over $o_{\kb}$, i.e. a regular scheme,  flat and proper over $o_{\kb}$, and suppose furthermore that $\calX \subset \calX^*$ is a dense open subvariety such that the boundary $\partial \calX^* = \calX^* - \calX$ is a divisor. As described in \cite{burgos-kramer-kuhn}, its arithmetic Chow groups (with $\log\text{-}\log$ singular Green objects along $\partial \calX^*$) are equipped with an intersection product
\[  {\widehat{\mathsf{CH}}{}^{p}_{\C}}(\calX^*)  \ \times  \ {\widehat{\mathsf{CH}}{}^{q}_{\C}}(\calX^*) \ \to \  {\widehat{\mathsf{CH}}{}^{p+q}_{\C}(\calX^*)} \]
and a pairing
\[ \ChowHatC(\calX^*) \times {\widehat{\mathsf{CH}}{}^{n-1}_{\C}}(\calX^*) \to \C, \qquad [\widehat\Zed_1 : \widehat\Zed_2] \ := \ \widehat\deg \, \left(\widehat\Zed_1 \cdot \widehat\Zed_2 \right) \]
that generalize the structures constructed by Gillet-Soul\'e \cite{gillet-soule} to the $\log\text{-}\log$ singular setting. As a special case, if $(0,g) \in \ChowHatC(\calX^*)$ and $\widehat\calL$ is a metrized line bundle, then
\begin{equation} \label{eqn:arithIntGreenHeight}
	\left[(0,g) : \widehat\calL^{n-1}\right] \ = \ \frac12 \sum_{\sigma \colon k \to \C} \ \int\limits_{\calX^*(\C_{\sigma})}  g \cdot c_1(\widehat\calL)^{\wedge n-1} .
\end{equation}
The authors are unaware of an arithmetic intersection theory applicable to Deligne-Mumford stacks, such as $\calM_{\calV}^*$, that also allows for $\log\text{-}\log$ singular Green functions. However, as this technicality is tangential to the main thrust of the present work, we will carry on assuming that such a generalization exists:

\begin{hypothesis} \label{hyp:arithInt}
There are well-defined products $ {\widehat{\mathsf{CH}}{}^{p}_{\C}}(\M_V^*)  \ \times  \ {\widehat{\mathsf{CH}}{}^{q}_{\C}}(\M_V^*) \ \to \  {\widehat{\mathsf{CH}}{}^{p+q}_{\C}(\M_V^*)}$ and a pairing
\[ \ChowHatC(\M_{\calV}^*) \times {\widehat{\mathsf{CH}}{}^{n-1}_{\C}}(\M_{\calV}^*) \to \C, \qquad [\widehat\Zed_1 : \widehat\Zed_2] \ := \ \widehat\deg \, \left(\widehat\Zed_1 \cdot \widehat\Zed_2 \right) \]
such that \eqref{eqn:arithIntGreenHeight} continues to hold.
\end{hypothesis}

With this assumption in place, we describe the content of Kudla's conjecture. Consider the generating series
\[ [\ThetaVKud(\tau): \widehat\omega^{n-1}] \ := \ \sum_{m}  \, [\ZedVKud(m,v) :  \widehat\omega^{n-1}] \,  q^m 
\] 
whose terms are obtained by  pairing the special cycles against a power of the tautological bundle $\widehat\omega = \widehat\omega^{\mathsf{taut}}$. Form a generating series $[ \ThetaVBru(\tau) : \widehat\omega^{n-1}]$ in the same way. 

On the other hand, consider the family of  Eisenstein series  $E_k(\tau,s)$, which are $\mathcal S^{\vee}$-valued  forms defined as follows: for a weight $k \in \Z$ and a complex parameter $s \in \C$, define
\[ E_k(\tau,s) \ = \ \sum_{\gamma \in \Gamma_{\infty} \bs \SL_2(\Z) } \left( v^{\frac{s+1 - k}{2} } \ \varphi_{\ov{0}, o_\kb}^{\vee} \right)\mid_k [\gamma]. \]
 This sum defines a holomorphic function for $\Re(s)>1$, and admits a meromorphic extension to $\C$. 
Strictly speaking, in what follows we should view $E_k(\tau,s)$ as a  form for $U(1,1)$, but \Cref{rmk:FormsOnSL2} applies here as well.

Of particular interest is the special value at $s = n-1$ of the derivative
	\[
		E'_{n}(\tau, n-1) \ := \ \left[ \frac{\partial}{\partial s} E_n(\tau, s) \right]_{s = n-1} 
	\]
of the weight $n$ Eisenstein series.

\begin{conjecture}[Kudla] \label{conj:KudlaArVol}
Up to some correction terms,
	\[ 
		[\ThetaVKud(\tau): \widehat\omega^{n-1}] \ \stackrel{?}{=}  \  2 \, \kappa \,   E_n'(\tau, n-1) \ + \ \left( \begin{matrix} \text{`vertical'}\\ \text{correction terms} \end{matrix}\right) \ + \ \left( \begin{matrix} \text{'boundary'} \\ \text{correction terms} \end{matrix} \right) 
	\]
where 
	\[ \kappa \ =\ \vol(\M_{\calV}^*(\C_{\sigma}), \, \mathrm d \Omega)  \ := \ \int_{\M_{\calV}^*(\C)} \, \mathrm d \Omega 
		\]
	is the `stacky' volume\footnote{The reader is cautioned that by our definition, $\kappa$ may be positive or negative according to the parity of $n$.} of $\M_{\calV}^*(\C_{\sigma})$ with respect to the measure $\mathrm d \Omega = c_1(\widehat \omega)^{\wedge n-1}$ determined by $\widehat \omega$.   

\end{conjecture}
Though these correction terms have not been explicitly spelled out in the literature, experience from low-dimensional settings, e.g.\ \cite{kry-special} suggest that  holomorphic Eisenstein series  encoding contributions arising from primes of bad reduction should appear in the formula, and \Cref{thm:DiffHeight} below suggests that contributions from the  boundary  also play a role.

The main results of this section concern the difference  $\ThetaVKud(\tau) - \ThetaVBru(\tau)$, which we write as
\begin{equation}\label{eqn:arithDiff}
\ThetaVKud(\tau) - \ThetaVBru(\tau) \ = \ - \frac{1}{4 \pi} \sum_{B \in \mathcal B_{\calV}} \, \Lsharp \Theta_{\Lambda_B}(\tau) \otimes \widehat{[B]} \ + \ (0, g(\tau, z)) 
\end{equation}
where $g(\tau,z)$ is as in the proof of \Cref{thm:ArithModularity}.
The following theorem, which would form the archimedean component of a putative intersection pairing, shows that the integral of $g(\tau,z)$ already contributes the `main term' in \Cref{conj:KudlaArVol}. We exclude the case $n=2$  because the proof invokes the Siegel-Weil formula; we expect, though we have not checked the details, that a regularized version of the Siegel-Weil formula can be employed to prove the $n=2$ case.
\begin{theorem} \label{thm:GreenInt}
Suppose $n>2$. Then
	\[ 
	 	\frac12 \sum_{\sigma \colon k \to \C} \ \int\limits_{\M_{\calV}^*(\C_{\sigma})} g(\tau, z) \cdot c_1(\widehat\omega)^{n-1} \ = \    2 \, \kappa \,   E'_n(\tau, s_0) .
	\]
\end{theorem}

\begin{proof}
Write $g(\tau,z) = \sum g_m(\tau,z) q^m$, where
\[
	g_m(\tau,z) \ = \ \KGr(m,v)(z) \ - \ \BGr{m}(z)  \ - \ \delta_{m,0} \varphi^{\vee}_{\ov{0}, o_\kb} \, \log v \ - \ \sum_{B}\, c(m,v) \, \lie g_B(z),
\]
and $c(m,v)$ is the $m$'th Fourier coefficient of $\frac{-1}{4 \pi} \Lsharp\Theta_{\Lambda_B}(\tau)$. 

Moreover, by assumption,
\[ \int_{\M_{\calV}^*(\C_{\sigma})}  \lie g_B \, d \Omega = 0\]
for each embedding $\sigma \colon \kb \to \C$, and so we only need to compute the integrals of $\KGr(m,v)$ and $\BGr{m}$; here we abbreviated $d \Omega = c_1(\widehat\omega)^{n-1}$. 

We begin by recalling that the Green functions were obtained by the regularized integral
\[ 
\KGr(m,v) - \BGr{m} - \delta_{m,0}\, \varphi^{\vee}_{\ov{0}, o_\kb} \log v = \langle P_{m,v} - F_m, \, \Theta_L \rangle^{\reg}.
\]
By \Cref{lem:IntGFTheta} below, whose proof we defer momentarily, we may interchange the integral on $\M^*_{\calV}(\C_{\sigma})$ with the regularized pairing above to obtain
\[ 
  \int_{\M^*_{\calV}(\C_{\sigma}) } \KGr(m,v) - \BGr{m} \ d \Omega 
  - \vol(\M_{\calV}^*(\C_{\sigma}))\, \log(v)\, \delta_{m,0}\, \varphi_{0,o_\kb}^{\vee}  \ = \ \langle P_{m,v} \ - \ F_m  , \, I(\Theta_L) \rangle^{\reg} 
\]
where 
\[ 
  I(\Theta_L)(\tau) \ := \  \int_{\M^*_{\calV}(\C_{\sigma}) } \Theta_L(\tau,z) \, d \Omega. 
\]
Thus
\[ 
 I(g)(\tau) :=	\int_{\M^*_{\calV}(\C_{\sigma}) } g(\tau,z) \ d \Omega \ = \ \sum_{m} \langle P_{m,v}-  F_m  , \, I(\Theta_L) \rangle^{\reg} \, q^m.
\]
Then, by \Cref{cor:vectorValLSec}, 
\[ I(g)(\tau) \ = \ -  \Lsharp \, I(\Theta_L) ; \]
to conclude the proof, we need to identify $\Lsharp I(\Theta_L)$ as the derivative of the Eisenstein series. 

The \emph{unitary Siegel-Weil formula}  \cite{ichino-SW} gives
	\begin{equation} \label{eqn:SeigelWeil}
			I(\Theta_L)(\tau) \ = \ \kappa  \  E_{n-2}(\tau, n-1). 
	\end{equation} 

A direct computation, most easily carried out in the ad\`elic framework as in \cite[(2.17)]{kudla-integrals}, implies 
	\begin{equation} \label{eqn:LEisenstein}
		\Low( E_k(\tau, s) ) \ = \ - \frac12 (s - (k-1)) \,  E_{k-2}(\tau, s);
	\end{equation}
the Eisenstein series $E_{n-2}(\tau,s)$ is holomorphic at $s=n-1$, and so
	\[
		\Low(E'_k(\tau, n-1)) \ = \ - \frac12 \, E_{n-2}(\tau, n-1) .
	\]
Moreover, by \cite[Theorem 2.11]{kudla-integrals}, the principal part of $E'_n(\tau, n-1)$ vanishes, and $E'_n(\tau,n-1)$ can easily be seen to be orthogonal to cusp forms, by unfolding. 
Therefore, by the uniqueness statement in \Cref{prop:uniquepreimage}, we have identified
\[ 
	2 \, \kappa \, E'_n(\tau, n-1) \ = \ - \Lsharp(I(\Theta_L)) \ =  \ I(g)(\tau).
\]
Since our original formula involves half the sum over both embeddings $\kb \hookrightarrow \C$, the theorem follows.
\end{proof}

As a consequence of the previous theorem and \eqref{eqn:arithDiff}, we obtain the following (putative) intersection formula: 
\begin{corollary} \label{thm:DiffHeight} Suppose $n>2$. Assuming the existence of a suitable intersection theory on $\calM_{\calV}^*$, as in \Cref{hyp:arithInt}, we have
	\[ 
		[\ThetaVKud(\tau) - \ThetaVBru(\tau): \widehat{\omega}^{n-1}] \ = \ 2 \, \kappa \,  E_n'(\tau, n-1) \  - \frac{1}{4 \pi} \ \sum_{B \in \mathcal B_{\calV}} \Lsharp \Theta_{\Lambda_B}(\tau) \cdot \left[ \widehat{[B]} :\widehat \omega^{n-1} \right].
	\]
	\qed
\end{corollary}

It remains to prove the following lemma: 
\begin{lemma} \label{lem:IntGFTheta}
Suppose  $n>2$. 
For any $m \in \Q$, we have
	\[ 
       \int_{\M^*_{\calV}(\C)} (\KGr(m,w) - \BGr{m})\, d\Omega 
       - \vol(\M_{\calV}^*(\C))\, \log(w)\, \delta_{m,0}\, \varphi_{0,o_\kb}^{\vee} = \langle P_{m,v}-F_{m}, I(\Theta_L) \rangle^{\reg}.
       \]
\end{lemma}
\begin{proof}
We have by \Cref{cor:ArchDiffRegLift} that
\begin{equation*}
  		\KGr(m,w)(z) -\BGr{m} - \log(w)\, \delta_{m,0}\, \varphi_0^{\vee}\ 
                = \  \lim_{T \to \infty} \left(\int_{\calF_T} \calF_{m,w}(\tau) \, \Theta_L(\tau,z) \,  d\mu(\tau) \ -A(T) \right),
\end{equation*}
with $A(T) = (\delta_{m,0}\, \varphi_0^{\vee} - \varphi_0^{\vee} \cdot c_{F_m}(0)) \, \log(T)$, and $\calF_{m,w}(\tau) = P_{m,w}(\tau) - F_m(\tau)$.
Our goal is to prove
\begin{align} 
	&\int_{\M_{\calV}^*(\C) } \ (\KGr(m,w) -\BGr{m})\ d \, \Omega \ - \vol(\M_{\calV}^*(\C))\, \log(w)\, \delta_{m,0}\, \varphi_0^{\vee}\ \notag\\
      =\ &\int_{\M_{\calV}^*(\C) } \left(\lim_{T \to \infty} \, \int_{\calF_T} \calF_{m,w}(\tau) \, \Theta_L(\tau,z)    \,  d\mu(\tau)-A(T)\right) \ d \Omega  \notag \\
		\label{eqn:interchangeLemma}
	\stackrel{?}{=} \  &\lim_{T \to \infty}\Bigg( \, \int_{\calF_T}\calF_{m,w}(\tau) \,  \Big( \int _{\M_{\calV}^*(\C)}\Theta_L(\tau,z) d \Omega \Big)  \,d\mu(\tau)
-A(T)\vol(\M_\calV^*(\C)) \Bigg), \notag
\end{align}
i.e., we need to justify the interchange of integral and limit in the third line.
Write
\begin{align*}
	\int_{\calF_T} \calF_{m,w}(\tau) \,  & \Theta_L(\tau,z)    \,  d\mu(\tau) \\	
      =& \ \int_{\calF_{w_0}} \, \calF_{m,w}(\tau) \, \Theta_L(\tau,z) \, d \mu(\tau)  \ + \ \int_{\calF_T - \calF_{v_0}} \calF_{m,w}(\tau) \, \Theta_L(\tau,z)   \,  d\mu(\tau), 
\end{align*}
where $w_0 := \max(w, 1/w)$; since $\Theta_L(\tau,z)$ is absolutely integrable on $M_\calV(\C)$ and $\calF_{w_0}$ is compact,
\begin{align*}
   \int_{\M_{\calV}^*(\C)} \int_{\calF_{w_0}} \, \calF_{m,w}(\tau) \, \Theta_L(\tau,z) \, d \mu(\tau) \, d \Omega \ 
  =& \ \int_{\calF_{w_0}} \int_{\M_{\calV}^*(\C)} \, \calF_{m,w}(\tau) \, \Theta_L(\tau,z) \,  d \Omega \, d \mu(\tau) \\ 
  =&  \  \int_{\calF_{w_0}} \,\calF_{m,w}(\tau) \, I( \Theta_L)(\tau) \, d \mu(\tau).
\end{align*}
Therefore, in order to justify interchanging integral and limit, we need to prove that
\begin{equation*}
  \abs{\int_{\calF_T - \calF_{w_0}} \, \calF_{m,w}(\tau) \cdot \Theta_L(\tau,z) \, d \mu(\tau) -A(T)},
\end{equation*}
evaluated at any basis element $\varphi_{\ov{m},\fraka}$,
is bounded by an integrable function, uniformly in $T$.
We let 
\[
  \calG_{m,w}(\tau) = \calF_{m,w}(\tau) - \delta_{m,0}\sum_{\fraka \mid \diffinv{\kb}} \varphi_{0,\fraka}^\vee \otimes \varphi_{0,\fraka} + c_{F_m}(0),
\]
which is of exponential decay; i.e. there is a constant $C>0$, s.t. $\abs{\calG_{m,w}(\tau)(\varphi_{\ov{m},\fraka})} \leq e^{-Cv}$ for all $\tau \in \calF$.
We obtain that 
\begin{align}
  \abs{\int_{\calF_T - \calF_{w_0}} \, \calG_{m,w}(\tau) \Theta_L(\tau,z) \, d \mu(\tau)} 
    &\leq  \int_{\calF_T - \calF_{w_0}} \, e^{-Cv}\abs{\Theta_L(\tau,z)} \, d \mu(\tau) \notag\\
    &\leq \int_{\calF} \, e^{-Cv}\Theta_L(iv,z) \, d \mu(\tau), \label{eq:expTheta_Liv}
\end{align}
where we understand the inequalities again as valid after evaluating at any basis element $\varphi_{\ov{m},\fraka}$.
The function in \Cref{eq:expTheta_Liv} is integrable since, by the Siegel-Weil formula
\[
      \int_{\calF} \, e^{-Cv}\int_{\M_{\calV}^*(\C)}\Theta_L(iv,z)\, d \Omega \, d \mu(\tau) 
  = \kappa\, \int_{\calF} \, e^{-Cv}E_{n-2}(iv, n-1) \, d \mu(\tau),
\]
which is finite because the integrand is bounded on $\calF$ and $\calF$ has finite volume.

The remaining term gives for $\varphi = \varphi_{\ov{m},\fraka}$,
\begin{gather*}
  \abs{\int_{\calF_T - \calF_{w_0}} \, (\delta_{m,0}\varphi_{0,\fraka} - c_{F_{m,\varphi}}(0)) \cdot \Theta_L(\tau,z) \, d \mu(\tau) -A(T)} \\ 
\\ = \abs{\int_{w_0}^T \, \sum_{\substack{\lambda \in \fraka^{-1}\calL \\ Q(\lambda)=0}} (\delta_{m,0}\varphi_{0,\fraka}(\lambda) - c_{F_{m, \varphi}}(0)(\lambda))e^{-4\pi R(\lambda,z)v} \frac{dv}{v} -A(T)} \\
\leq \abs{\KGr(0,w_0)(\delta_{m,0}\varphi_{0,\fraka} - c_{F_{m, \varphi}}(0))} \,  + \delta_{\fraka, o_\kb}\abs{\delta_{m,0} - c_{F_{m,\varphi}}(0)(0)} \log(w_0).
\end{gather*}
Since $\KGr(0,w_0)$ is a Green function (for the zero-cycle), it is integrable and this finishes the proof of the lemma.
\end{proof}

\subsection{A refined Bruinier-Howard-Yang theorem} \label{sec:refBHY}
As a final application of our results, we describe a refined version of the main theorem of \cite{bruinier-howard-yang-unitary}; we briefly recall the setup.

For an integer $m>0$, let $\mathcal M_{(m,0)}$ denote the moduli stack over $\Spec(o_\kb)$ whose $S$-points comprise the category of triples $\underline A = (A,i, \lambda)$, where:
 \begin{enumerate}
 \item $A$ is an abelian scheme of dimension $m$ over $S$
 \item $i \colon o_\kb \to \End(A)$ is an $o_\kb$-action such that the induced action on $\Lie(A)$ coincides with the structural morphism $o_\kb \to \calO_S$;
 \item $\lambda$ is a principal polarization such that the induced Rosati involution coincides with Galois conjugation on the image $i(o_\kb)$.
 \end{enumerate}
Similarly, we define the moduli stack $\calM_{(0,m)}$ to be the moduli space of triples $\underline A = (A, i, \lambda)$ as above, except that the $o_\kb$ action on $\Lie(A)$ is required to coincide with the conjugate of the structural map. 

By \cite[Proposition 2.1.2]{howard-unitary-2}, the spaces $\calM_{(m,0)}$ and $\calM_{(0,m)}$ are proper and \'etale over $\Spec(o_\kb)$, and therefore the same is true for
\[ \mathcal Y \ := \ \calM_{(1,0)} \ \times_{o_\kb} \ \calM_{(0,1)} \ \times_{o_\kb} \ \calM_{(n-1,0)}. \]
For a fixed  self-dual Hermitian $o_\kb$-lattice $\Lambda$ of signature $(n-1,0)$, consider the substack $\calY_{\Lambda} \subset \calY$ defined as the locus of triples $(\underline E_0, \underline E_1, \underline B)$ such that 
\[ \Hom_{o_\kb}(E_0, B) \ \simeq \ \Lambda ; \]
here we view $\Hom_{o_\kb}(E_0, B)$ as a Hermitian lattice via the formula \eqref{eqn:defHomHerm}. Note that this substack is merely a union of connected components of $\mathcal Y$, cf.\ \cite[Proposition 5.2]{bruinier-howard-yang-unitary}

There is a morphism (the \emph{small CM cycle} in the terminology of \emph{op.~cit.})
\[  \mathcal Y_{\Lambda}  \ \to   \ \mathcal M = \calM_{(1,0)} \times \calM^{\text{Kr\"a}}_{(n-1,1)}    \]
defined, at the level of moduli, by sending a point $ ( \underline A_0,  \underline A_1, \underline B) $ of $\calY$ 
 to the tuple 
 \[ (\underline A_0, \underline{A_1 \times B}, \Lie(B)); \] 
 here the $n$-dimensional abelian variety $A_1 \times B$ is equipped with the product $o_\kb$-action and polarization, and we view $\Lie(B) \subset \Lie(A_1 \times B)$ as a subsheaf satisfying Kr\"amer's condition as in \Cref{sec:UModProb}.

Consider the restriction
\[ \mathbf y  \colon \ \calY_{\calV,\Lambda} \ \to \ \calM_{\calV}, \qquad \calY_{\calV, \Lambda} := \ \calY_{\Lambda} \times_{\calM} \calM_{\calV}, \]
where $\calM_{\calV}$ is a component in the decomposition \eqref{eqn:MDecomp} indexed by a Hermitian vector space $\calV$ of signature $(n-1,1)$; this restriction is non-empty precisely when there exists an isometric embedding $\Lambda \subset \calV$. 

For the remainder of this section,  fix self-dual Hermitian lattices $\Lambda_0$ and $\Lambda$ of signature $(0,1)$ and $(n-1,0)$ respectively, and set 
\[ \calL \ := \ \Lambda_0 \oplus \Lambda \qquad \text{and} \qquad  \calV := \calL \otimes_{\Z} \Q. \]
Thus, by taking the pullback of an arithmetic divisor along the composition 
\[ \calY_{\calV, \Lambda} \to \calM_{\calV} \subset \calM_{\calV}^*\] 
and then applying the arithmetic degree map $\deghat \colon \ChowHatC(\calY_{\calV, \Lambda}) \to \C$, we obtain the linear functional
\[ [ \bullet : \calY_{\calV, \Lambda} ] \colon \ChowHatC(\calM^*_{\calV})  \to \C, \qquad \qquad [\widehat\calZ:\calY_{\calV,\Lambda}] \ = \ \deghat \,  (\mathbf y^* \widehat\calZ). \]
Note that since $\calY$ is proper, the subtleties involving Green functions that are $\log\text{-}\log$-singular at the boundary do not play a significant role, and this height pairing is indeed well-defined; see \cite[\S 3.1]{howard-unitary-2} for a more careful treatment of this point.

The main result of \cite{bruinier-howard-yang-unitary} is a formula, originally conjectured by Bruinier-Yang \cite{bryfaltings}, relating the quantities $[\ZedVBru(m):\calY_{\calV,\Lambda}]$ to the special values of the derivative of a \emph{Rankin-Selberg convolution $L$-function}, defined as follows. Consider the $S(\Lambda(\bbA_f))^{\vee}$-valued theta function
\[ \Theta_{\Lambda}(\tau) \colon \varphi \ \mapsto  \ \sum_{\lambda \in \Lambda_{\Q}} \varphi(\lambda) \, q^{(\lambda, \lambda)} \qquad \text{for } \varphi \in S(\Lambda(\bbA_f)). \]

For $g \in S_{n}(\mathcal S^{\vee})$ with Fourier coefficients $ c_g(m) \in \calS^{\vee}$, recall that we may view $\overline{c_g(m)} \in \calS$ as in \Cref{eqn:WeilRepDualConj}, and consider the Rankin-Selberg $L$-function
\[ L(s, g, \Theta_{\Lambda}) \ := \ \frac{\Gamma(s/2 + n -1)}{(4 \pi )^{s/2 + n-1}} \, \sum_{m>0} \frac{ \left(\varphi_0^\vee \otimes c_{\Theta_{\Lambda}}(m) \right) \cdot \overline{c_g(m)} }{m^{s/2 + n -1}}. \]

In the following theorem, we restrict to the case $n>2$ for convenience. As a consequence, the forms $F_m$ that were constructed in \Cref{sec:relat-non-holom}, and were used to define the cycles $\ZedVBru(m)$, coincide with the weak Maa\ss\ forms described in \cite[Lemma 3.10]{bruinier-howard-yang-unitary}. Note in particular  $F_m = 0$ for $m \leq 0$.
\begin{theorem}[{Bruinier-Howard-Yang, \cite[Theorem A]{bruinier-howard-yang-unitary}}] \label{thm:BHY} Suppose $n>2$. 
Let $f \in H_{2-n}(\calS)$. Then
\[ 
	\sum_{m \geq 0} c_f^+(-m) \cdot [\ZedVBru(m) :  \calY_{\calV, \Lambda}]  \ = \  - \deg_{\C}(\calY_{ \calV, \Lambda}(\C)) \cdot L'(0, \xi(f), \Theta_{\Lambda}).
\]
where $c_f^+(-m) \in \calS $ are the coefficients of the holomorphic part of $f$, cf.\  \eqref{eqn:weakMaassDecomp}.

\end{theorem}
This formula follows from combining the geometric contributions (whose determination forms the bulk of \cite{bruinier-howard-yang-unitary}) with the results of \cite{bryfaltings} detailing the contributions arising from the Green functions $\BGr{m}$; if we instead consider Kudla's Green functions $\KGr(m,v)$, we arrive at the following refinement.

For each weight $k$, there is an Eisenstein series
\[ 
  E_{k, \Lambda_{0}}(\tau,s) \ := \ \sum_{\gamma \in \Gamma_{\infty} \backslash \SL_2(\Z) }  \left( v^{s/2+k-1} \varphi_0^\vee \right) \Big\vert_{k} [\gamma] \ 
                                                = \ \sum_{\gamma \in \Gamma_{\infty} \backslash \SL_2(\Z) } \frac{\Im(\gamma \tau)^{s/2+k-1}}{(c\tau + d)^k}  \, \rho_{\Lambda_0}^{\vee}(\gamma)^{-1} \cdot \varphi_0^\vee  
\]
valued in $S(\Lambda_{0,\bbA_f})^{\vee}$. 
We may then view
\[ 
  E_{k,\Lambda_0}(\tau,s) \otimes \Theta_{\Lambda}(\tau) 
\]
as an $\calS^{\vee}$-valued modular form, of weight $k+n-1$, by pulling back via
\[ \calS \ \subset \  S(\calL) \ \simeq \ S(\Lambda_0)  \otimes_{\C}  S(\Lambda)\  \subset \ S(\Lambda_{0,\bbA_f}) \otimes_{\C} S(\Lambda(\bbA_f)). \]

\begin{theorem} \label{thm:refBHY} For any $n \geq 1$, consider the $\calS^{\vee}$-valued generating series 
\[ 
  [\ThetaVKud(\tau): \calY]  \ := \ \sum_m \  [\ZedVKud(m,v):\calY] \ q^m, 
\]
where we have abbreviated $\calY = \calY_{\calV, \Lambda}$.  Then
\[ 
    [\ThetaVKud(\tau) : \calY] = - \deg_{\C}(\calY(\C)) \ E_{1,\Lambda_0}'(\tau,0) \otimes \Theta_{\Lambda}(\tau). 
\]
\begin{proof}
	First suppose $m \in \Q$ with $m \neq 0$, and decompose
	\[ 
		[ \ZedVKud(m,v) : \calY ] \  = \  	[ \Zed(m) : \calY]_{\mathsf{fin}} \ + \ \KGr(m,v)(\calY(\C)) 
	\]
	into the geometric and archimedean contributions, the latter given by the weighted sum of the values of $\KGr(m,v)$ at the points comprising  $\calY(\C)\subset\calM_{\calV}(\C)$. 
	Note that as $\calY $ is disjoint from the boundary $\partial \calM^*_{\calV}$, the boundary components of $\ZedVKud(m,v)$ play no role. 
	
	Recall the complex uniformization
	\[ 
		\calM_{\calV}(\C) \  = \ \coprod_{[\calL_0, \calL_1]} \ \left[  \Gamma_{[\calL_0, \calL_1]} \ \big\backslash \ \domain(\calV) \right] 
	\]
	as in \eqref{eqn:MCpxUnif}, where the union is taken over isomorphism classes of pairs $(\calL_0, \calL_1)$ of self-dual Hermitian lattices of signature $(1,0)$ and $(n-1,1)$ respectively, and such that $\calV \simeq \Hom(\calL_0, \calL_1)_{\Q}$. This can also be expressed as follows: note that for a fixed $L_0$, the collection of lattices $\{ \Hom(\calL_0, \calL_1)\}$ obtained by varying $\calL_1$ is simply the set of isomorphism classes of self-dual lattices $\{ \calL' \subset \calV \}$. Since $d_\kb$ is odd, there is a single genus of such lattices \cite{jacobowitz}, and so  \eqref{eqn:MCpxUnif} may be rewritten as 
	\begin{equation} \label{eqn:MNewCpxUnif}
		\calM_{\calV}(\C) \ = \ \coprod_{[\calL_0]} \ \left[ \left( o_{k}^{\times} \times U_{\calV}(\Q) \right) \ \Big\backslash \  \domain(\calV) \times U_{\calV}(\bbA_f) \ \Big/ \  K_{\calL} \right] , 
	\end{equation}
	where $\calL = \Lambda_0 \oplus \Lambda \subset \calV$ is our previously fixed lattice,  $K_{\calL} = \mathrm{Stab}(\calL \otimes_{\Z} \widehat\Z) \subset U_{\calV}(\bbA_f)$, and the factor $o_\kb^{\times}$ acts trivially.
	
	On the other hand, the complex points $\calY(\C)$ can be written as a disjoint union of $\#\{ [L_0]\}$ many copies
	\[ 
				\calY(\C)  \ = \ \coprod_{[L_0]} \left[ \left(  o_\kb^{\times} \times \Aut(\Lambda) \times U_{\Lambda_0}(\Q) \right)  \ \Big\backslash \ \{ \mathtt z_0 \} \times U_{\Lambda_0}(\bbA_f)   \ \Big/ \ K_{\Lambda_0}     \right]
	\]
	where $\mathtt z_0$ is the negative definite line
	\[ \mathtt z_0  \ = \ \Lambda_0 \otimes_{\Z} \R  \ = \ \Lambda^{\perp} \  \in \ \domain(\calV), \]
	and 
	\[ 
		K_{\Lambda_0} \ = \ \mathrm{Stab}(\Lambda_0 \otimes \widehat \Z) \ \subset \ U_{\Lambda_0}(\bbA_f),
	\]
	 and $o_\kb^{\times} \times \Aut(\Lambda)$ acts trivially. In these terms, the map $\calY(\C) \to \calM_{\calV}(\C)$ is given, on each component indexed by an $[L_0]$, by the map
	\[ [\mathtt z_0, h_0] \ \mapsto \ \left[ \mathtt z_0, (\sm{ h_0 & \\ & \mathsf{Id}_{\Lambda}}) \right] ;   \]
	all these facts may be inferred from the discussion in \cite[\S 5.3]{bruinier-howard-yang-unitary}. In particular, noting that the cardinality of the automorphism group of each point of $\calY(\C)$ is the same, the (stacky) degree is given by
	\begin{align*}
		\deg(\calY(\C)) \ =& \ \sum_{[\calL_0]}  \ \frac{1}{|o_\kb^{\times}| \cdot | \Aut(\Lambda)|} \cdot \frac{1}{\# U_{\Lambda_0}(\Q) \cap K_{\Lambda_0}}  \# \left( U_{\Lambda_0}(\Q) \backslash U_{\Lambda_0}(\bbA_f) / K_{\Lambda_0} \right) \\
		=& \frac{\# \{[\calL_0]\}}{|o_\kb^{\times}|^2 \cdot |\Aut(\Lambda)|} \cdot  \# \left( U_{\Lambda_0}(\Q) \backslash U_{\Lambda_0}(\bbA_f) / K_{\Lambda_0} \right).
	\end{align*}
	In fact,  this  quantity is equal to $\frac{2^{1- o(d_\kb)} h_k^2}{|o_\kb^{\times}|^2 |\Aut(\Lambda)|}$, but we will not need this fact.
	
	The values of $\KGr(m,v)$ may be re-expressed in terms of the uniformization \eqref{eqn:MNewCpxUnif} as follows. Extend the Siegel theta function, as in \eqref{eq:siegeltheta}, to an $ S(\calV(\bbA_f))^{\vee}$-valued function, by setting 
	\[
		\Theta_{\calV}(\tau', [z,h])(\varphi) \ :=  \  v' \, \sum_{x \in \calV} \varphi(h^{-1} x) \, e^{-2 \pi v' R(x, z)} \, q'^{(x,x)} 
	\]
	for a Schwartz function $\varphi \in S(\calV(\bbA_f))$ and a pair $(z,h) \in \domain(\calV) \times U_{\calV}(\bbA_f)$. If $\varphi$ is $K_{\calL}$-invariant, then this construction yields a well-defined function on (each component of) the right hand side of \eqref{eqn:MNewCpxUnif}. 
	
	Now suppose $\varphi \ \in \   \calS$,
	which we may view as a Schwartz function in $ S(\calV(\bbA_f))^{K_{\calL}} $. Then the truncated Poincar\'e series $P_{m,v, \varphi}$ takes values in $ S(\calV(\bbA_f))^{K_{\calL}} $ as well, and, by tracing through the definition of $\KGr(m,v)$ in \Cref{sec:arithChow} and applying \Cref{lem:xireg} and \Cref{thm:Xitheta} gives
	\begin{align*}
		\KGr(m,v)([z,h])(\varphi) \ &= \ \left\langle P_{m,v, \varphi} , \  \Theta_{\calV}(\cdot, [z,h]) \right\rangle^{\reg}  \\
		&= \ \lim_{T \to \infty} \int_{\calF_T} \, P_{m,v,\varphi}(\tau') \cdot \Theta_{\calV}(\tau', [z,h]) \, d \mu(\tau') \ - \ S_m([z,h])(\varphi) \cdot \log T 
	\end{align*}
	on each component in \eqref{eqn:MNewCpxUnif}; here 
	\[ 
		S_m([z,h])(\varphi) \ := \ \sum_{\substack{ x \in z^{\perp} \cap \calV \\(x,x)=m}} \varphi( h^{-1} x) .
	\]
	
	We now consider evaluating at a point in $\calY(\C)$.
	Supposing
	\[ \varphi = \varphi_0 \otimes \varphi_1 \in  S(\calV(\bbA_f)) \simeq S(\Lambda_0( \bbA_f)) \otimes S(\Lambda(\bbA_f)) , \]
	it follows immediately from definitions that upon evaluating at a point in $\calY(\C)$, the Siegel theta function decomposes as
	\[ 
		\Theta_{\calV}(\tau', [\mathtt z_0, (\sm{h_0&\\& 1})])(\varphi)  \ = \ \Theta_{\Lambda_0}(\tau')(\varphi_0\circ h_0^{-1}) \ \cdot \ \Theta_{\Lambda}(\tau')( \varphi_1  )
	\]
	where $\Theta_{\Lambda}(\tau')$ is the weight $n-1$ theta function attached to $\Lambda$, and 
		\[
			\Theta_{\Lambda_0}(\tau')(\varphi_0 \circ h_0^{-1}) \ = \ v'  \sum_{a \in \Lambda_{0,\Q} } \ \varphi_0(h_0^{-1}a) \,   e^{2 \pi v' Q(a) } \, e^{2 \pi i \tau' Q(a) } \ = \ v' \sum_a  \, \varphi_0( h_0^{-1}a) \, e^{2 \pi i Q(a) \overline{\tau'}}  
		\]
		is the (non-holomorphic) Siegel theta function of weight $-1$ attached to the signature $(0,1)$ Hermitian space $\Lambda_{0,\Q}$. 
	The \emph{Siegel-Weil formula} for $\Lambda_{0,\Q}$, cf.\ \cite[Proposition 6.2]{ichinoSWreg}, then implies 
	\[ 
		\int_{U_{\Lambda_0}(\Q) \backslash U_{\Lambda_0}(\bbA_f)} \, \Theta_{\Lambda_0}(\tau')( \varphi_0 \circ h_0^{-1}) \, d h_0 \ = \ \frac12 \  E_{-1, \Lambda_0}(\tau', 0) (\varphi_0)
	\]
	where $d h_0$ is the Haar measure normalized so that $\mathrm{vol}(U_{\Lambda_0}(\Q) \backslash U_{\Lambda_0}(\bbA_f))=1$. If we further assume that $\varphi_0$ is $K_{\Lambda_0}$-invariant, then summing over a set of representatives 
	\[ 
		h_{0,1}, \dots, h_{0,t} \ \in \ U_{\Lambda_0}(\Q) \backslash U_{\Lambda_0}(\bbA_f) / K_{\Lambda_0} 
	\]
	gives
	\begin{align*}
		\sum_{i} \Theta_{\Lambda_0}(\tau')(\varphi_0 \circ h_{0,i})  \ =& \ \frac{\#(U_{\Lambda_0}(\Q) \cap K_{\Lambda_0})}{\vol(K_{\Lambda_0})}  \cdot 	\int\limits_{U_{\Lambda_0}(\Q) \backslash U_{\Lambda_0}(\bbA_f)} \, \Theta_{\Lambda_0}(\tau')( \varphi_0 \circ h_0^{-1}) \, d h_0  \\
		=& \  \# \left( U_{\Lambda_0}(\Q) \backslash U_{\Lambda_0}(\bbA_f) / K_{\Lambda_0} \right)  \ \cdot \ \frac12 \  E_{-1, \Lambda_0}(\tau', 0) (\varphi_0).
	\end{align*}
	Putting everything together, and keeping track of automorphisms,
	\begin{align*}
	\Theta_{\calV}(\calY(\C)) \ =& \ \frac{\#{[\calL_0]}}{|o_\kb^{\times}| \cdot |\Aut(\Lambda)|} \, \frac1{\#(U_{\Lambda_0}(\Q) \cap K_{\Lambda_0})} 	
			\sum_{i} \Theta_{\calV}(\tau',[\mathtt z_0, (\sm{h_{0,i} & \\ & 1 })]) \\
		=& \ \frac{\deg(\calY(\C))}{2} \ \left( E_{-1,\Lambda_0}(\tau', 0) \otimes \Theta_{\Lambda}(\tau') \right)
	\end{align*}
	as linear functionals on $(S(\Lambda_0(\bbA_f)) \otimes S(\Lambda(\bbA_f))^{K_{\Lambda_0} \times K_{\Lambda}} $.
	
	Note that if $\varphi = \varphi_0 \otimes \varphi_1$, then 
	\[ 
		S_m([\mathtt z_0, (\sm{h_0 & \\ & 1})])(\varphi_0\otimes\varphi_1)  \ = \  \varphi_0(0) \cdot \sum_{\substack{x \in \Lambda_{\Q} \\ (x,x) = m}} \, \varphi_1(x)  \ = \ \varphi_0(0) \cdot c_{\Theta_{\Lambda}}(m)(\varphi_1).
	\]
	Therefore we obtain
        \begin{align*}
          	\frac{\KGr(m,v)(\calY(\C))}{ \deg(\calY(\C))} =  \lim_{T\to \infty} 
                \bigg( &\frac{1}{2} \int_{\calF_T}P_{m,v}(\tau')  \cdot \left(E_{-1,\Lambda_0}(\tau',0) \otimes \Theta_{\Lambda}(\tau')\right) d \mu(\tau') \\
                            &- \left(\varphi_0^\vee \otimes c_{\Theta_{\Lambda}}(m) \right)\log T \bigg).
        \end{align*}
	In order to evaluate this integral, note that the relation \eqref{eqn:LEisenstein} implies, in the same way as in \Cref{thm:DiffHeight}, that
	\[
		L\left( E'_{1, \Lambda_0}(\tau',0) \right) \ =  - \frac12 \,E_{-1, \Lambda_0}(\tau',0)
	\]
	and so, since $\Theta_{\Lambda}$ is holomorphic,
	\[ 
		L\left( E'_{1, \Lambda_0}(\tau',0) \otimes \Theta_{\Lambda}(\tau') \right) \ =  - \frac12 \,E_{-1, \Lambda_0}(\tau',0) \otimes \Theta_{\Lambda}(\tau').
	\]
	This in turn implies that
	\[ 
		d(  E'_{1, \Lambda_0}(\tau',0) \otimes \Theta_{\Lambda}(\tau') \, d\tau') \ = \ \frac12 \, 	E_{-1, \Lambda_0}(\tau',0) \otimes \Theta_{\Lambda}(\tau') \ d\mu(\tau'),
	\]
	and so, unfolding the Poincar\'e series as in the proof of \Cref{thm:MaassSection}, cf.\ \eqref{eqn:TPunfoldfinal}, we find that for $T \gg 0$,
	\[
		\int_{\calF_T} P_{m,v}(\tau') \left( E_{-1,\Lambda_0}(\tau',0) \otimes \Theta_{\Lambda}(\tau') \right) d \mu(\tau') \ = \ 2 \left( c_{E'\otimes\Theta_{\Lambda}}(m,T) -  c_{E'\otimes\Theta_{\Lambda}}(m,v) \right),
	\]
	where we abbreviated $E'_{1, \Lambda_0}(\tau',0) \otimes \Theta_{\Lambda} = E' \otimes \Theta_{\Lambda}$. Thus
	\[ 
		\frac{\KGr(m,v)(\calY(\C))}{\deg(\calY(\C))}\ = \ - c_{E' \otimes \Theta_{\Lambda}}(m,v) \ + \ \lim_{T \to \infty} \left( c_{E' \otimes \Theta_{\Lambda}} (m, T) - \varphi_0^\vee \otimes c_{\Theta_{\Lambda}}(m)  \cdot \log T \right).
	\]
	On the other hand, the essence of the main theorem of \cite{bruinier-howard-yang-unitary} is the computation of the finite intersection $[\Zed(m): \calY]_{\mathsf{fin}}$, which, after a straightforward translation, can be expressed as
	\[
			[ \Zed(m): \calY]_{\mathsf{fin}} \ = \ - \deg(\calY)(\C) \cdot 	 \lim_{T \to \infty} \left( c_{E' \otimes \Theta_{\Lambda}} (m, T) - \varphi_0^\vee \otimes c_{\Theta_{\Lambda}}(m)  \cdot \log T \right) .
	\]
	Thus,  
	\[
		[\ZedVKud(m,v): \calY] \  = \ \KGr(m,v)(\calY(\C)) \ + \ 	[\Zed(m): \calY]_{\mathsf{fin}} \ = \ - \deg(\calY(\C)) \, c_{E' \otimes \Theta_{\Lambda}}(m,v) 
	\]
	as required.
	
	Turning to the $m=0$ term, by definition
	\[
		[ \ZedVKud(0,v) : \calY ]  \ = \ - [ \widehat\omega : \calY]\, \varphi_0^{\vee} \ + \ \left( \KGr(0,v) - \log v \otimes \varphi_0^{\vee}  \right) (\calY(\C)).
	\]
	A similar argument as above yields
	\begin{align*}
			\frac{\left( \KGr(0,v) - \log v \otimes \varphi_0^{\vee}  \right) (\calY(\C))}{\deg(\calY(\C))} \ = \ - & c_{E' \otimes \Theta_{\Lambda}} (0,v) \\ 
			&+ \ \lim_{T \to \infty} \left( c_{E' \otimes \Theta_{\Lambda}} (0, T) - \varphi_0^\vee \otimes c_{\Theta_{\Lambda}}(0)  \cdot \log T \right).
	\end{align*}
	On the other hand, translating \cite[Theorem 6.4]{bruinier-howard-yang-unitary} into our notation gives 
	\[
		[ \widehat \omega : \calY] \, \varphi_0^{\vee} \ = \ \deg_{\C}(\calY(\C)) \cdot  \lim_{T \to \infty} \left( c_{E' \otimes \Theta_{\Lambda}} (0, T) - \varphi_0^\vee \otimes c_{\Theta_{\Lambda}}(0)  \cdot \log T \right);
	\]	
	this relation follows from the Chowla-Selberg formula and the choice of the metric on $\widehat \omega$. Thus
	\[
		[ \ZedVKud(0,v):\calY ] \ = \ c_{E'\otimes \Theta_{\Lambda}}(0,v)
	\]
	as required.
\end{proof}
\end{theorem}

\begin{remark}{\ }

	\emph{(i)} We explain how \Cref{thm:BHY} follows from \Cref{thm:refBHY}. Suppose that $n>2$, and recall that the difference $\ThetaVKud(\tau) - \ThetaVBru(\tau)$ is modular of weight $n$, valued in $\ChowHatC(\calM_{\calV}^*) \otimes_{\C} \calS^{\vee}$. As \Cref{thm:refBHY} shows that $[\ThetaVKud(\tau):\calY_{\calV, \Lambda}]$ is modular, it follows that
	\[ 
		[\ThetaVBru(\tau) : \calY_{\calV, \Lambda}] \ := \ \sum_{m \geq 0} [ \ZedVBru(m) : \calY_{\calV, \Lambda}] \, q^m 
	\]
	is a (holomorphic) modular form as well. Moreover, since $\ThetaVKud(\tau) - \ThetaVBru(\tau)$ has trivial cuspidal holomorphic projection, cf.\ \Cref{thm:arithHolProj}, 
	\[
	\left\langle [\ThetaVKud(\tau):\calY_{\calV, \Lambda}] ,  \ g \right\rangle_{\Pet}^{\reg} \ = \ 	\left\langle [\ThetaVBru(\tau):\calY_{\calV, \Lambda}] , \  g \right\rangle_{\Pet}^{\reg}
	\]
	for every cusp form $g \in S_n(\calS^{\vee})$. 
	
	By \Cref{thm:refBHY},  the left hand side is 
	\begin{align*}
	\left\langle [\ThetaVKud(\tau):\calY_{\calV, \Lambda}] ,  \ g \right\rangle_{\Pet}^{\reg} \ =& \ - \deg(\calY_{\calV, \Lambda}(\C)) \int_{\SL_2(\Z) \backslash \uhp} \left( E'_{1,\Lambda_0}(\tau,0) \otimes \Theta_{\Lambda}(\tau) \right)\cdot  \overline{g(\tau)}  \ v^{n} \  d\mu(\tau) .
	\end{align*}
	Applying the standard unfolding argument for the Eisenstein series $E_{1,\Lambda_0}(\tau, s)$ for $\Re(s) \gg 0$ gives
	\begin{align*}
	\int_{\SL_2(\Z) \backslash \uhp} \left( E_{1,\Lambda_0}(\tau,s) \otimes \Theta_{\Lambda}(\tau) \right)\cdot  \overline{g(\tau)}  \ v^{n} \  d\mu(\tau)  \ = \ L(s, g, \Theta_{\Lambda}),
	\end{align*}
	so  
	\begin{equation} \label{eqn:refBYKudla}
		\left\langle [\ThetaVKud(\tau):\calY_{\calV,\Lambda}] ,  \ g \right\rangle_{\Pet}^{\reg} \ = \ - \deg(\calY_{\calV,\Lambda}(\C)) \, L'(0, g, \Theta_{\Lambda}).
	\end{equation}
	On the other hand, if 
	\[ H(\tau) \  = \ \sum_{m \geq 0} \ c_H(m) q^m \  \in \ M_n(\mathcal S^{\vee}) \]
	is any holomorphic modular form and $g = \xi(f) \in S_n(\mathcal S^{\vee})$ for some $f \in H_{2-n}(\mathcal S)$, then \cite[Proposition 3.5]{brfu06} gives
\[
		\langle H, \, g \rangle^{\reg}_{\Pet} \ = \  \sum_{m \geq 0} c_H(m) \cdot c^+_f(-m) \]
	where $c^+_f(-m)$ are the coefficients of non-positive index of the holomorphic part of $f$, cf.\ \Cref{sec:modforms}. Applying this observation to the form $H(\tau) = [\ThetaVBru(\tau):\calY_{\calV, \Lambda}]$ gives
	\[
		\left\langle [\ThetaVBru(\tau):\calY_{\calV,\Lambda}] ,  \ g \right\rangle_{\Pet}^{\reg} \ = \ \sum_{m \geq 0} c_f^+(-m)  \cdot [\ZedVBru(m):\calY_{\calV,\Lambda}] .
	\]
	Equating this expression with \eqref{eqn:refBYKudla} then yields \Cref{thm:BHY}. 
	
	\emph{(ii)} The analogue of \Cref{thm:BHY} for orthogonal Shimura varieties of signature $(n,2)$ has been formulated and proven by Andreatta-Goren-Howard-Madapusi Pera \cite{AGHMP-orthogonal}. The corresponding version of \Cref{thm:refBHY} is also true, and can be proven in exactly the same way by employing the finite intersection calculation of \emph{op.~cit.}

	A version of this result can also be formulated, and proved in the same manner, for \emph{big CM cycles}, using the finite intersection calculations appearing in \cite{AGHMP-bigCM}.
\end{remark}

\printbibliography
\end{document}

%% file: preamble.tex
\usepackage[usenames, dvipsnames]{xcolor}
\usepackage{amsmath,amsthm,amssymb, bm}
\usepackage[a4paper,includeheadfoot,margin=2.54cm]{geometry}
\usepackage{fancyhdr}
\usepackage{amscd}
\usepackage{setspace}
\usepackage[english]{babel}
\usepackage[latin1]{inputenc}
\usepackage{tikz-cd}
\usepackage{enumitem}
\usepackage{float}
\usepackage{url}
\usepackage[all]{xy}
\usepackage{comment}

\usepackage[pdfborder={0 0 0}, colorlinks = true, linkcolor=blue, citecolor=OliveGreen, pdftex]{hyperref}
\usepackage{cleveref}

\newcommand{\R}{\mathbb{R}} 
\newcommand{\Q}{\mathbb{Q}} 
\newcommand{\uhp}{\mathbb{H}} 
\newcommand{\C}{\mathbb{C}} 
\newcommand{\Z}{\mathbb{Z}} 

\newcommand{\Zhat}{\widehat{\mathbb{Z}}}
\newcommand{\Lhat}{\widehat{L}}
\newcommand{\adeles}{\mathbb{A}}

\newcommand*{\domain}{\mathbb{D}}

\newcommand*{\reg}{\mathrm{reg}}
\newcommand*{\Pet}{\mathrm{Pet}}

\newcommand{\Gt}{\widetilde{\Gamma}}
\newcommand{\G}{\Gamma}

\newcommand{\bs}{\backslash}

\newcommand{\bbA}{\mathbb A}

\newcommand{\calA}{\mathcal{A}}

\newcommand{\calC}{\mathcal{C}}
\newcommand{\calD}{\mathcal{D}}
\newcommand{\calE}{\mathcal{E}}
\newcommand{\calF}{\mathcal{F}}
\newcommand{\calG}{\mathcal{G}}

\newcommand{\calI}{\mathcal{I}}

\newcommand{\calL}{\mathcal{L}}
\newcommand{\calM}{\mathcal{M}}

\newcommand{\calO}{\mathcal{O}}

\newcommand{\calR}{\mathcal{R}}
\newcommand{\calS}{\mathcal{S}}

\newcommand{\calV}{\mathcal{V}}

\newcommand{\calX}{\mathcal{X}}
\newcommand{\calY}{\mathcal{Y}}
\newcommand{\calZ}{\mathcal{Z}}

\newcommand{\diffinv}[1]{\partial_#1^{-1}} 

\newcommand{\fraka}{\mathfrak{a}}

\newcommand{\frakg}{\mathfrak{g}}

\newcommand{\abcd}{\begin{pmatrix}a & b \\ c & d\end{pmatrix}}

\newcommand{\abs}[1]{\left\vert#1\right\vert}

\DeclareMathOperator{\End}{End}
\DeclareMathOperator{\Lie}{Lie}
\DeclareMathOperator{\Aut}{Aut}

\DeclareMathOperator{\SL}{SL}

\DeclareMathOperator{\Mp}{Mp}

\DeclareMathOperator{\divisor}{div}

\DeclareMathOperator{\vol}{vol}

\DeclareMathOperator{\Hom}{Hom}
\DeclareMathOperator*{\CT}{CT}

\DeclareMathOperator{\ddc}{dd^c}

\renewcommand{\Re}{\mathrm{Re}}
\renewcommand{\Im}{\mathrm{Im}}

\newtheorem{theorem}{Theorem}[section]
\newtheorem{proposition}[theorem]{Proposition}
\newtheorem{lemma}[theorem]{Lemma}
\newtheorem{corollary}[theorem]{Corollary}
\newtheorem{conjecture}[theorem]{Conjecture}

\theoremstyle{definition}
\newtheorem{definition}[theorem]{Definition}
\newtheorem{hypothesis}[theorem]{Hypothesis}
\usepackage{thmtools}
\declaretheorem[style=definition,qed=$\diamond$, sibling=theorem]{remark}

\DeclareMathOperator{\Spec}{Spec}

\DeclareMathOperator{\Div}{Div}

\newcommand{\divhat}{\widehat{\divisor}}
\newcommand{\deghat}{\widehat{\deg}}

\newcommand{\proj}{\mathsf{pr}} 

\newcommand{\Zed}{\mathcal{Z}}

\newcommand{\M}{\mathcal{M}}
\newcommand{\ZedVKud}{\widehat\Zed_{\calV}^{\, \mathsf{K}}}
\newcommand{\ZedVBru}{\widehat\Zed_{\calV}^{\, \mathsf{B}}}
\newcommand{\ThetaVKud}{\widehat\Theta_{\calV}^{\mathsf{K}}}
\newcommand{\ThetaVBru}{\widehat\Theta_{\calV}^{\mathsf{B}}}

\DeclareMathOperator{\D}{\mathcal{D}}

\DeclareMathOperator\ChowHatC{\widehat{\mathsf{CH}}{}^1_{\C}}
\newcommand\Kra{\text{Kr\"a}}

\newcommand{\otauthat}{ \widehat{\omega}^{\mathsf{taut}}}

\newcommand{\KGr}{\mathsf{Gr}^{\mathsf{K}}}
\newcommand{\KGrO}{\KGr_o}
\newcommand{\BGr}[1]{\mathsf{Gr}^{\mathsf{B}}({#1})}
\newcommand{\BGrO}[1]{\mathsf{Gr}_o^{\mathsf{B}}({#1})}
\newcommand{\kb}{{\bm{k}}}

\newcommand{\la}{\langle}
\newcommand{\ra}{\rangle}
\newcommand{\lie}[1]{\mathfrak{#1}}
\newcommand{\isomto}{\overset{\sim}{\longrightarrow}}

\newcommand{\ov}{\overline}

\newcommand{\sm}[1]{ \begin{smallmatrix} #1 \end{smallmatrix} }

\usepackage{mathrsfs}

\newcommand{\ALmod}[1]{A_{#1}^{\mathtt{mod}}}
\newcommand{\AkappaLexp}{A_{\kappa}^{!}}
\newcommand{\ALexp}[1]{A_{#1}^!}

\newcommand{\Low}{\mathbf{L}}
\newcommand{\Lsharp}{\Low^{\sharp}}

\usepackage[style=alphabetic, backend=bibtex, doi=false, url=true, firstinits=true, eprint=false, sorting=nyt, maxnames=6]{biblatex}
